\RequirePackage{fix-cm}
\documentclass[smallcondensed]{svjour3}       
\smartqed  

\usepackage{verbatim}
\usepackage{graphicx}
\usepackage{amscd,amssymb,epsfig,mathtools}
\usepackage[usenames,dvipsnames]{color}
\newcommand\Vor{\operatorname{Vor}}
\newcommand\inte{\operatorname{int}}

\newcommand\Conv{\operatorname{Conv}}

\newcommand{\epi}{\operatorname{epi}} 

\newcommand\msquare{\scalebox{0.6}{$\square$}}

\newcommand{\diam}{\operatorname{diam}}

\newcommand{\Z}{\mathbb{Z}}

\newcommand{\R}{\mathbb{R}}

\newcommand{\tir}[1]{\ensuremath{\overline {#1}}}

\def\whsq{\vbox to 5.8pt 
{\offinterlineskip\hrule 
\hbox to 5.8pt{\vrule height 
5.1pt\hss\vrule height 5.1pt}\hrule}} 
 
\def\Qed{{\hfill {\whsq}}} 
\def\<{\langle} 
\def\>{\rangle} 
\def\PP{{\mathop{{\rm I}\kern-.2em{\rm P}}\nolimits}} 
\def\FF{{\mathop{{\rm I}\kern-.2em{\rm F}}\nolimits}}   
\def\ZZ{{\mathop{{\rm I}\kern-.2em{\rm Z}}\nolimits}}

\begin{document}

\title{The second boundary value problem for a discrete Monge-Amp\`ere equation
}

\titlerunning{ The second boundary value problem for a discrete Monge-Amp\`ere equation}        

\author{Gerard Awanou}


\institute{  Gerard Awanou \at
              Department of Mathematics, Statistics, and Computer Science (M/C 249), University of Illinois at Chicago,
Chicago, IL, 60607-7045\\              
              Tel.: +1-312-413-2167\\
              Fax: +1-312-996-1491 \\
              \email{awanou@uic.edu} \\
              ORCID https://orcid.org/0000-0001-9408-3078                  
              }

\date{Received: date / Accepted: date}

\maketitle

\begin{abstract}

In this work we propose a discretization of the second boundary condition for the Monge-Amp\`ere equation arising in geometric optics and optimal transport. The discretization we propose is the natural generalization of the popular Oliker-Prussner method proposed in 1988. For the discretization of the differential operator, we use a discrete analogue of the subdifferential. Existence, unicity and stability of the solutions to the discrete problem are established. Convergence results to the continuous problem are given.

\end{abstract}

\section{Introduction}
In this paper we propose a 
discretization of the second boundary condition for the Monge-Amp\`ere equation. 
Let $\Omega$ and $\Omega^*$ be bounded convex domains of $\R^d$. Let $f$ 
be a non negative integrable function on $\Omega$ and $R>0$ an 
integrable function on $\Omega^*$. We are interested in discrete approximations of convex weak solutions in the sense of Aleksandrov of the model problem
\begin{align} \label{m1}
\begin{split}
R(D u(x) ) \det D^2 u(x) &= f(x) \text{ in } \tir{\Omega}\\
\chi_u (\tir{\Omega}) & = \tir{\Omega^*},
\end{split}
\end{align}
where 
the unknown is a convex function $u$ on $\Omega$ such that $\partial u(\Omega)=\Omega^*$, 
$D u$ denotes the gradient of $u$ and $D^2 u$ its Hessian. We use the notation $\partial u$ for the local subdifferential of $u$ and $\chi_u$ denotes the subdifferential of a specific convex extension $\tir{u}$ to $\R^d$ of $u$, c.f. section \ref{subdifferential}. That convex extension satisfies $\chi_u (\R^d) = \chi_u (\tir{\Omega})  = \tir{\Omega^*}$.
The epigraph of $\tir{u}$, c.f. section \ref{epigraph}, is an unbounded convex set for which there is a notion of asymptotic cone, c.f. section \ref{asym}. The asymptotic cone essentially gives the behavior at infinity of the convex extension $\tir{u}$. From $\Omega^*$, we construct a convex set $K_{\Omega^*}$ which turns out to be the asymptotic cone of the epigraph of the extension $\tir{u}$.  The equation $\chi_{u} (\tir{\Omega})  = \tir{\Omega^*}$ is then equivalent to prescribing the asymptotic cone of  the epigraph of  a certain convex extension to $\R^d$ of the convex function $u$ on $\Omega$. We derive an explicit expression of the extension in terms of the asymptotic cone, which we use to derive the numerical scheme.


We approximate $\tir{\Omega^*}$ by closed convex  polygons $Y \subset \tir{\Omega^*}$ and give an explicit formula for the extension of a mesh function $u_h$ on $\Omega$ which guarantees that the latter has an asymptotic cone $K$ associated with $Y$ with $\chi_{u_h} (\tir{\Omega}) \subset Y$, where $\chi_{u_h}$ denotes some discrete version of the subdifferential.
One then only needs to apply the discrete Monge-Amp\`ere operator in this class of mesh functions, c.f. \eqref{m2d3} below. It was thought \cite[p. 24]{Oliker03} that ''dealing with an asymptotic cone as the boundary condition is inconvenient''.

The left hand side of \eqref{m1} is to be interpreted as the density of a measure $\omega(R,u,.)$ associated to the convex function $u$ and the mapping $R$ c.f. section \ref{R-curvature}. It is defined through the subdifferential of $u$. 
Equations of the type \eqref{m1} appear for example in optimal transport and geometric optics. 
The compatibility condition $\int_{\Omega} f(x) d x = \int_{\Omega^*} R(p) d p$ is required, c.f. section \ref{R-curvature}.

\subsection{Short description of the scheme}

In this paper we consider Cartesian grids and a discrete analogue of the subdifferential considered in  \cite{DiscreteAlex2,Mirebeau15} for the Dirichlet problem. 
Let $h$ be a small parameter, $a+h \mathbb{Z}^d$ for $a \in \R^d$ be the set of mesh points. 
The description of the scheme is given in section \ref{disc-R-curvature}. 
We assume for now that $\Omega=(0,1)^d, a=(1/2, \ldots,1/2)$, $f>0$, $f \in C(\Omega)$ and $\Omega^*$ is a convex polygonal domain 
with vertices $a_j^*, j=1,\ldots,N^*$. Denote by $\Omega_h$ the set of mesh points in $\Omega$ and by $\partial \Omega_h$ the set of mesh points in $\Omega$ closest to $\partial \Omega$ in directions of the canonical basis of $\R^d$.  
The unknown in the discrete scheme is a function defined on $\Omega_h$ which we refer to as a mesh function. 
Given a stencil $V$, i.e. the choice $V(x)$ of a subset of $\mathbb{Z}^d \setminus \{ \, 0 \, \}$ for $x \in \Omega_h$, and an associated  discrete analogue $\partial_V u_h$ of the subdifferential, we define the discrete Monge-Amp\`ere operator by $\omega_V(R,u_h, x)=\int_{\partial_V (u_h) (x)} R(p) dp$  for a mesh point $x \in \Omega_h$. The discretization we analyze consists in solving the nonlinear problem
$$
\omega_V(R,u_h, x) = h^d f(x), x \in \Omega_h,
$$
with unknown mesh values $u_h(x), x \in \Omega_h$. The evaluation of $\omega_V(R,u_h, x)$ requires mesh values $u_h(x), x \notin \Omega_h$. They are given by the discrete extension formula
$$
u_h(x) = \min_{ y \in \partial \Omega_h }  \max_{1 \leq j \leq N} (x-y) \cdot a_j^* +u_h(y),
$$
motivated by Theorem \ref{formula2} below.
The above formula implicitly enforces the second boundary condition as we discuss below. For this example, in the case $f=1$, the simple choice of the right hand side $h^d f(x)$ assures the discrete compatibility condition \eqref{mass-conservation} below. See \eqref{m2d3} below for a suitable right hand side and section \ref{Numerical} for other modifications. 

\subsection{Relation with semi-discrete optimal transport (SDOT)}
A quantization of $f$ is a partition of the domain $\Omega$ into closed cells $E_i, i=1,\ldots, N$ with diameter $\diam(E_i)$ and non empty interiors such that $E_i \cap E_j$ has Lebesgue measure 0 for $i \neq j$, $\cup_{i=1}^N E_i = \tir{\Omega}$. For $x_i$ in the interior of $E_i$ 
and $h=\max \{ \, \diam(E_i), i=1,\ldots, N \, \}$, 
$\mu_h = \sum_{i=1}^N \bigg( \int_{E_i} f(x) dx \bigg) \delta_{x_i}$ weakly converges to the measure with density $f$. Weak convergence of measures is discussed in section \ref{cvg}. In SDOT \cite{Yau2013,levy2018notions,kitagawa2016newton,Aurenhammer98,merigot2011multiscale,Li-Nochetto2021}, one seeks a 
mesh function $u_h$ such that
\begin{equation} \label{sdot}
\int_{\partial_d u_h(x_i) \cap \Omega^*} R(p) dp = \int_{E_i} f(x) dx, i=1, \ldots N,
\end{equation}
 where the discrete subdifferential is defined by
\begin{equation} \label{disc-subd-f}
 \partial_d u_h (x_i) = \{ \, p \in \R^d, u_h(x_j) \geq u_h(x_i) + p \cdot (x_j-x_i), \text{ for all } j=1,\dots,N \, \}.
 \end{equation}
 The computation of $ \partial_d u_h(x_i), i=1,\ldots, N$ is obtained through the construction of a power diagram \cite[Section 5.1]{berman2018convergence}. One then takes the intersection of the diagram with $\Omega^*$. 
 The cells $ \partial_d u_h (x_i), i=1,\dots,N$ are usually interpreted in terms of the Legendre transform of $u_h$, are known as Laguerre cells and form a partition of $\R^d$. 
 Since $\int_{\Omega} f(x) dx = \sum_{i=1}^N  \int_{E_i} f(x) dx = \sum_{i=1}^N \int_{\partial_d u_h(x_i) \cap \Omega^*} R(p) dp = \int_{\Omega^* \cap Y} R(p) dp$ where $Y=\cup_{i=1}^N  \partial_d u_h (x_i)$, we see that in general, the discrete subdifferential is not the usual subdifferential of a piecewise linear convex function. If this were the case, and the second boundary condition $\partial_d u_h (\R^d)
 \subset \Omega^*$ holds, then $Y \subset \Omega^*$. By the compatibility condition we obtain $|Y \cap \Omega^*| = |\Omega^*|$. Since
 $Y \cap \Omega^*  = Y \subset \tir{ \Omega^*}$, we obtain $| \tir{ \Omega^*} \setminus Y| = 0$. This implies that $Y$, which is closed,  
 is dense in $\tir{ \Omega^*}$
 and hence $Y=\tir{ \Omega^*}$. By Lemma \ref{subd-piecewise-c} below, $Y$ would be polygonal and  recall that $\Omega^*$ is not necessarily polygonal. Contradiction.  
 We note that for $x_k$ in the interior of the convex hull of $x_i, i=1,\dots,N$, the discrete subdifferential is equal to the usual subdifferential, c.f. \cite[Lemma 2.1]{Nochetto19}.

The method we have proposed can be seen as a variant where the condition $\partial u_h(\tir{\Omega}) \subset Y
 \subset \tir{\Omega^*}$ is enforced explicitly through a convex extension. Here, $Y$ is a polygonal approximation of $\Omega^*$ and we also denote by $u_h$ the piecewise linear convex function with vertices at the mesh points $x_i, i=1,\ldots N$, c.f section \ref{subdifferential} for a definition. 
Let $x_i, i=N+1,\ldots, M$ be points in $\R^d$ such that $\tir{\Omega}$ is contained in the convex hull of $\{ \, x_i, i=1,\ldots,M \, \}$. It is required that for a normal $n$ to a facet of $Y$ and $i=1,\ldots, N$, there is a node $x_j, j=1,\ldots,M$ such that $x_j-x_i$ is parallel to $n$. This ensures that $\partial u_h(x_i) \subset Y$, c.f. Lemma \ref{inc-sub-lem} for Cartesian meshes.  
The parameter $h$ and $\partial u_h(x_i)$
are defined analogously as in SDOT. However, in this context, the discrete subdifferential is the same as the usual subdifferential, hence the notation, c.f. for example \cite[Lemma 4]{awanou2019uweakcvg}. 
We now require that $\int_{\partial u_h(x_i) } R(p) dp = \int_{E_i} f(x) dx, i=1, \ldots N$ with $u_h(x_i)$ for $i=N+1,\ldots,M$ obtained through the discrete extension formula. Here $\partial \Omega_h$ consists in the mesh points $x_i$ on the boundary of the convex hull of $\{ \, x_i, i=1,\ldots,N \, \}$, c.f. Theorem \ref{formula2}. 
The stencil $V$ is now chosen 
in such a way that $\partial u_h(x_i) = \partial_V u_h(x_i)$ for $i=1,\ldots,N$. Note that with the assumption $f>0$ on $\Omega$, for $x \in \partial \Omega_h$, $|\partial u_h(x) | \neq 0$ since $\partial u_h(x) \subset Y \subset \Omega^*$ and $R>0$ on $\Omega^*$.

We view the method proposed as the natural generalization of the Oliker-Prussner method \cite{Oliker1988} in the sense that 
it uses the notion of asymptotic cone and the usual subdifferential as in the original studies of the second boundary value problem \cite{Bakelman1994}. Compared with the Dirichlet problem, where boundary values are given at the additional nodes, here these values are obtained from the discrete extension formula. Convergence rates for the method proposed were given in \cite{berman2018convergence}. 
See also \cite{merigot2020quantitative}. We shall give a detailed argument of the convergence without convergence rates.

\subsection{Some advantages of the proposed approach}

The main ingredient in the implementation of SDOT is the computation of the convex envelope of a finite set of points. This is a classical and  hard problem in itself and is well studied in computational geometry, so that widely available software libraries can be used. If $N$ is the number of Dirac masses used in SDOT, a convex hull of $N$ points in $\R^{d+1}$ is constructed. Using for example
 the quickhull algorithm, this results in a computational complexity O($N \log N$) for $d=2$ and O($n^{\lfloor (d+1)/2 \rfloor}$) for $d \geq 3$, i.e. at least a computational complexity O($N^2$) for $d \geq 3$. As pointed out in \cite[Remark 5.5 ]{berman2018convergence}, the damped Newton's method used in \cite{kitagawa2016newton} requires to find the volume of the intersection of the cells in the power diagram. This has a worst-case complexity O($N^2$). 
 In summary, the use of a damped Newton's method in SDOT results in a worst-case complexity O($N^2$) for $d=2$ and $d=3$. 
 
 On Cartesian meshes, 
the complexity of the proposed approach for setting up the nonlinear discrete equations is dimension independent and given by O($N \#V)$, where $\#V$ denotes 
the  maximum of $\{ \, \# V(x), x \in \Omega_h \, \}$ and $N$ denotes the number of mesh points. For the stencil $V_{max}$ discussed below, $\#V =$ O$(N)$ and in that case the complexity is O($N^2$). However, the proposed approach allows to choose a stencil $V$ for which  $\#V$ is a constant independent of $N$, resulting in a linear complexity O$(N)$. A damped Newton's method is also used for solving the nonlinear equations. This requires to compute at each mesh point the volume of the facets of the discrete subdifferential, resulting again in a complexity O($N \#V)$. In summary, the proposed approach allows to choose a stencil the size of which has an upper bound independent of $N$, leading to a method with linear complexity. In the latter case, convergence of the discretization then holds for $f \in C(\Omega)$. 

For $f \in L^1(\Omega), f>0$,  with a stencil $V_{max}$ 
chosen such that $\partial_{V_{max}} u_h(x)=\partial \Gamma_2(u_h)(x)$ for all $x \in \Omega_h$ and a certain convex envelope $\Gamma_2(u_h)$ of $ u_h$,  
our convergence results can be seen as a version of arguments given in \cite[Proposition 2.3]{berman2018convergence} as $u_h$ is then equal to its convex envelope on $\Omega_h$. 

Existence and uniqueness of a solution are proved. 


\subsection{Relation with other work}

While there have been previous numerical simulations of the second boundary value problem \eqref{m1}, c.f. \cite{FroeseSJSC12,Benamou2014,Prins2015,LindseyRubinstein,kawecki2018finite}, advances on theoretical guarantees are very recent \cite{LindseyRubinstein,benamou2017minimal,Froese2019,Bonnet-Mirebeau,Brusca-Hamfeldt}. The approach in \cite{LindseyRubinstein,Froese2019} is to enforce the constraint $\chi_u (\tir{\Omega})  = \tir{\Omega^*}$ at the discrete level at all mesh points of the computational domain. Open questions include uniqueness of solutions to the discrete problem obtained in \cite{Froese2019}, existence of a solution to the discrete problem analyzed in \cite{benamou2017minimal}
and existence of a solution to the discrete problem obtained in \cite{LindseyRubinstein} for a target density $R$ only assumed to be locally integrable. 

Our work is closer to the one by Benamou and Duval \cite{benamou2017minimal} who proposed a convergence analysis based on the notion of minimal Brenier solution. Yet the two methods are fundamentally different. For example, the method in \cite{benamou2017minimal} is reported to have first order convergence for the gradient. For our method, taking forward and backward differences result in a $O(1)$ convergence rate for the gradient, i.e. the numerical errors for the gradient are merely bounded. The first order convergence rate is nevertheless achieved by selecting an element of the discrete subdifferential. 
Our analysis relies exclusively on the notions of Aleksandrov and viscosity solutions with guarantees on existence and uniqueness of a solution to the discrete problem. The uniqueness of a solution of the discrete problem is important for the use of globally convergent Newton's methods. 
Unlike the approaches in  \cite{LindseyRubinstein,benamou2017minimal,Froese2019}, we do not use a discretization of the gradient in the first equation of \eqref{m1}. See also \cite{qiu2020note} for the Dirichlet problem. Convergence of the discretization does not assume any regularity on solutions of \eqref{m1} and is proven for mesh functions, and their convex envelopes. Convergence of mesh functions implies the convergence of their convex envelopes \cite[Lemma 10]{awanou2019uweakcvg}. 
Another difference of this work with \cite{benamou2017minimal} is that we do not view the second boundary condition as an equation to be discretized. Analogous to methods based on power diagrams  \cite{Yau2013,levy2018notions}, the unknown is sought as a function over only the domain $\Omega$ with the second boundary condition enforced implicitly. 

For the approach in \cite{Yau2013,levy2018notions,kitagawa2016newton,Aurenhammer98,merigot2011multiscale},  
for efficiency and a convergence guarantee of an iterative method for solving the discrete equations, the use of power diagrams with a damped Newton's method is advocated \cite{kitagawa2016newton}.  
However, that approach results in a worst-case complexity O($N^2$) for $d=2$ and $d=3$. 
To avoid the complication of constructing power diagrams in three dimensions for the Dirichlet problem, Mirebeau in \cite{Mirebeau15} proposed a scheme which is medius between finite differences and power diagrams.
The discretization of \eqref{m1} analyzed in this paper 
is also medius between finite differences and power diagrams. 
Dealing with the second boundary condition requires to take into account the domain $\Omega^*$, and hence 
our discretization depends on $\Omega^*$. As with \cite{Mirebeau15} the implementation of our scheme does not require any of the subtleties required to deal with power diagrams in three dimensions. 
The proof of convergence of a damped Newton's method for solving the nonlinear equations resulting from the discretization, has been given in \cite{Awanou-damped}. As with the approaches in \cite{Yau2013,levy2018notions,kitagawa2016newton,Aurenhammer98,merigot2011multiscale,qiu2020note}, numerical integration may be required. 
%


\subsection{Organization of the paper}
 
We organize the paper as follows: In the next section we introduce some notation and the weak formulation of \eqref{m1}. We then describe the numerical scheme and recall some results on the convex envelopes of mesh functions. Existence, uniqueness and stability of solutions are given in section \ref{an-scheme}. In section \ref{asym} we review the notion of asymptotic cone of convex sets. This leads to the extension formula which has motivated the numerical scheme. We then recall the interpretation of \eqref{m1} as \cite{Oliker03} '' the second boundary value problem for Monge-Amp\`ere equations arising in the geometry of convex hypersurfaces \cite{Bakelman1994} and mappings with a convex potential \cite{Caffarelli92}.'' With the notion of asymptotic cone we prove further results about convex extensions. 
Section \ref{comp-section} is a review of polyhedral set theory and uses a matrix formalism to revisit most of the results we prove in section \ref{asym} directly from the geometric definition of asymptotic cone. Section \ref{comp-section} may be viewed as an appendix.
In section \ref{wcg} we present results about weak convergence of Monge-Amp\`ere measures for discrete convex  mesh functions. 
In section 
\ref{cvg} 
we give several convergence results for the approximations. 
The results in sections \ref{an-scheme} and \ref{cvg} assume that $f>0$ in $\Omega$. 
In section \ref{degenrate-case}, we consider the degenerate case $f\geq 0$. 
Numerical experiments are reported in section \ref{Numerical}. We give some additional remarks in the appendix. Therein we 
revisit 
convex extensions in terms of 
infimal convolution. 

\section{The discrete scheme} \label{thescheme}

In this section, we introduce some notation and recall the interpretation of \eqref{m1} as the second boundary value problem for Monge-Amp\`ere equations arising in the geometry of convex hypersurfaces. We then recall discrete versions of the notion of subdifferential and describe the numerical scheme. We now assume that $R=0$ on $\R^d \setminus \Omega^*$. Recall that $R>0$ on $\Omega^*$ and $R \in L^1(\Omega^*)$.

\subsection{R-curvature of convex functions} \label{R-curvature}

Let $v$ be a convex function on $\R^d$. For $y \in \R^d$, the normal image of the point $y$ (with respect to $v$) or the subdifferential of $v$ at $y$ is defined as
\begin{align*}
\chi_v (y) = \{ \, q \in \R^d: v(x) \geq v(y) + q \cdot (x-y), \, \text{for all} \, x \in\R^d\,\}.
\end{align*}
For $y \in \Omega$, the local normal image of the point $y$ (with respect to $v$) or the local subdifferential of $v$ at $y$ is defined as
\begin{align*}
\partial v (y) = \{ \, q \in \R^d: v(x) \geq v(y) + q \cdot (x-y), \, \text{for all} \, x \in \Omega\,\}.
\end{align*}
Since we have assumed that $\Omega$ is convex and $v$ is convex, the local normal image and the normal image coincide for $y \in \Omega$ \cite[Exercise 1]{Guti'errezExo}. We recall that a domain is a non empty open and connected set. In particular, $\Omega^*$ is non empty. 

For $q, y \in \R^d$ and $\mu \in \R$, the set of points $\{ \, (x,z) \in \R^{d+1}, x \in \R^d, z \in \R, z= \mu + q \cdot (x-y) \, \}$ is called a hyperplane. When $q \in \chi_v (y)$, $v(y) + q \cdot (x-y)$ is called a supporting hyperplane. It is known that when $v$ is differentiable at $y$, $\chi_v(y) = \{ \, D v(y)\, \}$. For the function $v$ given by $v(x)=|x|, x \in \R$, we have $\chi_v(0)=[-1,1]=\chi_v(\R)$. 

For any subset $E \subset \R^d$, the normal image of $E$ (with respect to $v$) is defined as 
$$\chi_v(E) = \cup_{x \in E}\chi_v(x).$$ 
The set $\partial v(E)$ is defined analogously.

The presentation of the R-curvature of convex functions given here is essentially taken from \cite{Bakelman1994} to which we refer for further details. 
It can be shown that $\chi_v(E)$ is Lebesgue measurable when $E$ is also Lebesgue measurable. The R-curvature of the convex function $v$ is defined as the set function
$$
\omega(R,v,E) = \int_{\chi_v(E)} R(p) d p,
$$
which can be shown to be a measure on the set of Borel subsets of $\R^d$.  For an integrable function $f \geq 0$ on $\Omega$ and extended by 0 to $\R^d$, equation \eqref{m1} is the equation in measures
\begin{align} \label{m1m}
\begin{split}
\omega(R,u,E) &= \int_E f(x) d x \text{ for all Borel sets } E \subset \tir{\Omega} \\ 
\chi_u (\tir{\Omega})  &= \tir{\Omega^*}.
\end{split}
\end{align}
This implies the compatibility condition 
\begin{equation} \label{necessary}
\int_{\Omega} f(x) d x = \int_{\Omega^*} R(p) d p.
\end{equation} 
In \eqref{m1m}, the unknown is a convex function $u$ defined on $\tir{\Omega}$ with a convex extension, c.f. section \ref{subdifferential}, that satisfies
$\chi_u (\R^d) = \chi_u (\tir{\Omega})  = \tir{\Omega^*}$.

\subsection{Discretizations of the R-curvature  } \label{disc-R-curvature}

We consider a non degenerate polygonal domain $Y \subset \tir{\Omega^*}$  with boundary vertices $a^*_j, j=1,\ldots,N^*$. 
We first solve an approximate problem where the solution satisfies $\chi_u (\tir{\Omega})=Y$. In view of the compatibility condition \eqref{necessary}, we consider a modified right hand side 
\begin{equation} \label{modified-f}
\tilde{f}(x) = (1 - \epsilon) f(x), \quad \epsilon = \frac{\int_{\Omega^*\setminus Y} R(p) d p}{\int_{\Omega} f(x) d x}.
\end{equation}
The truncation $\tilde{f}$ depends on $Y$ and that dependence will be made explicit in section \ref{cvg} where we use the notation $\tilde{f}_{Y}$.

Note that since $R>0$ on $\Omega^*$, by \eqref{necessary} $\int_{\Omega} f(x) d x >0$. Furthermore
$$
\int_{\Omega^*\setminus Y} R(p) d p = \int_{\Omega^*} R(p) d p - \int_{Y} R(p) d p
= \int_{\Omega} f(x) d x - \int_{Y} R(p) d p <  \int_{\Omega} f(x) d x,
$$
so that $0<\epsilon <1$. Moreover, in view of \eqref{necessary}, we obtain
$$
  \int_{\Omega} f(x) d x - \int_{\Omega^*\setminus Y} R(p) d p
 =  \int_{\Omega^*} R(p) d p -  \int_{\Omega^*\setminus Y} R(p) d p=\int_{Y} R(p) d p. 
$$
Therefore
\begin{equation} \label{necessary2}
 \int_{\Omega} \tilde{f}(x) d x = \int_{Y} R(p) d p.
\end{equation}
We therefore consider, using a slight abuse of notation for $u$, the problem: find $u$ convex on $\R^d$ such that
 \begin{align} \label{m2}
\begin{split}
\omega(R,u,E) &= \int_E \tilde{f}(x) d x \text{ for all Borel sets } E \subset \tir{\Omega}\\
\chi_u (\tir{\Omega}) & = Y.
\end{split}
\end{align}

Let $h$ be a small positive parameter and let
$
\mathbb{Z}^d_h = a+ \{\, m h, m \in \mathbb{Z}^d \, \}
$
denote the orthogonal lattice with mesh length $h$, with an offset $a \in \R^d$. 
Put $\Omega_h=\Omega \cap  \mathbb{Z}^d_h$ and denote by $\{ \, r_1,\ldots,r_d \, \}$ the canonical basis of $\R^d$. 
If $\Omega=(0,1)^d$ and we take $a=(1/2,\cdots,1/2)$, then $\tir{\Omega}= \cup_{x \in \Omega_h} x+ [-h/2,h/2]^d$. This partition of 
$\tir{\Omega}$ implies the mass conservation condition \eqref{mass-conservation} below. 
\begin{definition} 
A stencil $V$ is a set valued mapping from $\Omega_h$ to the set of finite subsets of $\mathbb{Z}^d\setminus \{0\}$. 
\end{definition} 

We will make the abuse of notation of writing $e \in V$ for $e \in V(x)$ when considering the points $x \pm h e$.

A subset $W$ of $\mathbb{Z}^d$ is symmetric with respect to the origin if $\forall y \in W, -y \in W$. Recall that a facet of a polygon $Y \subset \R^d$ is a $(d-1)$-dimensional face of $Y$, c.f. section \ref{comp-section} for the definition of faces.

We define $V_{min}$ to be a finite subset of $\mathbb{Z}^d\setminus \{0\}$ which is symmetric with respect to the origin, contains the elements of the canonical basis of $\R^d$, and contains a vector parallel to a normal to each facet of the domain $Y$. 


The assumption that $V_{min}$ contains a normal to each facet of the domain $Y$ may seem restrictive. 
However the approximate polygonal domain $Y$ to $\Omega^*$ can be chosen such that normals to its facets are parallel to vectors in $\mathbb{Z}^d$. 

Next, we consider the domain 
$$
\Omega_{ext} = \Omega_h \cup \{ \, x+ h e: x \in \Omega_h, e \in V_{min}
\, \}.
$$ 
Recall that $\Omega_h = \Omega \cap \mathbb{Z}_h^d$. The stencil $V_{max}$ is 
defined for $x \in \Omega_h$ as
\begin{equation} \label{vmax}
V_{max}(x) = \{ \, e \in \mathbb{Z}^d \setminus \{0\}, \exists y \in \Omega_{ext}, y = x+ h e \, \}.
\end{equation}

{\bf Assumption} The stencil $V$ is required to satisfy
$$
V_{min} \subset V(x) \subset V_{max}(x), x \in \Omega_h.
$$
Let 
$$
\partial \Omega_h = \{ \, x \in \Omega_h  \text{ such that for some } 
e \in \{ \, \pm r_1,\ldots, \pm r_d \, \}, 
x + h e \notin \Omega_h
\,  \}.
$$
Note that we have by our definition
$$
\partial \Omega_h \subset \Omega_h, 
$$
and if $x \in \Omega_h \setminus \partial \Omega_h$, then for all $e \in \{ \, \pm r_1,\ldots, \pm r_d \, \}$, $x+ h e \in \Omega_h$.

Recall that $\{ \, r_1,\ldots,r_d \, \}$ denotes the canonical basis of $\R^d$. For $x \in \Omega_h$ and $ e \in \mathbb{Z}^d$ let
$
h^e_x = \sup \{ \, r h, r \in [0,1] \ \text{and} \ x+rh e \in \tir{\Omega} \, \}
$. We define
 $$
 \mathcal{N}_h^1 = \Omega_h \cup  \{ \, x \in \partial \Omega,  \exists y \in \Omega_h  \ \text{and} \ e \in 
 V(y) \cup \{ \, 0 \, \}
 \text{ such that } x= y + h^e_y e  
  \}.
 $$
 We also define
 $$
 \mathcal{N}_h^2 = \{ \, x \in \mathbb{Z}^d_h, x=y+he, e \in 
 V_{max}(y) \cup \{ \, 0 \, \} \text{ and } y \in \Omega_h\, \},
 $$
 where the 
 stencil $V_{max}$ is given by \eqref{vmax}, i.e. 
 $e \in V_{max}(x)$ if and only if $x= y - h e$ for $y \in \Omega_{ext} = \Omega_h \cup \{ \, x+ h e: x \in \Omega_h, e \in V_{min} \, \}$. Recall that $V_{min}$ is symmetric with respect to the origin,  contains $(r_1, \ldots, r_d)$ as well as vectors parallel to normals of the facets of $Y$. 
 We have
 $$
 \mathcal{N}_h^1  \subset \tir{\Omega} \subset \Conv( \mathcal{N}_h^2).
 $$
 We claim that $ \mathcal{N}_h^2= \Omega_{ext} $. By definition, $\Omega_h \subset  \mathcal{N}_h^2$ and for $x \in \Omega_h$ and $e \in V_{min}$, $x+h e \in  \mathcal{N}_h^2$ since $V_{min} \subset V_{max}(x)$. Thus $\Omega_{ext} \subset  \mathcal{N}_h^2$. Let $z \in  \mathcal{N}_h^2$, $z = y_1 + h e, y_1 \in \Omega_h, e \in V_{max}(y_1)$. Let $y_2 \in \Omega_{ext}$ such that $y_1 = y_2 - he$. 
Thus $z = y_2$ and $ \mathcal{N}_h^2 \subset \Omega_{ext}$. This gives $\mathcal{N}_h^2 \subset \Omega_{ext}$. The claim is proved.

The unknown in the discrete scheme is a mesh function (not necessarily the interpolant of a convex function) on $\Omega_h$ which is extended to $\mathbb{Z}^d_h$ using the discrete extension formula 
\begin{equation} \label{extension}
v_h(x) = \min_{ y \in \partial \Omega_h }  \big( v_h(y) + \max_{1 \leq j \leq N} (x-y) \cdot a_j^* \big),
\end{equation}
motivated by Theorem \ref{formula2} below.

We consider the following analogue of the subdifferential of a function. For $x \in \mathbb{Z}^d_h$ and a mesh function $v_h$, we define 
$$
\partial_V v_h(x) = \{ \, p \in \R^d, p \cdot  (h e) \geq v_h(x) -v_h(x-h e) \, \forall  e \in V(x)  \, \},
$$
and consider the following discrete version of the R-Monge-Amp\`ere measure 
$$
\omega_V(R,v_h,E) := \int_{\partial_V v_h(E)} R(p) d p,
$$
where $\partial_V v_h(E) = \cup_{x \in E}\partial_V v_h(x)$. 

For the Dirichlet problem, a discrete version of the R-curvature has been used in \cite{qiu2020note} where a generalization of the discretization proposed in \cite{Oliker1988} for $R=1$ was studied. Integration of the density function $R$ (and hence the need of numerical integration) over power diagrams appears in the semi-discrete approach to optimal transport \cite{Yau2013,levy2018notions,kitagawa2016newton,Aurenhammer98,merigot2011multiscale}.

A discretization based on $\partial_V v_h$ may not be accurate for $V=V_{min} $ while for $V=V_{max}$ one may need to use power diagrams and a damped Newton's method as in semi-discrete optimal transport. For the case of the 
stencil $V_{max}$, we define
$$
\omega_a(R,v_h,E) := \int_{\partial_{V_{max}} v_h(E)} R(p) d p.
$$
The discretization considered in \cite{Mirebeau15} used a symmetrization of the subdifferential. The subscript $a$ in the notation $\omega_a(R,v_h,E)$ recalls that we use here an asymmetrical version.

The coordinates of a vector $e \in \mathbb{Z}^d$ are said to be co-prime if their great common divisor is equal to 1.
For a quadratic polynomial $p$ such that $0 < \lambda \leq D^2 p \leq \Lambda$ and for $x \in \R^d, p(x) = 1/2 \ x^T M x$ for a $d \times d$ matrix $M$ with condition number less than $\kappa$ for 
 $\kappa >0$, consistency of $\partial_V p(x)$ at mesh points $x$ 
at a distance $h \sqrt{d} \kappa$ from $\partial \Omega$, 
can be proven as in   \cite{Mirebeau15,Nochetto19}, 
provided $V(x)$ contains all vectors $e \in V_{max}(x)$ with co-prime coordinates such that $|e| \leq 1/2 \sqrt{d} \kappa$. 

For $\kappa >0$, define $V_{\kappa}$ to be a mesh independent stencil such that 
$V_{\kappa}$ consists of all vectors $e \in \mathbb{Z}^d \setminus \{ \, 0 \, \}$ with co-prime coordinates such that $|e| \leq 1/2 \sqrt{d} \kappa$. The factor $1/2 \sqrt{d}$ is motivated by Lemma \ref{condition-number} below.
Given $x \in \Omega_h$ such that $d(x,\partial \Omega) > h \sqrt{d} \kappa$, we have $V_{\kappa} \subset V_{max}(x)$, 
since for $e \in V_{\kappa}$, $|h e| \leq h/2 \sqrt{d} \kappa < h \sqrt{d} \kappa < d(x,\partial \Omega)$ and hence 
$y=x+h e \in \Omega_h \subset \Omega_{ext}$. If necessary, by taking $\kappa$ large, we may assume that $V_{min} \subset V_{\kappa}$.



In section \ref{cvg}, we first prove convergence of the discretization for $V=V_{max}$. Then we allow 
$V=V_{\kappa} \cap V_{max}$
and thus have a two-scale approximation $u_{h,\kappa}$. Note that the size of $V_{\kappa} \cap V_{max}(x)$ for $x \in \Omega_h$, is uniformly bounded in $x$, with an upper bound independent of $N$. For that reason, the complexity of the resulting method is O$(N)$. 

We will show that as $h \to 0$, $u_{h,\kappa}$ converges uniformly on $\tir{\Omega}$ to a continuous function $v_{\kappa}$ which solves $R(D v) \det D^2 v = f$ in in the sense of viscosity. For $f \in C(\Omega)$, we then compare $v_{\kappa}$ to a class of strictly convex quadratic polynomials parameterized by $\kappa$. The limit as $\kappa \to + \infty$ of $v_{\kappa}$ is a convex function which solves \eqref{m2}.



We define for a function $v_h$ on $\mathbb{Z}^d_h$, $e \in \mathbb{Z}^d$ and $x \in \mathbb{Z}^d_h$ 
$$
\Delta_{h e} v_h (x) =  v_h(x+ h e ) - 2 v_h(x) + v_h(x- h e ). 
$$

\begin{definition}  \label{def-dc}
A mesh function $v_h$ on $\Omega_h$ extended to $\mathbb{Z}^d_h$ using \eqref{extension} is discrete convex if $\Delta_{h e} v_h(x) \geq 0$ for all 
$x \in \Omega_h$ and $e \in V_{max}(x)$ such that $x\pm he \in \mathcal{N}_h^2$. A mesh function $v_h$ is $V$-discrete convex if $\Delta_{h e} v_h(x) \geq 0$ for all 
$x \in \Omega_h$ and $e \in V_{}(x)$ such that $x\pm he \in \mathcal{N}_h^2$. 
\end{definition}

A $V_{max}$-discrete convex mesh function is discrete convex. Denote by $\mathcal{C}_h$ the set of discrete convex mesh functions.  

\begin{definition} \label{def-as-cone}
A mesh function on $\Omega_h$ which is extended to $\mathbb{Z}^d_h$ using the discrete extension formula \eqref{extension}, and which is discrete convex is said to have asymptotic cone $K$ associated with $Y$.
\end{definition}

Below, we will consider only discrete convex mesh functions with asymptotic cone $K$. We can now describe our discretization of the second boundary value problem: find $u_h \in \mathcal{C}_h$ with asymptotic cone $K$ such that
\begin{align} \label{m2d3}
\begin{split}
\omega_V(R,u_h,\{\, x \,\})&= \int_{E_x} \tilde{f}(t) dt , x \in \Omega_h,
\end{split}
\end{align}
where $(E_x)_{x \in \Omega_h}$ form a partition of $\Omega$, i.e. $E_x \cap \Omega_h = \{ \, x \, \}, \cup_{x \in \Omega_h} E_x = \tir{\Omega}$, and $E_x \cap E_y$ is a set of measure 0 for $x \neq y$. In the interior of $\Omega$ one may choose as
 $E_x=x+[-h/2,h/2]^d$ the cube centered at $x$ with $E_x \cap \Omega_h = \{ \, x \, \}$. The requirement that the sets $E_x$ form a partition is essential to assure the mass conservation \eqref{necessary2} at the discrete level, i.e. 
\begin{align} \label{mass-conservation}
\sum_{x \in \Omega_h} \omega_V(R,u_h,\{\, x \,\}) = \sum_{x \in \Omega_h} \int_{E_x} \tilde{f}(t) dt = \int_{\Omega}  \tilde{f}(t) dt = \int_{Y} R(p) d p. 
\end{align} 
The unknowns in \eqref{m2d3} are the mesh values $u_h(x), x \in \Omega_h$. For $z \notin \Omega_h$, the value $u_h(z)$ needed for the evaluation of $\partial_{V} v_h (x) $ is obtained from the discrete extension formula \eqref{extension}.

Let $u_h$ be discrete convex with asymptotic cone $K$. Recall that the values of $u_h$ on $ \mathcal{N}_h^1 \setminus \Omega_h$ are given by \eqref{extension}. 
Let
\begin{align*}
\partial_h u_h(x) = \{ \, p \in \R^d, u_h(y) \geq u_h(x) + p \cdot (y-x) \, \forall  y \in \mathcal{N}_h^1
\, \},
\end{align*}
and recall that
$$
\partial_V u_h(x) = \{ \, p \in \R^d, p \cdot  (h e) \geq u_h(x) -u_h(x-h e) \, \forall  e \in V(x)  \, \}.
$$
We consider two kinds of convex envelopes of the mesh function $u_h$ 
\begin{align*} 
\Gamma_1(u_h)(x) &= \sup_{L \text{ affine} }\{ \, L(x): L(y) \leq u_h(y)  \text{ for all } y \in \mathcal{N}_h^1 \, \} \text{ and }\\
\Gamma_2(u_h)(x) &= \sup_{L \text{ affine} }\{ \, L(x): L(y) \leq u_h(y)  \text{ for all } y \in \mathcal{N}_h^2
\, \},
\end{align*}
which are piecewise linear convex functions, c.f. for example \cite[p. 11]{awanou2019uweakcvg}. We note that $ \mathcal{N}_h^1$ depends on the stencil $V$. 
Note also that the definition of the convex envelope $\Gamma_1(u_h)$ above allows an ''infinite slope'' at points of $\R^d$ not in $\Conv (\mathcal{N}_h^1)$. 
If $u$ is a convex function on $\Omega$, 
we can extend $u$ to $\R^d$, c.f. \eqref{ext-sh} below, by
\begin{equation*} 
\tir{u}(x) = \sup \{ \, u(y) +  (x-y) \cdot z, y \in \Omega, z \in \partial u(y)  \, \}.
\end{equation*}
We denote by $\chi_{u}$ the subdifferential of the extended function to $\R^d$. 
Thus $\chi_{\Gamma_1(u_h)}$ denotes the subdifferential of the extension to $\R^d$ 
of $\Gamma_1(u_h)$, 
i.e. for $x \notin \Conv(\mathcal{N}^1_h)$
\begin{equation} \label{gamma1-ext}
\Gamma_1(u_h) (x) = \sup \{ \, \Gamma_1(u_h)(y) + q \cdot (x-y), y \in (\Conv(\mathcal{N}^1_h))^\circ, 
q \in \partial \Gamma_1(u_h)(y)
\, \},
\end{equation}
where for a set $D$, $D^\circ$ denotes its interior.

In \cite{awanou2019uweakcvg}, we introduced the notation
\begin{equation*} 
\Delta_e v_h (x) = \frac{2}{h^e_x + h^{-e}_x} \bigg( \frac{v_h(x+ h^e_x e ) - v_h(x)}{h^e_x}  + \frac{v_h(x- h^{-e}_x e ) - v_h(x)}{h^{-e}_x} \bigg).
\end{equation*}
A notion of $V$-discrete convexity was introduced in \cite[Definition 3]{awanou2019uweakcvg} by requiring $\Delta_e v_h (x) \geq 0$ for all $e \in V(x)$. Therein the focus was on mesh functions which converge to a convex function. To require that discrete convexity holds on all directions supported by the mesh, $V$ was taken as  $V=\mathbb{Z}^d \setminus \{ \, 0 \, \}$, which is not correct. 

The correct definition of discrete convexity in the sense of \cite{awanou2019uweakcvg} is to require that $\Delta_e v_h (x) \geq 0$ for all $e \in \mathbb{Z}^d$ for which $x+ h^e_x e \in  \mathcal{N}_h^1$ and $x- h^{-e}_x e \in  \mathcal{N}_h^1$. 




The above remark also applies to the work in \cite{DiscreteAlex2}. In addition, the convergence analysis therein for the Dirichlet problem, holds for a stencil $V_{max}$ which 
contains $\{ \, e \in \mathbb{Z}^d, x+ h^e_x e \in  \mathcal{N}_h^1 \, \}$.

The following theorem follows from  \cite[Lemmas 6 and 7]{awanou2019uweakcvg}, \cite[Theorem 6]{awanou2019uweakcvg}  and \cite[Theorem 4]{awanou2019uweakcvg} where we considered $\partial_h u_h$ in connection with $\Gamma_1(u_h)$. 

\begin{theorem} \label{old-01} 
If $x \in \Omega_h$ and 
$\Gamma_1(u_h)(x) = u_h(x)$, then $\partial \Gamma_1(u_h)(x) =  \partial_h u_h(x)$. If $x \in \Omega_h$ and $\Gamma_1(u_h)(x) \neq u_h(x)$, then $ \partial_h u_h(x)=\emptyset$.  If $x \in \Conv (\mathcal{N}_h^1)$, for any  $p \in   \chi_{\Gamma_1(u_h)}(x)$, $\exists y \in \mathcal{N}_h^1$ such that $p \in   \chi_{\Gamma_1(u_h)}(x) \cap \chi_{\Gamma_1(u_h)}(y)$ and $\Gamma_1(u_h)(y) = u_h(y)$. 

Moreover, for a  subset $E \subset (\Conv (\mathcal{N}_h^1))^\circ $, $\partial_h u_h(E) = \partial \Gamma_1(u_h)(E)$ up to a set of measure 0 and thus
$$
\omega(R,\Gamma_1(u_h),E) = \int_{\partial_h u_h(E)} R(p) dp.
$$
\end{theorem}

Analogous to Theorem \ref{old-01}, we have 

\begin{theorem} \label{from-out} 
If $x \in \Omega_h$ and 
$\Gamma_2(u_h)(x) = u_h(x)$, then $\partial \Gamma_2(u_h)(x) =  \partial_{V_{max}} u_h(x)$. If $x \in \Omega_h$ and $\Gamma_2(u_h)(x) \neq u_h(x)$, then $ \partial_{V_{max}} u_h(x)=\emptyset$.  If $x \in \Conv (\mathcal{N}_h^2)$, for any  $p \in   \chi_{\Gamma_2(u_h)}(x)$, $\exists y \in \mathcal{N}_h^2$ such that $p \in   \chi_{\Gamma_2(u_h)}(x) \cap \chi_{\Gamma_2(u_h)}(y)$ and $\Gamma_2(u_h)(y) = u_h(y)$. 

Moreover, for a  subset $E \subset (\Conv (\mathcal{N}_h^2))^\circ $, $\partial_{V_{max}} u_h(E) = \partial \Gamma_2(u_h)(E)$ up to a set of measure 0 and thus
$$
\omega_a(R,u_h,E) = \omega_a(R,\Gamma_2(u_h),E).
$$
\end{theorem}

\begin{remark} \label{dc-in-the-sense}
We observe that if $f>0$ on $\Omega$ and $V=V_{max}$, a mesh function $v_h$ which solves \eqref{m2d3}
is discrete convex, as defined in \cite{awanou2019uweakcvg}. 
This follows from Lemma \ref{gamma-lem} below which gives $v_h = \Gamma_1(v_h)$ on $\mathcal{N}_h^1$. Since $\Gamma_1(v_h)$ is piecewise linear convex on $\mathcal{N}_h^1$,  $\Delta_e v_h (x) \geq 0$ for all $x \in \Omega_h$, i.e. $v_h$ is discrete convex as defined in \cite{awanou2019uweakcvg}.
\end{remark}

The next lemma shows that the $V$-discrete convexity assumption is automatically satisfied for a discrete solution when $f>0$.

\begin{lemma} \label{auto-dc}
If $f>0$ in $\Omega$, a mesh function on $\Omega_h$ extended to $\mathbb{Z}^d_h$ using \eqref{extension}, and which solves 
  \eqref{m2d3} is $V$-discrete convex.
\end{lemma}

\begin{proof} 
It is a consequence of Lemma \ref{inc-sub-lem} below that a discrete convex mesh function $v_h$ which solves  \eqref{m2d3} using the discrete extension formula  \eqref{extension} satisfies $\partial_V v_h(\Omega_h) \subset Y \subset \Omega^*$. Recall that $R>0$ on $\Omega^*$. If $f>0$ in $\Omega$, and $x \in \Omega_h$, we have $\omega_V(R,u_h,\{\, x \,\}) >0$ and hence $\partial_V v_h(x) \subset \Omega^*$ is a set with a non 
zero Lebesgue measure. In particular, it is non empty. Assume that $e \in V(x)$ 
and $x\pm he \in \mathcal{N}_h^2$. For $p \in \partial_V v_h(x)$, we have
$$
v_h(x) -v_h(x-h e) \leq p \cdot (h e) \leq v_h(x+h e) -v_h(x). 
$$
This implies that $v_h(x) -v_h(x-h e) \leq v_h(x+h e) -v_h(x)$ and 
hence $\Delta_{h e} v_h (x) \geq 0$ for all $e \in V(x)$. \Qed
\end{proof}

\begin{remark}
From Lemma \ref{auto-dc}, the $V$-discrete convexity assumption does not need to be explicitly imposed when $f>0$ in $\Omega$. 
However, unless $V = V_{max}$, 
uniform limit of 
$V$-discrete convex mesh functions need not be convex. 
\end{remark}

The support function $k_{Y}$ of the closed convex set $Y$ 
is defined by
$
k_{Y}(p) = \sup_{z \in Y} p \cdot z.
$
The definition essentially says that for the direction $p$, $Y$ lies on one side of the hyperplane $p \cdot z = k_{Y}(p)$. 
For $x=(x_1,\ldots,x_d) \in R^d$, put $||x||_1=\sum_{i=1,\ldots,d} |x_i |$.

We need the following lemma which follows from \cite[Proposition 4.3]{benamou2017minimal}. 


\begin{lemma} \label{pre-lip-lem}
Let $v_h$ be a mesh function and 
$e \in V_{min}$ such that $\Delta_{h e} v_h(x) \geq 0$ for $x \in \Omega_h$, with $v_h(x)$ for 
$x \notin \Omega_h$ given by \eqref{extension}. 
Then, for integers $k$ and $l$ with $k \geq 0, l \leq 0$ such that $x+k h e$ and $x+l he$  are in $\Omega_h$
\begin{multline} \label{lipsh-1}
 -k_{Y}(-h e) \leq v_h(x+ l he) - v_h(x+(l-1) h e) \\ \leq  v_h(x+(k+1) he)-v_h(x+k h e)  \leq k_{Y}(h e). 
\end{multline}
Moreover
\begin{equation} \label{Lipschitz}
 |v_h(x)-v_h(y)| \leq C ||x-y||_1,
\end{equation}
for $x, y \in \tir{\Omega} \cap \mathbb{Z}^d_h$ and for a constant $C=\max \{\, | k_{Y}(-r_i)|, | k_{Y}(r_i)|, i=1,\ldots,d  \, \}$ independent of $h$ and $v_h$.
\end{lemma}

\begin{proof}
%

Let $x \in \Omega_h$ and $e \in V_{min}$. Since by assumption
$\Delta_{h e} v_h(x) \geq 0$, we have
$$
v_h(x+h e) - v_h(x) \geq v_h(x) - v_h(x-h e). 
$$
Therefore for integers $k$ and $l$ with $k \geq 0, l \leq 0$ such that $x+k h e$ and $x+l he$  are in $\Omega_h$
$$
v_h(x+(k+1) he) - v_h(x+k h e) \geq v_h(x+ l he) - v_h(x+(l-1) h e). 
$$
Let us now assume that $k$ and $e$ are such that $x+k h e \in \Omega_h$ but $x+(k+1) he \notin \Omega_h$. Then by definition, since $x+k he \in \partial \Omega_h$
\begin{align*}
v_h(x+(k+1) he) \leq  \max_{1 \leq j \leq N} h e \cdot a_j^* +v_h(x+k h e).
\end{align*}
It follows that
$$
v_h(x+(k+1) he)-v_h(x+k h e) \leq \max_{1 \leq j \leq N} h e \cdot a_j^*. 
$$
This can be written
$$
v_h(x+(k+1) he)-v_h(x+k h e) \leq k_{Y}(h e). 
$$
Assume now that $x+(l-1) h e \notin \Omega_h$ but $x+l h e \in \Omega_h$. Then
$$
v_h(x+(l-1) h e) \leq  \max_{1 \leq j \leq N} -h e \cdot a_j^* + v_h(x+ l he).
$$
It follows that
$$
 v_h(x+ l he) - v_h(x+(l-1) h e) \geq  -k_{Y}(-h e).
$$
In summary, for integers $k$ and $l$ with $k \geq 0, l \leq 0$ such that $x+k h e$ and $x+l he$  are in $\Omega_h$ \eqref{lipsh-1} holds.  
%

The proof of \eqref{Lipschitz} is given in \cite[Proposition 4.3 (5)]{benamou2017minimal}. Note that in \eqref{lipsh-1}, $x+(k+1) he$ and $x+(l-1) h e$ may not be in $\Omega_h$. Let now $x$ and $y$ in 
$\tir{\Omega} \cap \mathbb{Z}^d_h$ and put $y=x+\sum_{i=1}^d l_i h r_i$ where we recall that $\{ \, r_1,\ldots,r_d \, \}$ denotes the canonical basis of $\R^d$ and its elements are in $V(z)$ for all $z \in \Omega_h$ by assumption. Rewriting \eqref{lipsh-1} as
\begin{align*}
 -k_{Y}(-h e) & \leq v_h(x+ l he) - v_h(x+(l-1) h e)  \leq k_{Y}(h e) \\
 -k_{Y}(-h e) & \leq v_h(x+(k+1) he) - v_h(x+k h e)   \leq k_{Y}(h e),
\end{align*}
we see that if $l_i \geq 0$, we have $-l_i h k_{Y}(-r_i) \leq v_h(x+ l_i h r_i) - v_h(x) \leq l_i h k_{Y}(r_i)$ while when
$l_i \leq 0$, $- |l_i| h k_{Y}(-r_i)  \leq v_h(x) - v_h(x+ l_i h r_i) \leq |l_i| h k_{Y}(r_i)$. Therefore
\begin{equation} \label{temp-Lip01}
|v_h(x+ l_i he) - v_h(x)| \leq |l_i| h  \max \{\, | k_{Y}(-r_i)|, | k_{Y}(r_i)|  \, \}, 
\end{equation}
which gives
$$
|v_h(y) - v_h(x)| \leq h \sum_{i=1}^d |l_i| \max \{\, | k_{Y}(-r_i)|, | k_{Y}(r_i)|, i=1,\ldots,d  \, \}.
$$
The proof is complete. \Qed
\end{proof}

The next lemma describes how the discrete extension formula \eqref{extension} enforces the second boundary condition.

\begin{lemma} \label{inc-sub-lem}
Assume that  $\Delta_{he} v_h(x) \geq 0$ for all $x$ in $\Omega_h$ and $e \in V_{min} \subset V(x)$, with $v_h(x)$ for 
$x \notin \Omega_h$ given by \eqref{extension}. We have
$$
\partial_V v_h  (\Omega_h)  \subset Y.
$$
\end{lemma}

\begin{proof}
With $k=l=0$ in Lemma \ref{pre-lip-lem}, we obtain for $x \in \Omega_h$ and $e \in V_{min}$
\begin{equation} \label{pre-lip-lem-case}
 -k_{Y}(-h e) \leq v_h(x) - v_h(x- h e) \leq v_h(x+he)-v_h(x)  \leq k_{Y}(h e). 
\end{equation}
Let $p \in   \partial_V v_h  (x)$. Since for $e \in V_{min}$, $-e \in V_{min}$, we have
$p \cdot (- h e) \geq v_h(x) - v_h(x+ h e)$ for all $e \in V(x)$, that is
$$
p \cdot ( h e) \leq  v_h(x+ h e) - v_h(x) \leq k_{Y}(h e) = h k_{Y}( e) .
$$
This proves that $p \cdot  e \leq k_{Y}(e)$ for all $e \in V_{min}$. Since $V_{min}$ contains vectors parallel to the normals to facets of the polygon $Y$, we conclude that $p \in Y$ and thus $ \partial_V v_h(\Omega_h) \subset Y$.
The proof is complete. \Qed
\end{proof}

\section{Stability, uniqueness and existence} \label{an-scheme}

Adding a constant to a solution of  \eqref{m2d3} results in another solution. We will require that $v_h(x^1_h)=\alpha$ for an arbitrary number $\alpha$ and a mesh point $x^1_h$. Recall that $v_h$ is defined only at mesh points. We will assume that $x^1_h \to x^1$ for a point $x^1 \in \tir{\Omega}$.  

The stability of solutions is an immediate consequence of \eqref{Lipschitz}. 
\begin{theorem} \label{stability-thm}
Solutions $v_h \in \mathcal{C}_h$ of  \eqref{m2d3} with $v_h(x^1_h)=\alpha$ for an arbitrary number $\alpha$ and $x^1_h \in \Omega_h$, are bounded independently of $h$. 
\end{theorem}

\begin{proof} 
Since for $v_h \in \mathcal{C}_h$ and $x \in \Omega_h$, $\Delta_{h e} v_h (x) \geq 0$ for all $e \in V$,  $v_h$ is bounded independently of $h$ by \eqref{Lipschitz}. \Qed
\end{proof}


\begin{theorem} \label{unicity}
For $f>0$ in $\Omega$, solutions of the discrete problem \eqref{m2d3}  are unique up to an additive constant for $V=V_{max}$. 
\end{theorem}

\begin{proof} 
The proof is the same as 
the proof of uniqueness of a solution to \eqref{m1} in the class of convex polyhedra, i.e. when the right hand side is a sum of Dirac masses. See for example \cite[Theorem 17.2]{Bakelman1994} for a sketch of the proof for convex polyhedra. The proof therein requires non trivial Dirac masses, hence our assumption that $f>0$. 

We first note that if $u_h$ is a solution of \eqref{m2d3}, then $u_h + C$ is also a solution of \eqref{m2d3} for a constant $C$. Let $v_h$ and $w_h$ be two solutions of \eqref{m2d3}. We may assume that $v_h(x) \geq w_h(x)$ for all $x \in \Omega_h$, if necessary by adding a constant to $w_h$. Furthermore, we may also assume that there exists $x^1 \in \Omega_h$ such that $v_h(x^1)=w_h(x^1)$. For convenience, and by an abuse of notation, we do not mention the dependence of $x^1$ on $h$. To prove the existence of $x^1$, let $a=\min\{ \, v_h(x)-w_h(x), x \in \Omega_h\, \}$. Since $\Omega_h$ is finite, there is $x^1 \in \Omega_h$ such that $a=v_h(x^1)-w_h(x^1)$.
With $s_h(x)=w_h(x)+a$, we obtain $v_h(x)\geq s_h(x)$ for all $x \in \Omega_h$ with $v_h(x^1)=s_h(x^1)$.

It follows from \eqref{extension} that $v_h \geq w_h$ on $\mathbb{Z}^d_h$. We show that $v_h=w_h$ and hence any two solutions can only differ by a constant. 

Since $v_h(x^1)=w_h(x^1)$ and  $v_h(x) \geq w_h(x)$ for all $x \in \mathbb{Z}^d_h$, we have 
$\partial_V w_h(x^1) \subset \partial_V v_h(x^1)$. Next, we note that as $f>0$ in $\Omega$, $\partial_V w_h(x^1)$ is a non empty polygon with facets given by hyperplanes orthogonal to directions $e$ in a subset $\hat{V}$ of $V$. We consider a subset of $V$ because some faces may only intersect  $\partial_V w_h(x^1)$ at a vertex. 

If there is some $\hat{e} \in \hat{V}$ such that  $v_h (x^1+h \hat{e} ) > w_h (x^1+h \hat{e})$, then $\partial_V v_h(x^1)  \setminus \partial_V w_h(x^1)$ has non zero measure. Since $R >0$ on $\Omega^*$ and by Lemma \ref{inc-sub-lem} $\partial_V v_h(x^1) \subset \Omega^*$, and by assumption $\omega_a(R,v_h,\{\, x^1\, \})=\omega_a(R,w_h,\{\, x^1\, \})$, this is impossible from properties of the Lebesgue integral. 
We have proved that under the assumption that $v_h(x^1)=w_h(x^1)$ we must have  
$(v_h-w_h)(x^1\pm h e)=0  \ \forall e \in \hat{V}$. 

Let $P_1$ denote the convex hull of $x^1$ and the points $x^1 + h e, e \in \hat{V}$. By Lemma \ref{gamma-lem} below we have $v_h=\Gamma_2(v_h)$ on $\Omega_h$. Recall that $\Gamma_2(v_h)$ is a piecewise linear convex function. 
Also $w_h=\Gamma_2(w_h)$ on $\Omega_h$. Therefore $\Gamma_2(v_h)=\Gamma_2(w_h)$ on $\partial P_1$ with $\omega_a(R,\Gamma_2(v_h),\{\, x^1\, \})=\omega_a(R,\Gamma_2(w_h),\{\, x^1\, \})$. 
Because $\Gamma_2(v_h)$ and $\Gamma_2(w_h)$ are piecewise linear convex, by construction of $\hat{V}$, at all other points $x$ of $P_1$, we have $\omega_a(R,\Gamma_2(v_h),\{\, x\, \})=\omega_a(R,\Gamma_2(w_h),\{\, x\, \})=0$.
By unicity of the solution to the Dirichlet problem for the Monge-Amp\`ere equation \cite[Theorem 2.1]{Trudinger2008}, we obtain $\Gamma_2(v_h)=\Gamma_2(w_h)$ on $P_1$. Hence $v_h=w_h$ on $P_1\cap \Omega_h$. 

Next, we choose a point $x^2$ on $\partial P_1\cap \Omega_h$ and denote by $P_2$ the corresponding polygon. Repeating this process with points on $\partial P_{i-1}\cap \Omega_h, i>2$, we obtain a sequence of mesh points $x^i$ and associated polygons $P_i$ of non zero volumes on which $v_h=w_h$. 

Next, we observe that $\cup_i P_i = \Conv(\mathcal{N}_h^2)$ as the points $x^i$ are projections onto $\R^d$ of vertices on the lower part of the convex polygon which is the epigraph of $\Gamma_2(w_h)$ on $\Conv(\mathcal{N}_h^2)$. We conclude that $v_h=w_h$. \Qed
\end{proof}


\begin{lemma} \label{perturbation-lem}
Let $x^1 \in \Omega_h$ and $v_h$ be discrete convex with asymptotic cone $K$. Assume that 
$\omega_V(R, v_h, \{ \, x \, \}) >0$ for all $x \in \Omega_h$. Let $w_h$ and $q_h$ be defined on $\mathbb{Z}_h^d$ by
$w_h(x) = v_h(x)$ for $x \neq x^1, x \in \Omega_h$ and $w_h(x^1) = v_h(x^1) - \epsilon$, 
$q_h(x) = v_h(x)$ for $x \neq x^1, x \in \Omega_h$ and $q_h(x^1) = v_h(x^1) + \epsilon$. The values of $w_h$ and $q_h$ on $\mathbb{Z}_h^d \setminus \Omega_h$ are given by \eqref{extension}.
There exists $\epsilon_0 >0$ such that for 
$0 < \epsilon \leq \epsilon_0$, $w_h$ and $q_h$ are discrete convex with asymptotic cone $K$, 
$q_h \geq v_h \geq w_h$
on $\mathbb{Z}_h^d$. Moreover, if $x^1 \in \Omega_h \setminus \partial \Omega_h$, $\omega_V(R, w_h, \{ \, x^1 \, \} )>
\omega_V(R, v_h, \{ \, x^1 \, \} )>
\omega_V(R, q_h, \{ \, x^1 \, \} )$. 
\end{lemma}

\begin{proof} 
Let $\epsilon_1 = \min \{ \,
\Delta_{he} v_h(x), x \in \Omega_h, e \in V(x), x\pm he \in \mathcal{N}_h^2 \, \}$. We have $\epsilon_1>0$
since $\omega(R, v_h, \{ \, x \, \}) >0$ for all $x \in \Omega_h$. Otherwise, there would be $x_0 \in \Omega_h$ and a direction $e \in V(x_0)$ such that $\Delta_{h e} v_h(x_0)=0$. In that case, $\partial_V v_h(x_0)$ is contained in the hyperplane $p \cdot e = (v_h(x_0+h e)-v_h(x_0) )/h =(v_h(x_0)-v_h(x_0-h e))/h$, and hence $\omega_V(R,v_h,\{\, x_0 \,\}) =0$, a contradiction.  

Let $\epsilon>0$. We have  $\Delta_{he}w_h(x^1) =  \Delta_{he}v_h(x^1)+ 2 \epsilon \geq \epsilon_1 + 2 \epsilon$. 
We claim that $\Delta_{he}w_h(x) \geq \epsilon_1 - 2 \epsilon$ for all $x \in \Omega_h, x \neq x^1$. This is because, for $x \in \Omega_h$ and $e \in V(x)$, $w_h(x+he) \geq v_h(x+he) - \epsilon$. When $x+h e \in \Omega_h$, this follows from the definition of $w_h$. Assume that $x+h e \in \mathbb{Z}^d_h \setminus \Omega_h$ and put $\psi(s)=\max_{j=1,\ldots,N} (x+he - s)\cdot a_j^*$. Let $s_0 \in \partial \Omega_h$ such that $w_h(x+he) = w_h(s_0)+\psi(s_0)$. If $s_0=x^1$ and $v_h(x+he)=v_h(s_0)+\psi(s_0)$, we have
$w_h(x+he) = v_h(x+he) - \epsilon$. If $s_0=x^1$ and $v_h(x+he)=v_h(s_1)+\psi(s_1)$ for $s_1 \neq s_0$, then by definition $v_h(s_0)+\psi(s_0) \geq v_h(x+he)$ and thus $w_h(x+he)=v_h(s_0)-\epsilon+\psi(s_0) \geq v_h(x+he)-\epsilon$. When $s_0 \neq x^1$ we have $w_h(x+he)=v_h(s_0)+\psi(s_0) \geq v_h(x+he)$. 
This proves the claim when $x+h e \in \mathbb{Z}^d_h \setminus \Omega_h$. 

With a similar argument, we have $\Delta_{he}q_h(x^1) =  \Delta_{he}v_h(x^1)- 2 \epsilon \geq \epsilon_1 - 2 \epsilon$
and $\Delta_{he}q_h(x) \geq \epsilon_1$ for all $x \in \Omega_h, x \neq x^1$.

We have  $\Delta_{he}w_h(x) \geq \epsilon_1 - 2 \epsilon$ for all $x \in \Omega_h$. We conclude that for $\epsilon \leq \epsilon_1/2$, $w_h$ is discrete convex. By construction $w_h$ has asymptotic cone $K$. Similarly, $\Delta_{he}q_h(x) \geq \epsilon_1 - 2 \epsilon$ for all $x \in \Omega_h$. So, for $\epsilon \leq \epsilon_1/2$, $q_h$ is discrete convex with asymptotic cone $K$.

It is immediate that $q_h \geq v_h \geq w_h$
on $\mathbb{Z}_h^d$.  Let $\{ \, e_1,\ldots, e_m \, \} \subset \mathbb{Z}^d$ denote a set of normals to the facets of $\partial_V q_h(x^1)$  
and let $\{ \, s_1,\ldots, s_n \, \} \subset \mathbb{Z}^d$ denote a set of normals to the facets of $\partial_V w_h(x^1)$. 
By construction of $\partial_V v_h(x^1)$, $\{ \, e_1,\ldots, e_m \, \} \subset V(x^1)$. Similarly $\{ \, s_1,\ldots, s_n \, \} \subset V(x^1)$.
When $x^1 \in \Omega_h \setminus \partial \Omega_h$, we get $v_h(x)=w_h(x)=q_h(x)$ for $x \neq x^1$. Thus
\begin{align*}
\partial_V w_h(x^1) = 
\{ \, p \in \R^d, p \cdot (h s_j) \leq w_h(x^1 + h s_j) - w_h (x^1), j=1,\ldots,n \, \} \\
=  \{ \, p \in \R^d, p \cdot (h s_j) \leq v_h(x^1 + h s_j) - v_h (x^1) + \epsilon, j=1,\ldots,n \, \} \\
\supsetneq \{ \, p \in \R^d, p \cdot (h s_j) \leq v_h(x^1 + h s_j) - v_h (x^1), j=1,\ldots,n \, \} \\
\supset  \{ \, p \in \R^d, p \cdot (h e) \leq v_h(x^1 + h e) - v_h (x^1), \forall e \in V(x^1) \, \}  = 
\partial_V v_h(x^1).
\end{align*}
We conclude that $|\partial_V w_h(x^1)|  > |\partial_V v_h(x^1)|$. Similarly, 
\begin{multline*}
\partial_V q_h(x^1) = \{ \, p \in \R^d, p \cdot (h e_i) \leq q_h(x^1 + h e_i) - q_h (x^1), i=1,\ldots,m \, \}. 
\end{multline*}
This gives $\partial_V q_h(x^1) =  \{ \, p \in \R^d, p \cdot (h e_i) \leq v_h(x^1 + h e_i) - v_h (x^1) - \epsilon, i=1,\ldots,m \, \} \subsetneq   \{ \, p \in \R^d, p \cdot (h e_i) \leq v_h(x^1 + h e_i) - v_h (x^1), i=1,\ldots,m \, \} = \partial_V v_h(x^1)$. 
This implies $|\partial_V q_h(x^1)|  < |\partial_V v_h(x^1)|$. The proof is complete with $\epsilon_0 = \epsilon_1/2$. \Qed
\end{proof}

When $V \neq V_{max}$, it may be necessary to have additional requirements for uniqueness. 
Let $u_h$ be a solution of \eqref{m2d3} and let us assume that we have $\mathcal{N}_h^2 = \mathcal{N}_{h,a}^2 \cup \mathcal{N}_{h,b}^2$ with  $\mathcal{N}_{h,a}^2 \cap \mathcal{N}_{h,b}^2=\emptyset$. Assume furthermore that for $x \in \Omega_h \cap \mathcal{N}_{h,a}^2$, and $e \in V(x)$ such that $e$ is a normal to a facet of $\partial_V u_h(x)$, we have $x + h e \in \mathcal{N}_{h,a}^2$. A similar requirement is made for $x \in \Omega_h \cap \mathcal{N}_{h,b}^2$. Then, adding a constant to $u_h$ on $\mathcal{N}_{h,b}^2$ may result in another solution. 

In the next theorem, we observe that when $V \neq V_{max}$, if $u_h$ and $v_h$ are solutions and $u_h$ is not equal to $v_h$ up to a constant, it is not possible to have $u_h\geq v_h$ up to a constant with equality only at one point. 

\begin{theorem} \label{unicity2} 
Assume that $f>0$ in $\Omega$ and $V_{min} \subset V \subset V_{max}$. Let $u_h$ and $v_h$ be two solutions of 
the discrete problem \eqref{m2d3} such that up to a constant added to $u_h$, we have $u_h \geq v_h$ on $\Omega_h$. 
Then it is not possible to have equality up to a constant only at one point $x^1 \in \Omega_h \setminus \partial \Omega_h$. If in addition  $V(x) = V_{max}(x)$ for all $x \in \partial \Omega_h$, then it is not possible to have equality up to a constant only at one point $x^1 \in \partial \Omega_h$. 
\end{theorem}

\begin{proof} Let $u_h$ and $v_h$ be two 
mesh functions which are discrete convex with asymptotic cone $K$. 

{\bf Part 1} Assume that there exists $z \in \Omega_h$ such that $u_h(x) - v_h(x) \geq u_h(z) - v_h(z)$ for all $x \in \Omega_h$. We prove that $ \omega_V(R, u_h,\{\, z\, \})  \geq  \omega_V(R, v_h,\{\, z\, \}) $. 

We claim that for $x \notin \Omega_h$ we have $u_h(x) - v_h(x) \geq u_h(z) - v_h(z)$. Let $y_1$ and $y_2$ in $\partial \Omega_h$ such that $u_h(x) = u_h(y_1) + k_Y(x-y_1)$ and $v_h(x) = v_h(y_2) + k_Y(x-y_2)$. We have by definition of $v_h(x)$, $v_h(y_2) + k_Y(x-y_2) \leq v_h(y_1) + k_Y(x-y_1)$. Moreover
\begin{align*}
u_h(x) - v_h(x) &= u_h(y_1) - v_h(y_2) + k_Y(x-y_1) - k_Y(x-y_2) \\
& \geq u_h(y_1) - v_h(y_2)  + v_h(y_2)  - v_h(y_1)  = u_h(y_1) -  v_h(y_1) \\
& \geq u_h(z) - v_h(z),
\end{align*}
since $\partial \Omega_h \subset \Omega_h$.

Next, for $e \in V(z)$, we have 
\begin{align*}
u_h(z+ h e) - v_h(z+ h e)  \geq u_h(z) - v_h(z),
\end{align*}
and thus for $p \in \partial_V v_h (z)$
\begin{align*}
u_h(z+ h e) & \geq v_h(z+ h e)  + u_h(z) - v_h(z)  
\geq v_h(z) + p \cdot (h e) + u_h(z) - v_h(z)  \\
& = u_h(z)  + p \cdot (h e), 
\end{align*}
which shows that $p \in  \partial_V u_h (z)$. This proves the claim.

{\bf Part 2}  Let $\epsilon >0$. Assume now that $u_h$ and $v_h$ are two solutions of \eqref{m2d3}. For all $x \in \Omega_h$
$$
\omega_V(R, u_h, \{ \, x \, \}) = \omega_V(R, v_h, \{ \, x \, \} )>0.
$$
As in the proof of Theorem \ref{unicity}, we may assume that $u_h(x) \geq v_h(x)$ for all $x \in \Omega_h$ with  
$u_h(x^1) = v_h(x^1)$ for some $x^1 \in \Omega_h$. By assumption, for $x \in \Omega_h$ and $x \neq x^1$, $u_h(x)>v_h(x)$.

We consider the case that $x^1 \in \Omega_h \setminus \partial \Omega_h$ so that we can use Lemma \ref{perturbation-lem}.  Let $w_h$ denote the perturbation of $v_h$ constructed in Lemma \ref{perturbation-lem} with $\epsilon$. We have 
 $$
 \omega_V(R, w_h, \{ \, x^1 \, \}) >
\omega_V(R, v_h, \{ \, x^1 \, \}).
$$
Let $q_h$ denote the perturbation of $u_h$ constructed in Lemma \ref{perturbation-lem}. We have 
 $$
 \omega_V(R, u_h, \{ \, x^1 \, \}) >
 \omega_V(R, q_h, \{ \, x^1 \, \}). 
$$
Since $\omega_V(R, u_h, \{ \, x^1 \, \}) = \omega_V(R, v_h, \{ \, x^1 \, \} )$, we obtain
\begin{equation} \label{unicity2-con01}
 \omega_V(R, w_h, \{ \, x^1 \, \})  >  \omega_V(R, q_h, \{ \, x^1 \, \}). 
\end{equation}
Recall that for $\epsilon$ sufficiently small, both $w_h$ and $q_h$ are discrete convex with asymptotic cone $K$. Assume that $u_h \neq v_h$ and choose $\epsilon$ sufficiently small such that
$$
2 \epsilon < \min \{ \, u_h(x) - v_h (x): x \in \Omega_h, u_h(x) > v_h(x) \, \}. 
$$
We have $q_h \geq u_h \geq v_h \geq w_h$ on $\Omega_h$. Moreover, using $u_h(x^1) = v_h(x^1)$,
$$
q_h(x^1)-w_h(x^1)= u_h(x^1) + \epsilon - (v_h(x^1) - \epsilon) = 2  \epsilon.
$$
In addition, for $x \neq x^1, x \in \Omega_h$, $q_h(x)-w_h(x)=u_h(x)-v_h(x) \geq 2  \epsilon = q_h(x^1)-w_h(x^1)$. 
Therefore,
$q_h - w_h$ has a minimum at $x^1$ and are both discrete convex with asymptotic cone $K$. From Part 1, we conclude that $ \omega_V(R, q_h,\{\, x^1\, \})  \geq  \omega_V(R, w_h,\{\, x^1\, \}) $. This contradicts \eqref{unicity2-con01}. We conclude that $u_h=v_h$ at more than one point.

{\bf Part 3} Next, we consider the case that $x^1 \in \partial \Omega_h$.  By the assumption that $V(x) = V_{max}(x)$ for all $x \in \partial \Omega_h$, $f>0$ on $\Omega_h$ and Theorem \ref{from-out}, we have $\Gamma_2(v_h)(x^1)=v_h(x^1)=u_h(x^1)=\Gamma_1(v_h)(x^1)$. As in the proof of Theorem \ref{unicity}, we consider the convex decomposition  $\cup_{i=1}^n P_i = \Conv(\mathcal{N}_h^2)$ associated with $\Gamma_2(v_h)$ with $x^1 \in P_1$. Here $P_1$ is the convex hull of $x^1$ and the points $x^1\pm h e, e \in \hat{V}$, where $\hat{V} \subset V(x^1)$ denotes the set of normals to the facets of $\partial_{V_{max}} v_h(x^1)$. As in the proof of Theorem \ref{unicity} we get 
$(u_h-v_h)(x^1\pm h e)=0  \ \forall e \in \hat{V}$. 

By assumption, for $x \in \Omega_h$ and $x \neq x^1$, $u_h(x)>v_h(x)$. Therefore for $ e \in \hat{V}$, $x+h e \notin \Omega_h$. By assuming that $h$ is sufficiently small or the domain $\Omega$ is large relative to the size of $e \in \hat{V}$, we conclude that all points $x+ h e$ for $e \in \hat{V}$ are in the same closed half-space. But the set of normals to the facets of a polygon cannot all lie in the same half-space, as a consequence of \cite[Proposition 1]{klain2004minkowski}. That is, if $a_e$ denotes the volume of the facet of $\partial_{V_{max}} v_h(x^1)$ with normal $e$, $\sum_{e \in \hat{V}} a_e e =0$. If for a unit vector $w$, we have $w \cdot e \geq 0$ for all 
$e \in \hat{V}$, then $\sum_{e \in \hat{V}} a_e w \cdot e =0$ and hence $w \cdot e=0$ for all $e \in \hat{V}$. Since the set of normals to the facets of the non degenerate polygon $\partial_{V_{max}} v_h(x^1)$ span $\R^d$, we obtain $w=0$. 
Contradiction.  We conclude that $u_h=v_h$ at more than one point.
\Qed
\end{proof}





The proof of existence of a solution to \eqref{m2d3} in the case $V=V_{max}$ is 
identical to 
the case of convex polygonal approximations \cite[Theorem 17.2]{Bakelman1994}. 

\begin{lemma} \label{cont-integral} 
Let $v_h^k$ be a sequence of discrete convex mesh functions with asymptotic cone $K$ such that $v_h^k(x) \to v_h(x)$ for all $x$ in $\Omega_h$. Then $v_h$ is discrete convex with asymptotic cone $K$ and for all $x \in \Omega_h$, 
$\omega_V(R,v_h^k,\{\, x\, \}) \to \omega_V(R,v_h,\{\, x\, \})$.

\end{lemma}

\begin{proof}
Let $x \in \mathbb{Z}^d_h \setminus \Omega_h$ and assume that 
$v_h^k(x)=v_h(s_k) + (x-s_k) \cdot a^*_k$ for some $s_k \in \partial \Omega_h$ and a vertex $a^*_k$ of $Y$. Here $\max_{j=1,\ldots,N}  (x-s_k) \cdot a^*_j= (x-s_k) \cdot a^*_k$. Since $\Omega_h$ is finite, up to a subsequence, we obtain for $k$ sufficiently large, $s_k=s \in \partial \Omega_h$ and $a^*_k=a^*$ for a vertex $a^*$ of $Y$. We thus have $v_h(x)=v_h(s) + (x-s) \cdot a^*$ with $\max_{j=1,\ldots,N}  (x-s) \cdot a^*_j= (x-s) \cdot a^*$. Hence 
$v_h^k(x) \to v_h(x)$ for all $x \in \mathbb{Z}^d_h$ and 
$v_h$ has asymptotic cone $K$.

By a similar argument, if for $x\in \Omega_h$ and $e \in V$, $\Delta_{he} v_h^k(x)\geq 0$, then $\Delta_{he} v_h(x)\geq 0$.

We now prove that for all $x \in \Omega_h$, 
$\omega_a(R,v_h^k,\{\, x\, \}) \to \omega_a(R,v_h,\{\, x\, \})$. 
We have for $x \in \Omega_h$
\begin{multline*}
\int_{\partial_V v_h^k(x)} R(p) dp - \int_{\partial_V v_h(x)} R(p) dp  \\ = \int_{\partial_V v_h^k(x) \setminus \partial_V v_h(x) } R(p) dp -  \int_{\partial_V v_h(x) \setminus \partial_V v_h^k(x) } R(p) dp.
\end{multline*}

If $p \in \partial_V v_h^k(x) \setminus \partial_V v_h(x)$ there exists $e \in V$ such that 
$$
v_h(x+he)  - v_h(x)  < p \cdot (h e) \leq v_h^k(x+he) - v_h^k(x). 
$$
Put $\alpha = v_h(x+he)  - v_h(x)$ and $\beta = v_h^k(x+he) - v_h^k(x)$. We have
$|p \cdot (h e) -(\alpha+\beta)/2 | \leq \beta-\alpha$. As $k \to \infty$, $\beta \to \alpha$. 
Therefore, given $\delta>0$, there exists $k_0$ such that for all $k \geq k_0$, $|p \cdot (h e) -\alpha |\leq \delta$, where we used $\alpha=(\alpha+\beta)/2-(\beta-\alpha)/2$. This also gives $|p \cdot (-h e) - (-\alpha) |\leq \delta$. 


Recall that $\partial_V v_h^k(x) \subset Y$ is bounded. We conclude that there is a constant $C$ which depends on $e$ and $h$ such that
$| \partial_V v_h^k(x) \setminus \partial_V v_h(x) | \leq C \delta$. Since $R$ is integrable, there exists $\delta>0$ such that if $|S| < C \delta$, we have
$\int_S R(p) dp < \epsilon/2$. It follows that
$\big| \int_{\partial_V v_h^k(x) \setminus \partial_V v_h(x) } R(p) dp \big| < \epsilon/2$ for $k \geq k_0$.

With a similar argument, we have $\big| \int_{\partial_V v_h(x) \setminus \partial_V v_h^k(x) } R(p) dp \big| < \epsilon/2$ for $k \geq k_1$ for an integer $k_1$. This proves that for $k \geq \max\{ \, k_0, k_1 \, \}$,
$|\omega_V(R,v_h^k,\{\, x\, \}) - \omega_V(R,v_h,\{\, x\, \})| < \epsilon$ and completes the proof.  \Qed
\end{proof}

The last statement of the above lemma can also be proven from the continuity of the mapping $v_h \mapsto \int_{\partial_V v_h(x)} R(p) dp$, c.f. for example \cite[Proposition 2.3]{kitagawa2016newton}. 

\begin{definition} \cite[Section 2.2]{Yau2013} \label{def-conv-subd}
A convex subdivision $\mathcal{T}$ of a convex polyhedron $P$ is a subdivision of $P$ into convex polyhedra $K$, also called cells, such that
\begin{itemize}
\item $\cup_{K \in \mathcal{T}} K = P$
\item  if $K$ and $L$ are both in $\mathcal{T}$, then so is their intersection
\item if $K \in \mathcal{T}$ and $L \subset K$, then $L \in \mathcal{T}$ if and only if $L$ is a face of $K$.
\end{itemize}
\end{definition}
Associated to the piecewise linear convex function $
u(x) = \max\{ x \cdot p_i  + h_i: i=1,\ldots, M\},
$
where $p_i \in \R^d$, $h_i \in \R$ for all $i$, is a convex subdivision of $\R^d$ whose top dimensional cells are given by
$$
W_i = \{ \, x \in \R^d, x \cdot p_i  + h_i \geq x \cdot p_j  + h_j, j=1,\ldots, M
\, \},
$$
for $i=1,\ldots,M$.

\begin{remark} The proof of existence of a solution in the case $V=V_{max}$, given below, uses the convex subdivision of a piecewise linear convex function. For $V$ not necessary equal to $V_{max}$, the proof of convergence of a damped Newton's method for solving \eqref{m2d3} given in \cite{Awanou-damped} also gives existence of a solution to  \eqref{m2d3}. 
Therein, a subsequence of the damped Newton's iterations is shown to converge to a solution. If the problem is known to have a unique solution, then the whole sequence converges to the unique solution.  \end{remark}

\begin{theorem} \label{existence}
There exists a solution to \eqref{m2d3} for $f>0$ on $\Omega$ and $V=V_{max}$.
\end{theorem}

\begin{proof}
Let $\Omega_h=\cup_{i=1}^M \{ \, x^i \, \}$ and $\mu_i=\int_{E_{x^i}} \tilde{f}(x) dx, i=1,\ldots,M$.
Let $\mathcal{A}$ denote the set of discrete convex mesh functions on $\Omega_h$ with asymptotic cone $K$ 
such that for $v_h \in \mathcal{A}$, $v_h(x^1)=\alpha$ for $\alpha \in \R$ and $0 \leq \omega_a(R,v_h,\{\, x^i\, \})\leq \mu_i, i=2,\ldots,M$ with $\omega_a(R,v_h,\{\, x^1\, \})=\int_{Y} R(p) dp - \sum_{i=2}^M \omega_a(R,v_h,\{\, x^i\, \})$.

The set $\mathcal{A}$ is not empty since $v_h$ given by the restriction to $\Omega_h$ of $k_{(x^1,\alpha)}(x):=\alpha+\max_{j=1,\ldots,N} a^*_j \cdot (x-x^1)$ is in $\mathcal{A}$. Note that $k_{(x^1,\alpha)}$ is a piecewise linear convex function with only one vertex $(x^1,\alpha)$, c.f. section \ref{revisited}, and $\partial k_{(x^1,\alpha)} (x^1) = \partial k_{(x^1,\alpha)} (\R^d) = Y$. We then observe that $\Gamma_2(v_h)=k_{(x^1,\alpha)}$ \cite[Theorem 3]{awanou2019uweakcvg} and by Theorem \ref{from-out}, up to a set of measure 0, $\partial_{V_{max}} v_h (x^1) = \partial k_{(x^1,\alpha)} (x^1)  $.
Next, we consider the mapping
$
L: \R^M \to \mathcal{A}
$
defined by $L(\zeta)=v_h$ with $v_h$ defined by $v_h(x^i)=\zeta_i, i=1,\ldots,M$ and $\zeta=(\zeta_i)_{i=1,\ldots,M}$. The mapping $L$ is a bijection and we put $\mathbb{A}=L^{-1}\mathcal{A}$. 

We claim that $\mathbb{A}$ is a compact subset of $\R^M$. 
Let $\zeta^k \in \mathbb{A}, k \geq 1$ such that 
$\zeta^k \to \zeta$ and put $v_h^k=L(\zeta^k)$. By assumption, $\zeta^k_1=\alpha$ for all $k$. Thus $\zeta_1=\alpha$. It follows from Lemma \ref{cont-integral} that the set $\mathbb{A}$ is closed. By Lemma \ref{pre-lip-lem} and \eqref{Lipschitz}, for all $\zeta \in \mathbb{A}$ and $v_h=L(\zeta)$ we have $|v_h(x_i)| \leq C, i=1,\ldots,d$ and $C$ is independent of $i$. Thus $\mathbb{A}$ is bounded. We conclude that $\mathbb{A}$ is a compact subset of $\R^d$.

Define $\mathcal{F}: \R^M \to \R$ by $\mathcal{F}(\zeta) = \sum_{i=1}^M \zeta_i$. Since $\mathbb{A}$ is compact, $\mathcal{F}$ has a minimum $f_0$ at some $\zeta^0 \in \mathbb{A}$. Put $L(\zeta^0) = v_h^0$. We show that $v_h^0$ solves \eqref{m2d3}. 

Assume that $v_h^0$ does not solve \eqref{m2d3}. Since $\omega_a(R,v_h^0,\{\, x^i\, \})\leq \mu_i, i=2,\ldots,M$ we must have for some $l \in \{ \, 2\ldots, M\, \}$, $\omega_a(R,v_h^0,\{\, x^l\, \}) < \mu_l$. 
Define $\hat{v}_h$ by
$$
\hat{v}_h(x^i) = v_h^0(x^i), i \neq l \text{ and } \hat{v}_h(x^l) = v_h^0(x^l) - \epsilon,
$$
for $\epsilon>0$. The values of $\hat{v}_h$ on $\mathbb{Z}^d_h \setminus \Omega_h$ are given by \eqref{extension}. 

We have $\mathcal{F}(\hat{v}_h)= f_0 - \epsilon$. We show that for $\epsilon$ sufficiently small $\hat{v}_h \in \mathbb{A}$ and hence this yields a contradiction and concludes the proof. By construction $\hat{v}_h(x^1)=\alpha$ and 
by Lemma \ref{perturbation-lem}, $\hat{v}_h$  is discrete convex with asymptotic cone $K$ for $\epsilon \leq \epsilon_0$ and $\epsilon_0>0$.

For $i \neq l$ and $i \geq 2$ we have
$\omega_a(R,\hat{v}_h,\{\, x^i\, \}) \leq \omega_a(R, v_h^0,\{\, x^i\, \}) \leq \mu_i$. Arguing as in Lemma \ref{perturbation-lem}  we have $\omega_a(R,\hat{v}_h,\{\, x^l\, \}) \geq \omega_a(R, v_h^0 ,\{\, x^l\, \})$ and using Lemma \ref{cont-integral}, for $\epsilon$ sufficiently small we obtain $\omega_a(R, v_h^0 ,\{\, x^l\, \}) \leq \omega_a(R,\hat{v}_h,\{\, x^l\, \}) <  \mu_l$. 

Finally, by Lemma \ref{inc-sub-lem}, $\partial_V v_h^0(\Omega_h) \subset Y$ and 
$\sum_{i=1}^M \omega_a(R,v_h^0,\{\, x^i\, \})=\int_{Y} R(p) dp$ by assumption. Therefore $\partial_V v_h^0(\Omega_h) = Y$ since $\partial_V v_h^0(\Omega_h)$ is a union of polygons. Also, by Lemma \ref{inc-sub-lem}, $\partial_V \hat{v}_h(\Omega_h) \subset Y$. We claim that $\partial_V \hat{v}_h(\Omega_h) = Y$. 

Let $p \in Y$ and assume that $p \in \partial_V v_h^0(x), x \in \Omega_h$. If $x=x^l$, then
$p \in \partial_V v_h^0(x) \subset \partial_V \hat{v}_h(x) \subset \partial_V \hat{v}_h(\Omega_h)$. If $x\neq x^l$, we have either $p \in  \partial_V \hat{v}_h(x) \subset \partial_V \hat{v}_h(\Omega_h)$ or 
$p \notin \partial_V \hat{v}_h(x)$. Assume that $p \notin \partial_V \hat{v}_h(x)$. We show that $p \in  \partial_V \hat{v}_h(x^l)$. Since $p \notin \partial_V \hat{v}_h(x)$, $\exists$  $\hat{e} \in V$ such that 
$$
\hat{v}_h(x+h \hat{e}) - \hat{v}_h(x) < p \cdot (h \hat{e}).
$$
We must have $x+h \hat{e}=x^l$. Otherwise, as $x\neq x^l$, we would have $\hat{v}_h(x+h \hat{e})=v_h^0(x+h \hat{e})$ and $\hat{v}_h(x) = v_h^0(x)$, thus a contradiction with $p \in \partial_V v_h^0(x)$.
Thus 
\begin{equation} \label{p-x-l}
p \cdot (-h \hat{e})<\hat{v}_h(x^l-h \hat{e})-\hat{v}_h(x^l).
\end{equation}
Since $p \in  \partial_V v_h^0(x)$, we have for all $e \in V$,  $p\cdot(h e + h  \hat{e}) \leq 
v_h^0(x+h e + h  \hat{e}) - v_h^0(x)=v_h^0(x^l + h  e) - v_h^0(x^l-h \hat{e})$. This gives by \eqref{p-x-l}
$p \cdot (h e) \leq \hat{v}_h(x^l+h e)-\hat{v}_h(x^l)$ where we also used $x^l+h e \neq x^l$. We conclude that $p \in  \partial_V \hat{v}_h(x^l)$ and $Y \subset \partial_V \hat{v}_h(\Omega_h)$. As a consequence $\omega_a(R,\hat{v}_h,\{\, x^1\, \})=\int_{Y} R(p) dp - \sum_{i=2}^M \omega_a(R,\hat{v}_h,\{\, x^i\, \})$. Here, we use the observation that for $x \in \Omega_h$, $\omega_a(R,\hat{v}_h,\{\, x \, \})=\omega_a(R,\Gamma_2(\hat{v}_h),\{\, x \, \})$ and for $x, y \in \Omega_h$ with $x\neq y$, $\partial \Gamma_2(\hat{v}_h)(x) \cap \partial \Gamma_2(\hat{v}_h)(y)$
is a set of measure 0. This concludes the proof that $\hat{v}_h \in \mathbb{A}$. \Qed
\end{proof}


\section{Asymptotic cone of convex sets} \label{asym}

In this section we first review the geometric notion of asymptotic cone and give an analytical formula, with a geometric interpretation, for the extension to $\R^d$ of a convex function on a polygon $\Omega$,  in such a way that it has a prescribed behavior at infinity, i.e. a prescribed asymptotic cone. The prescribed asymptotic cone will be constructed from a polygon $Y$ which approximates the domain $\Omega^*$ appearing in the second boundary condition. We will use the term polygon to also refer to a polygonal domain. 
Figure \ref{cone-fig} taken from \cite{Awanou-ams} illustrates the results discussed in this section. 
Using the notion of asymptotic cone we reformulate the second boundary condition. This allows to prove more results about convex extensions.

\subsection{Asymptotic cones} \label{epigraph}

We will use the notation $\R^{d+1}$ for a set of points and for a vector space over $\R$ endowed with the operations of scalar multiplication and addition. This makes $\R^{d+1}$ a Euclidean space with associated vector space $\R^{d+1}$. When emphasizing the geometric nature of some of the notions discussed below, we will use capital letters for points in the Euclidean space $\R^{d+1}$ and lower case letters for vectors. Thus we have a mapping $\R^{d+1} \times \R^{d+1} \to \R^{d+1}$ which maps $(P,e)$ to $P+e$. We will use the notation $O$ for the origin in $\R^{d+1}$. If $Q=P+e$ we write $e= \overrightarrow{PQ}$. 

Let $L$ be a line in $\R^{d+1}$, $A$ be some point of $L$, and $e \in \R^{d+1}$ be a direction vector of $L$. The sets
$$
L_{A,e}^+ = \{ \, X \in L,  \overrightarrow{A X} = \lambda e,  \lambda \geq 0\, \} \text{ and }
L_{A,e}^- = \{ \, X \in L,  \overrightarrow{A X} = \lambda e,  \lambda \leq 0\, \},
$$
are the rays of $L$ with vertex $A$. 

The Minkowski sum of $S \subset \R^{d+1}$ and $T \subset \R^{d+1}$ is defined to be $S+T = \{ \, s+t, s \in S,  t \in T\, \} $. 

Let $M \subset \R^{d+1}$ be a set. 
We denote by $K_A(M)$ the set of points in $M$ lying on the rays starting from the point $A \in M$. If there are no such rays, we set $K_A(M)=\{ \, A \, \}$. We say that a set $K_1$ is a parallel translation of $K_2$ if $K_2=e+K_1$ for some direction $e \in \R^{d+1}$. It is known that when $M$ is convex, $K_A(M)$ is convex and independent (up to a parallel translation) of the point $A \in M$ and is called asymptotic cone of the convex set $M$ 
\cite[Theorem 1.8 and Corollary 1]{Bakelman1994}. For a convex bounded set $M$, we have  $K_A(M)=\{ \, A \, \}$ for all $A \in M$ and $\{ \, A \, \}$ is a parallel translation of $\{ \, B \, \}$ for all $A, B \in M$. 

\begin{definition} \label{def-asc}
The asymptotic cone $K_A(M)$ of a convex set $M$ is defined for $A \in M$ as
\begin{multline*}
\{ \, B: B \in L_{A,e}^+ \text{ for } e \in \R^d \text{ and } L_{A,e}^+ \subset M  \, \}= 
\{ \, B: B = A + \mu e, \mu \geq 0,  e \in \R^{d+1}, \\ A + \lambda e \in M \ \forall   \lambda \geq 0 \, \}.
\end{multline*}
It is unique up to parallel translation, and is in that sense independent of the point $A$, 
i.e. $K_B(M) = K_A(M)+\overrightarrow{AB}$.
\end{definition}

The reason of the term ''cone'' in the name asymptotic cone will be clear from section \ref{comp-section} below where we give a formal definition of cone. 
Moreover, we will be interested in a specific example of cone which we will refer to as polyhedral angle (formal definitions are in section \ref{comp-section}). An intuitive notion of cones and polyhedral angles as illustrated in Figure 
\ref{pointed_cone-fig} is enough for this paper. 

 \begin{figure}[tbp]
\begin{center}
\includegraphics[angle=0, height=4.5cm]{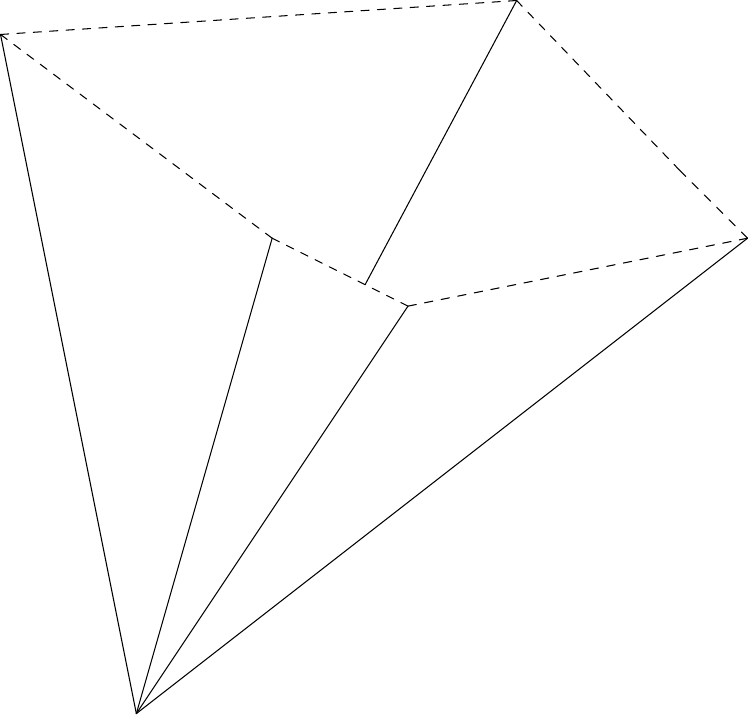}
\includegraphics[angle=0, height=4.5cm]{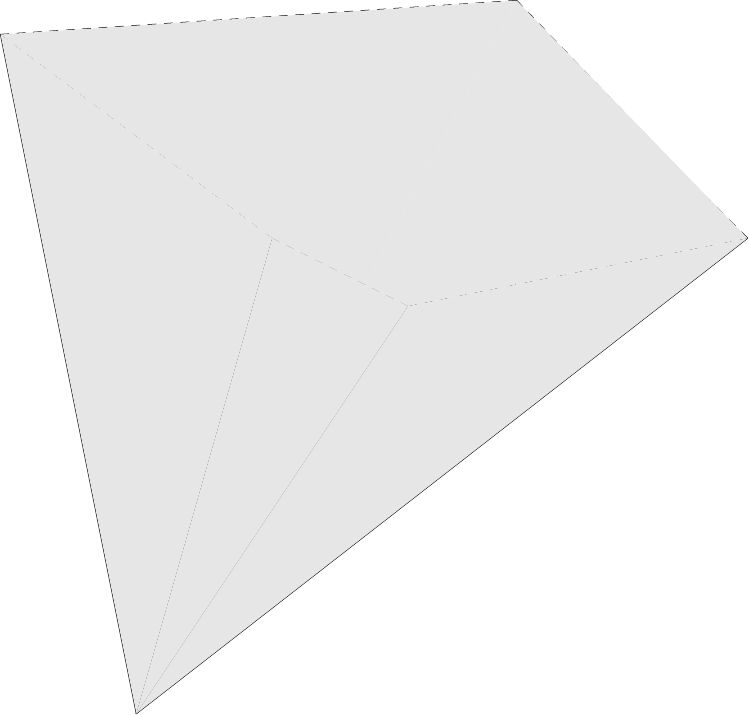}
\end{center}
\caption{ A polyhedral angle in $\R^3$. The dashed polygon is a virtual cut of the unbounded set. To emphasize that a polyhedral angle has non zero Lebesgue measure, a filled version is shown.} 
\label{pointed_cone-fig}
\end{figure}


We denote by $\Conv (D)$ the convex hull of the set $D \subset \R^d$, i.e. the smallest convex set containing $D$. It is known that $\Conv (D)$ is the set of all convex combinations of elements of $D$, i.e. the set of elements $\sum_{i=1}^n \lambda_i x_i$, $n \in \mathbb{N}$, $x_i \in D$, $\lambda_i \in [0,1]$ and $\sum_{i=1}^n \lambda_i =1$.

Let $Y \subset \R^d$ be a convex polygon with vertices $a_1^*, a_2^*, \ldots, a_{N^*}^* \in \R^d$. We have 
$Y=\Conv\{ \, a_1^*, a_2^*, \ldots, a_{N^*}^* \, \}$. In this paper, we use the mention $*$ for objects related or which will be related to $\Omega^*$. As we will associate below a cone $K$ to $Y$, we avoid the  $*$ notation for $Y$ to avoid confusion with the dual of a cone. 
We assume that $Y$ is non degenerate in the sense that it has non zero Lebesgue measure. 
Define for $(p,\mu) \in \Omega \times \R$ 
the  function on $\R^d$
\begin{equation} \label{ex-cc}
k_{(p,\mu)}(x) = \max_{1 \leq i \leq N^*} (x-p) \cdot a_i^* +\mu.
\end{equation}
Recall that the epigraph of $k_{(p,\mu)}$ is the set
$$
K_{(p,\mu)} = \{ \, (x,w) \in \R^d \times \R, w \geq k_{(p,\mu)}(x) \, \}.
$$
We will refer to sets of  the type $K_{(p,\mu)}$ as polyhedral angles, and refer to Figures \ref{pointed_cone-fig} and \ref{translation} for illustrations. In other words, a polyhedral angle is the epigraph of a function of type $k_{(p,\mu)}$ given in \eqref{ex-cc}. In section \ref{comp-section} we give a more general definition of polyhedral angle. We only need the class of polyhedral angles introduced above in this paper. 

It is crucial for the reader to see the connection between the graph of a function $k_{(p,\mu)}$ for given $(p,\mu)$ and the polyhedral angles depicted in Figures \ref{pointed_cone-fig} and \ref{translation}. For another example, the function defined on $\R$ by $w=|x|$, i.e., $w=\max \{ \, -x, x \, \}$ is a function of the form $k_{(p,\mu)}$. Its epigraph is a polyhedral angle. 

To the polygon $Y$ we associate the polyhedral angle 
$$
K\equiv K_{(0,0)},
$$ 
which depends only on the vertices of $Y$. In section \ref{cvg}, we will approximate the closure of the bounded convex domain $\Omega^*$ by polygons $Y \subset \tir{\Omega^*}$. The polyhedral angle $K$ associated with $Y$ is an example of a more general construction, which we now describe.

For each $p \in \tir{\Omega^*}$ one associates the half-space $Q(p)=\{ \, (x,z) \in \R^d \times \R, z \geq p \cdot x \, \}$. The convex set $K_{\Omega^*}$ is defined as the intersection of the half-spaces $Q(p), p \in \tir{\Omega^*}$, i.e.
\begin{equation} \label{cone-def}
K_{\Omega^*} := \cap_{p \in  \tir{\Omega^*} } \, Q(p).
\end{equation} 
Recall that the support function of the closed convex set $\Omega^*$ is defined for $x \in \R^d$ by
\begin{equation} \label{cone-def-cvx}
k_{\Omega^*}(x) := \sup_{p \in  \tir{\Omega^*} } \, p \cdot x.
\end{equation} 
The convex set $K_{\Omega^*}$ is the epigraph of $k_{\Omega^*}$ and the latter is a supremum of affine functions ($x \mapsto p \cdot x$), the gradients of which are in $\tir{\Omega^*}$. A slight abuse of notation is made in the notations $K_{\Omega^*}$ and $k_{\Omega^*}$ for convenience, as previously, a point $(p,\mu)$ was used as a subscript for $K$ and $k$.

In the case $\tir{\Omega^*}=Y$ is a non degenerate convex polygon with vertices $a_i^*, i=1,\ldots,N^*$, although the corresponding convex set $K_{\Omega^*}$ is by definition the intersection of an infinite number of half-spaces, i.e. 
$\cap_{p \in Y } \, Q(p)$, we claim that if $\tir{\Omega^*}=Y$, we have $K_{\Omega^*} = \cap_{ i=1 }^{N^*} \, Q(a_i^*)$.

Indeed, $\cap_{p \in Y } \, Q(p) \subset \cap_{ i=1 }^{N^*} \, Q(a_i^*)$. To prove the reverse inclusion, note that if $p \in Y$, $p=\sum_{i=1}^{N^*} \lambda_i a_i^*, \sum_{i=1}^{N^*} \lambda_i =1, 0 \leq \lambda_i \leq 1$. Let $(x,z) \in  \cap_{ i=1 }^{N^*} \, Q(a_i^*)$. We have $z \geq a_i^* \cdot x$ for all $i$ and thus $z \geq p \cdot x$, i.e. $(x,z) \in \cap_{p \in Y } \, Q(p)$. 

Thus $K_{\Omega^*}$ for $\tir{\Omega^*}=Y$ is the polyhedral angle $K$ introduced above, i.e. 
$K_{Y} = \cap_{ i=1 }^{N^*} \, Q(a_i^*)=K$. In this case, $k_{\Omega^*}(x)=k_{(0,0)}(x)=\max_{i=1,\ldots,N^*}   (x \cdot a_i^*)$.

The result given in the following lemma is illustrated in Figure \ref{translation}.

 \begin{figure}[tbp]
\begin{center}
\includegraphics[angle=0, height=4.5cm]{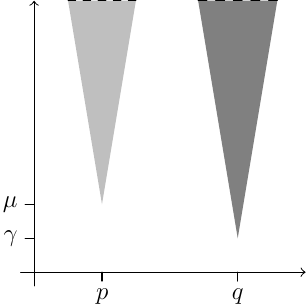}
\end{center}
\caption{ Epigraph of $k_{(p,\mu)}$ with a parallel translation as epigraph of $k_{(q,\gamma)}$. The dotted lines at the top of the figure represent virtual cuts of the unbounded epigraphs. As $\gamma < \mu$ in the figure, more of the unbounded epigraph is shown.} 
\label{translation}
\end{figure}

\begin{lemma} \label{ex-as-cone}
The epigraph of $k_{(p,\mu)}$ for $(p,\mu) \in \mathbb{R}^d \times \mathbb{R}$
is a convex set in $\R^{d+1}$ equal to its asymptotic cone. Furthermore, the epigraph of $k_{(p,\mu)}$ can be obtained from the one of  $k_{(q,\gamma)}$ for  $(q,\gamma) \in \mathbb{R}^d \times \mathbb{R}$ by a parallel translation.
\end{lemma}

\begin{proof}
As the maximum of convex functions, $k_{(p,\mu)}$ is a convex function and hence $K_{(p,\mu)}$ is a convex set. Next, we show that $K_{(p,\mu)} = K_{(q,\gamma)} - (q-p,\gamma-\mu)$. 

Let $(r,\eta) \in \mathbb{R}^d \times \mathbb{R}$. We show that $(r,\eta) \in K_{(p,\mu)}$ if and only if $(r,\eta) \in K_{(q,\gamma)} - (q-p,\gamma-\mu)$. 
Using the definitions and a few algebraic calculations, one shows that
$\eta \geq k_{(p,\mu)}(r)$ if and only if $\eta+(\gamma-\mu) \geq k_{(q,\gamma)}(r+(q-p))$. 
Note that $\eta \geq k_{(p,\mu)}(r)$ if and only if $(r,\eta) \in K_{(p,\mu)}$. Also, $\eta+(\gamma-\mu) \geq k_{(q,\gamma)}(r+(q-p))$ is equivalent to $(r+ q-p, \eta + \gamma-\mu) = (r,\eta) +(q-p,\gamma-\mu) 
\in K_{(q,\gamma)}$. Thus, $(r,\eta) \in K_{(p,\mu)}$ if and only if $(r,\eta) +(q-p,\gamma-\mu) 
\in K_{(q,\gamma)}$, i.e. $(r,\eta) 
\in K_{(q,\gamma)}- (q-p,\gamma-\mu) $. This proves the claim.

By definition of asymptotic cone of a convex set $M$, we have $K_A(M) \subset M$ for $A \in M$. Thus $K_{(p,\mu)}(K_{(p,\mu)}) \subset K_{(p,\mu)}$, i.e. the asymptotic cone of $K_{(p,\mu)}$ is contained in $K_{(p,\mu)}$. 

Let now $(q',\gamma') \in K_{(p,\mu)}$.  We find a direction $e \in \R^{d+1}$ such that the ray $L^+_{(p,\mu),e}$ with direction $e$ and vertex $(p,\mu)$ is contained in $K_{(p,\mu)}$ and $(q',\gamma')$ is on that ray.

Put $e=(q'-p,\gamma'-\mu)$. Then $(q',\gamma') = (p,\mu) + e$. So $(q',\gamma')  \in L^+_{(p,\mu),e}$. Since $(q',\gamma') \in K_{(p,\mu)}$ we have
$$
\gamma' \geq (q'-p) \cdot a_i^* + \mu, \quad \forall i=1,\ldots,N.
$$
It follows that
$$
\mu + \lambda (\gamma'-\mu) \geq \lambda (q'-p) \cdot a_i^* + \mu, \quad \forall i=1,\ldots,N.
$$
From the definition of $k_{(p,\mu)}$ we have $\mu + \lambda (\gamma'-\mu) \geq k_{(p,\mu)}(p+\lambda (q'-p))$ which proves that $(p,\mu) + \lambda e \in K_{(p,\mu)}$ for all $\lambda \geq 0$. \Qed
\end{proof}

Recall that, by Lemma \ref{ex-as-cone}, $K_{(p,\mu)}=K_{(0,0)}+(p,\mu)=K+(p,\mu)$. We recall the following equivalent characterization of the asymptotic cone \cite{auslender2006asymptotic}.

\begin{lemma} \label{as-eq-lem}
Let $M \subset \R^{d+1}$ be a closed convex set, $e \in \R^{d+1}$ and $A \in M$. The following two statements are equivalent

\begin{enumerate}
\item $L_{A,e}^+ \subset M$

\item $\exists \lambda_k \in \R, \lambda_k >0, \lambda_k \to \infty$ and $\exists A_k \in M$, $k \in \mathbb{N}$ such that $A_k/\lambda_k \to e$ as $k \to \infty$. 
\end{enumerate}

\end{lemma}

\begin{proof} Assume that $ L_{A,e}^+ \subset M$ and let $\lambda_k \to \infty$. Then $A_k=A+\lambda_k  e \in M$ and $A_k/\lambda_k  \to e$.

Conversely suppose $\lambda_k \to \infty$ and $A_k \in M$ is such that $A_k/\lambda_k \to e$. Put $d_k=(A_k-A)/\lambda_k$. Then
$A_k = A+ \lambda_k d_k \in M$ and $d_k \to e$. Let $\lambda>0$ and choose $k$ sufficiently large such that $\lambda \leq \lambda_k$. Since $M$ is convex
$$
A + \lambda d_k = \bigg(1-\frac{\lambda}{\lambda_k}\bigg) A + \frac{\lambda}{\lambda_k} A_k,
$$
is in $M$ and hence its limit $A+\lambda e$ is in $M$ as $M$ is closed. \Qed
\end{proof}

Recall the convex set $K_{\Omega^*}$, c.f. \eqref{cone-def}.

\begin{lemma} \label{conv-hull-union}
Let S be a closed and bounded convex set and let $M$ denote the convex hull of the union of $S$ and  $A+K_{\Omega^*}$ for $A \in S$.
Then the closure of $M$ is given by $S+K_{\Omega^*}$. 
\end{lemma}

\begin{proof} Let $x \in M$. There exist points $A_i \in S, i=1,\ldots,m$ and points $C_i \in K_{\Omega^*}, i=m+1,\ldots,n$ for integers $m$ and $n$ with scalars $\alpha_i, i=1,\ldots,n$ such that
$$
x=\sum_{i=1}^m \alpha_i A_i + \sum_{i=m+1}^n \alpha_i(A + C_i), \ \text{with} \ \sum_{i=1}^n \alpha_i=1, 0 \leq \alpha_i \leq 1. 
$$
Since $K_{\Omega^*}$ is convex and the origin $O \in K_{\Omega^*}$, $\big(\sum_{i=1}^m \alpha_i \big) O + \sum_{i=m+1}^n \alpha_i C_i \in K_{\Omega^*}$. On the other hand $\sum_{i=1}^m \alpha_i A_i + \sum_{i=m+1}^n \alpha_i A \in S$. Thus $M \subset S+K_{\Omega^*}$, and so $\tir{M} \subset \tir{S+K_{\Omega^*}}$.

Let now $x \in S+K_{\Omega^*}$, i.e. $x=s + z$ with $s \in S$ and $z\in K_{\Omega^*}$. 
Let $\epsilon >0$ and note that $z/\epsilon \in K_{\Omega^*}$. We consider the point
$$
A_{\epsilon} = A+z + (1-\epsilon)(s-A) = \epsilon (A + \frac{z}{\epsilon}) + (1-\epsilon) s.
$$
The point $A_{\epsilon}$ is a convex combination of a point in $K_{\Omega^*}+A$ and a point in $S$. Thus $A_{\epsilon} \in M$. As $\epsilon \to 0$, $A_{\epsilon} \to s +z = x$. This proves that $x \in \tir{M}$. 

We have $S+K_{\Omega^*} \subset \tir{M} \subset  \tir{S+K_{\Omega^*}}$. To conclude the proof, we show that $S+K_{\Omega^*}$ is closed. Since $S$ is a closed and bounded set and $K_{\Omega^*}$ is closed, $S+K_{\Omega^*}$ is closed. To prove this claim, let $x_l=s_l+a_l$ be a sequence in $S+K_{\Omega^*}$, $s_l \in S$ and $a_l \in K_{\Omega^*}$. We assume that $x_l$ converges to $x$. 
If necessary, by taking a subsequence, as $S$ is bounded and closed, we may assume that $s_l$ converges to $s$ in $S$. Then, $a_l=x_l-s_l$ converges as the difference of two convergent sequences to an element $a \in K_{\Omega^*}$ as $K_{\Omega^*}$ is closed. We have $a=x-s$ and hence $x=a+s \in S+K_{\Omega^*}$. We conclude that $S+K_{\Omega^*}$ is closed.  The proof is complete. 
\Qed
\end{proof}

We note that in the above lemma 
the closure of the convex hull of the union of $S$ and  $A+K_{\Omega^*}$ for $A \in S$ is independent of the choice of $A$.

We illustrate Lemma \ref{conv-hull-union} in Figure \ref{cone-fig}. But first, we rewrite the Minkowski sum of two sets as a union of sets. 

Let $S$ and $T$ be two subsets of $\R^{d+1}$. 
Then we have $S+T  =  \{ \, t+S, t \in T\, \} = \cup_{t \in T} t+S$. 
We say that the sum $S+T$ is obtained by sweeping the set $S$ over $T$,
\begin{equation} \label{min-union}
S+T = \cup_{t \in T} t+S.
\end{equation}
Clearly, if $r \in S+T$, $r=s+t$ for some $s \in S$ and $t \in T$. Thus $r \in t+S$. The reverse inclusion is also immediate.

We have $S+K_{\Omega^*} = \cup_{s \in S} \, (s+ K_{\Omega^*})$. Note that the sets $s+K_{\Omega^*}$ are parallel translates of each other. Thus Lemma \ref{conv-hull-union} says that the closure of the convex hull of the union of $S$ and  $A+K_{\Omega^*}$ for $A \in S$ is obtained by sweeping $K_{\Omega^*}$ over $S$. 

Recall Definition \ref{def-asc} of asymptotic cone of a convex set.

\begin{theorem} \label{as-poly}
Let S be a closed and bounded convex set and let $M$ denote the convex hull of the union of $S$ and  $A+K_{\Omega^*}$ for $A \in S$.
Then $K_A(\tir{M}) = A+K_{\Omega^*}$, i.e. the closure of $M$ has asymptotic cone $A+K_{\Omega^*}$. 
\end{theorem}

\begin{proof} 
By Lemma \ref{conv-hull-union}, $\tir{M}=S+K_{\Omega^*}$. Recall the notation $K_A(W)$ for $A \in W$ for the asymptotic cone of the convex set $W$. 
We prove that $K_{A} (S+K_{\Omega^*}) = A+K_{\Omega^*}$. We first note that if $S \subset T$ and $A \in S$, then $K_A(S) \subset K_A(T)$. Indeed if $B \in K_A(S)$, then there is a direction $e$ such that $B \in L_{A,e}^+ \subset S \subset T$. 

Since $A+ K_{\Omega^*} \subset S+K_{\Omega^*}$ we have
$A+ K_{\Omega^*} = K_A(A+ K_{\Omega^*}) \subset K_{A}(S+K_{\Omega^*})$. Let now $B \in K_{A}(S+K_{\Omega^*})$ 
and let $e$ such that $B=A +\mu e$ for some $\mu >0$ and $L_{A,e}^+ \subset S+K_{\Omega^*}$. We show that $L_{A,e}^+ \subset A +K_{\Omega^*}$. 

By Lemma \ref{as-eq-lem} there exists a sequence $\lambda_k \to \infty$ and sequences $s_k \in S$ and $b_k \in K_{\Omega^*}$ such that
$(s_k+b_k)/\lambda_k \to e$. But $S$ is compact and so we may assume that the sequence $s_k$ converges to $s \in S$. This implies that $s_k/\lambda_k \to 0$ and hence $b_k/\lambda_k \to e$. By Lemma \ref{as-eq-lem} again, $L_{O,e}^+ \subset K_{\Omega^*}$, where $O$ is the origin of $\R^{d+1}$. It follows that $L_{A,e}^+ \subset A + K_{\Omega^*}$.  \Qed \end{proof}


\subsection{Convex extensions} \label{subdifferential}

Let us consider 
a convex function $u_0 \in C(\tir{\Omega})$ such that $\partial u_0(\Omega)=\Omega^*$. One can extend $u_0$ to $\R^d$  by
\begin{equation} \label{ext-asc}
\tilde{u}(x) = \inf \{ \, u_0(y) + \sup_{z \in  \tir{\Omega^*}} (x-y) \cdot z,  y \in \tir{\Omega}\, \}.
\end{equation}
The above formula 
was interpreted as a minimal convex extension in some sense or a special form of infimal convolution \cite[(15)]{benamou2017minimal}. Another extension formula used in \cite[p. 157]{ChouWang95} is given by
\begin{equation} \label{ext-sh}
\tir{u}(x) = \sup \{ \, u_0(y) +  (x-y) \cdot z, y \in \Omega, z \in \partial u_0(y)  \, \}.
\end{equation}
We consider below a generalization of \eqref{ext-asc}. 

We recall that a point $(x,\mu)$ is on the lower part of the boundary of a convex set $M \subset \R^{d+1}$ if $(x,\mu) \in M$ and $(x,\mu)-(0,\ldots,0,\lambda) \notin M$ for all $\lambda>0$. Recall also that given a domain $U \subset \R^d$, e.g. $U=\Omega$ or $U=\R^d$, and a function $v$ defined on $U$, the graph of $v$ is the subset of $\R^{d+1}$ given by
$$
\{ \, (x, v(x)): x \in U\, \}.
$$

Let $u$ be a piecewise linear convex function on $\Omega$ and $E\subset \Omega$ bounded. The graph $\mathcal{M}_u=\{ \, (x, u(x)),  x \in E\, \}$ of $u$ is the lower part of the boundary of a convex polygonal domain $S=\{ \, (x,\mu) \in \R^{d+1}, x \in E, u(x) \leq \mu \leq \mu_{max} \, \}$, where $\mu_{max}
=\max_{x \in E} u(x)$. We refer to the vertices of $S$ on $\mathcal{M}_u$ as the vertices of $u$. 

The projection $U \subset \R^d$ of a convex set $M \subset \R^{d+1}$ is the set $\{ \,  x: x \in \R^d, \exists \lambda \in \R, (x,\lambda) \in M\, \}$. We give an example of projection of a convex set in Figure \ref{cone-fig}.

\begin{definition} \label{def-cv-function}
A convex set $M \subset \R^{d+1}$ defines a function $v$ on its projection $U \subset \R^d$ if the graph of $v$ on $U$ is equal to the lower part of the boundary of $M$.
\end{definition}

As an example, the polyhedral angle $K_{(p,\mu)}$ defines the function $k_{(p,\mu)}$ on $\R^d$. We also say that  the polyhedral angle $K_{(p,\mu)}$ has boundary given by the graph of $k_{(p,\mu)}$. The convex set $K_{\Omega^*}$, c.f. \eqref{cone-def}, defines the convex function $k_{\Omega^*}$, c.f. \eqref{cone-def-cvx}, on $\R^d$. It is known that $\chi_{k_{\Omega^*}}(\R^d) = \tir{\Omega^*}$,  \cite[p. 22]{Oliker03}.

\begin{definition}

We say that a convex function $v$ on $\R^d$ has asymptotic cone $K_{\Omega^*}$ if its epigraph $M$ has asymptotic cone $A+K_{\Omega^*}$ for $A \in M$. 

\end{definition}

Recall that the asymptotic cone of a convex set $M$ is a particular convex set associated with $M$. It contains all half-lines starting at a point $A \in M$ and contained in $M$. When $M$ is the epigraph of a function $v$, the lines in the asymptotic cone $K_A(M)$ give the behavior of $v$ at infinity. 

\begin{lemma} \label{as-2-bvp}
Let $v$ be a convex function on $\R^d$ such that $\chi_v (\R^d)  = \tir{\Omega^*}$. Then $v$ has asymptotic cone $K_{\Omega^*}$. 
\end{lemma}

\begin{proof} 
A point $A \in \R^{d+1}$ is denoted $(x,z)$ for $x \in \R^d$ and $z \in \R$. Let $M$ denote the epigraph of $v$ and assume that $A_1=(a_1,u_1) \in \partial M$. 
Note that $M$ is unbounded and $\partial M$ is the lower part of the boundary of $M$. We first prove that $A_1+K_{\Omega^*} \subset K_{A_1}(M)$. Let $(x,w) \in A_1+ K_{\Omega^*}$ and put $e=(x,w)-(a_1,u_1)$. We show that for all $\lambda>0$, $A_1 + \lambda e \in M$. Assume by contradiction that this does not hold. Let $B$ be the point of intersection with $\partial M$ of the line through $A_1$ and with direction $e$. The half-line $L_{B,e}^+$ is then not contained in $M$. Choose $C \in L_{B,e}^+, C \neq B$ and put
\begin{align*}
B& =(x_B,z_B)=A_1 + \mu_1 e =(a_1,u_1) + \mu_1 (x-a_1,w-u_1) \\
C&=(x_C,z_C)=A_1 + \mu_2 e = (a_1,u_1) + \mu_2 (x-a_1,w-u_1),
\end{align*}
for $\mu_1\geq 0$, and $\mu_2 >0$. By construction $\mu_2-\mu_1>0$. Now let $p \in \chi_v(x_B)$. Since the plane $z=p\cdot(x-x_B)+z_B$ is a supporting hyperplane to $M$ at $B$, we can choose 
$C \notin M$, in addition to $C \in L_{B,e}^+, C \neq B$, such that
$z_C < p\cdot(x_C-x_B)+z_B$.
But $z_B = u_1+ \mu_1(w-u_1)$, $z_C= u_1+ \mu_2(w-u_1)$, $x_B= a_1+ \mu_1(x-a_1)$ and $x_C= a_1+ \mu_2(x-a_1)$.
As $z_C-z_B=(\mu_2-\mu_1) (w-u_1)$ and $x_C-x_B=(\mu_2-\mu_1) (x-a_1)$, we obtain $w-u_1< p \cdot(x-a_1)$. 

By assumption $\chi_v (\R^d)  = \tir{\Omega^*}$ and hence $p \in \tir{\Omega^*}$. Since $(x,w) \in A_1+K_{\Omega^*}$, $(x,w)-(a_1,u_1) \in K_{\Omega^*}$, and by the definition \eqref{cone-def} of $K_{\Omega^*}$, we have
$w \geq p \cdot (x-a_1)+u_1$. 
This contradicts $w-u_1< p \cdot(x-a_1)$. 

Next, we prove that $K_{A_1}(M) \subset A_1+K_{\Omega^*}$. Let $(x,w) \in K_{A_1}(M)$. The half-line $L_{A_1,e}^+$ is contained in $M$ with $e=(x,w)-(a_1,u_1)$. That is $(a_1,u_1) + \lambda (x-a_1,w-u_1) \in M$ for all $\lambda >0$. 

For each $p \in  \tir{\Omega^*}$ we can find $x_p \in \R^d$ such that $z=p\cdot(x-x_p) + v(x_p)$ is a supporting hyperplane to $M$ at $(x_p,v(x_p))$. Thus 
$$
u_1 + \lambda (w-u_1) \geq p \cdot  (a_1 + \lambda (x-a_1) - x_p) + v(x_p). 
$$
This gives $w-u_1 \geq p \cdot(x-a_1) + (p \cdot (a_1-x_p)  + v(x_p)-u_1)/\lambda$. Taking $\lambda \to \infty$ we obtain
$w \geq  p \cdot(x-a_1) + u_1$ for all $p \in \tir{\Omega^*}$. Thus $(x,w) \in A_1+ K_{\Omega^*}$ and the proof is complete. \Qed
\end{proof}

Let $S$ be a closed bounded convex set and let $\widetilde{S}$ denote its projection onto $\R^d$. Let $v$ denote the convex function defined by $S$ on $\widetilde{S}$. 
Put$$
D^* = \partial v((\widetilde{S})^\circ). 
$$
Assume that $(\widetilde{S})^\circ \neq \emptyset$ and $\tir{D^*} \subset \tir{\Omega^*}$. Recall that $K_{\Omega^*}$ is the epigraph of $ \sup_{p \in  \tir{\Omega^*} } \, p \cdot x$, c.f. \eqref{cone-def}.  
The set $K_{\Omega^*}+S=\cup_{s \in S} \, (s+K_{\Omega^*})$, which is convex by Lemma \ref{conv-hull-union}, also defines a convex function on $\R^d$ which extends $u$ to $\R^d$. This is proven in the next theorem where the assumption that $\tir{D^*} \subset \tir{\Omega^*}$ is used to prove that $u=v$ on $\widetilde{S}$. 
By sweeping $K_{\Omega^*}$ over $S$, $K_{\Omega^*}+S$ is the union of parallel translations of $K_{\Omega^*}$ and hence the values of the convex function $u$ on $\R^d$, i.e. the lower part of the boundary of $K_{\Omega^*}+S$, can be obtained from the lower part of the boundaries of some $s+K_{\Omega^*}, s \in S$. Note that the lower part of the boundary of $(y,\mu)+K_{\Omega^*}$ for $(y,\mu) \in \R^d \times \R$ is the epigraph of $\mu+\sup_{p \in  \tir{\Omega^*} } \, p \cdot (x-y)$. In the appendix we give a different proof of the next theorem using results on infimal convolution. 

\begin{theorem}  \label{formula-0}
Let $S$ be a closed bounded convex set which defines a convex function $v$ on the projection $\widetilde{S}$ of $S$
onto $\R^d$. Let $D^* = \partial v((\widetilde{S})^\circ)$ and assume that $(\widetilde{S})^\circ \neq \emptyset$ and $\tir{D^*} \subset \tir{\Omega^*}$. 
The convex set $K_{\Omega^*}+S$ defines a convex function $u$ on $\R^d$ which extends $v$ from $\widetilde{S}$ to $\R^d$ by 
\begin{equation} \label{partial-on-1}
u(x) = \inf_{y \in \widetilde{S}} v(y) +  \sup_{p \in  \tir{\Omega^*} } \, p \cdot (x-y), x \notin \widetilde{S}.
\end{equation}
\end{theorem}

\begin{proof}
Elements of $S$ take the form $(y,\mu), y \in \widetilde{S}$ and $\mu \in \R$. We have by Lemma \ref{conv-hull-union} and \eqref{min-union}
\begin{equation} \label{union-stuff}
S+K_{\Omega^*} = \cup_{(y,\mu) \in S} (y,\mu) + K_{\Omega^*}.
\end{equation}
We refer to Figure \ref{cone-fig} for an illustration of the above equality in the case $K_{\Omega^*}$ is polygonal, in which case  \eqref{partial-on-1} simplifies to  \eqref{partial-on-3} below. 
Equation \eqref{partial-on-3} is also illustrated in Figure \ref{cone-fig}. 
By Definition \ref{def-cv-function}, $S+K_{\Omega^*}$ defines a convex function $u$ on $\R^d$. This means that for $x \in \R^d$, $(x,u(x)) \in S+K_{\Omega^*}$ and 
if $(x,\mu) \in S+K_{\Omega^*}$, then $\mu \geq u(x)$, since by definition of lower part of $S+K_{\Omega^*}$, when $\lambda < u(x)$, $(x,\lambda) \notin S+K_{\Omega^*}$. Recall that $v$ denotes the convex function on $\widetilde{S}$ defined by the convex set $S$. We first show that $u=v$ on $\widetilde{S}$.

Since $0 \in K_{\Omega^*}$, $S \subset S+K_{\Omega^*}$ and recall that for $x \in \widetilde{S}, (x,v(x)) \in S  \subset S+ K_{\Omega^*}$. Thus $u(x) \leq v(x)$ for all $x \in \widetilde{S}$.

Assume that there exists $x \in  \widetilde{S}$ such that $u(x) < v(x)$. As $(x,v(x))$ is on the lower part of the boundary of $S$, $(x,u(x)) \not \in S$. But $(x,u(x)) \in S+K_{\Omega^*}$. By \eqref{union-stuff}, we can find $(y,\mu) \in S$ such that $(x,u(x)) \in (y,\mu) + K_{\Omega^*}$. Since $\tir{D^*} \subset \tir{\Omega^*}$, we have $K_{\Omega^*} \subset K_{D^*}$. Indeed, let $(x,w) \in K_{\Omega^*}$. We have $w \geq p \cdot x$ for all $p \in \Omega^*$. In particular,  $w \geq p \cdot x$ for all $p \in D^*$ and hence  $(x,w) \in K_{D^*}$.  This proves the claim. Therefore, $(x,u(x)) \in (y,\mu) + K_{D^*}$. 

Let $\tir{v}$ denote the convex extension of $v$ to $\R^d$ using supporting hyperplanes, i.e. the procedure described by \eqref{ext-sh}. 
By \cite[Lemma 4]{awanou2019uweakcvg}
$\chi_{\tir{v}}(S) =\chi_{\tir{v}}(\R^d) = \tir{D^*}$. By Lemma \ref{as-2-bvp}, $\tir{v}$ has asymptotic cone $K_{D^*}$. Thus, if $M$ denotes the epigraph of $\tir{v}$, for all $(y,\mu) \in  M$,  $y \in \R^d, \mu \in \R$,  we have $(y,\mu) + K_{D^*} \subset M$ and therefore 
for $x \in  \widetilde{S}$ and $(x,u(x)) \in (y,\mu)+ K_{D^*} \subset M$, we have $u(x) \geq \tir{v}(x)=v(x)$.




Next, we give an analytical proof of \eqref{partial-on-1}. Note that for $(y,\mu) \in S$, $(y,\mu) + K_{\Omega^*}$ defines the convex function $k_{(y,\mu)}(x)
=\sup_{p \in  \tir{\Omega^*} } \, p \cdot (x-y)+\mu$. 
Here, we make a slight abuse of notation, c.f. \eqref{ex-cc} where a max over a finite number of points is used for $k_{(y,\mu)}$.
Since $(y,\mu) + K_{\Omega^*} \subset S+K_{\Omega^*} $ for each $(y,\mu) \in S$, we have
$
u(x) \leq k_{(y,\mu)}(x) 
$ for $(y,\mu) \in S$. As $u(y)=v(y)$ for $y \in \widetilde{S}$, we have $(y,u(y)) \in S$ for $y \in \widetilde{S}$. We conclude that for $x \in \R^d$
\begin{equation} \label{ext-part0}
u(x) \leq k_{(y,u(y))}(x) =\sup_{p \in  \tir{\Omega^*} } \, p \cdot (x-y)+u(y),
\end{equation}
for $y \in \widetilde{S}$. We next show that for $x \not \in \widetilde{S}$, we can find $y \in \widetilde{S}$ such that $u(x) =  k_{(y,u(y))}(x)$. 

Since $(x,u(x)) \in S+K_{\Omega^*}$, by \eqref{union-stuff} we can choose $(y,\mu) \in S$ such that $(x,u(x)) \in (y,\mu) + K_{\Omega^*}$. Using the definition of lower part of the boundary of $ (y,\mu) + K_{\Omega^*}$,  $\mu \geq v(y)$ for $(y,\mu) \in S$  and $u=v$ on $\widetilde{S}$ we get
\begin{multline*}
u(x) \geq k_{(y,\mu)}(x) =  \sup_{p \in  \tir{\Omega^*} } \, p \cdot (x-y) +\mu \geq 
\sup_{p \in  \tir{\Omega^*} } \, p \cdot (x-y) + v(y) \\
  = \sup_{p \in  \tir{\Omega^*} } \, p \cdot (x-y) + u(y). 
\end{multline*}
We conclude from \eqref{ext-part0} that \eqref{partial-on-1} holds. \Qed
\end{proof}

Let $a_j^*, j=1,\ldots,N^*$ denote the vertices of a non degenerate convex polygon $Y \subset \R^d$. 
Thus, the interior of $Y$ is a convex domain in $\R^d$. 
Recall that 
$K$ denotes the polyhedral angle which is the epigraph of $\max_{1 \leq j \leq N^*} (x \cdot a_j^*)$. Recall also that when $\tir{\Omega^*}=Y$, $K_{\Omega^*}$ is the polyhedral angle $K$. In this case, \eqref{partial-on-1} becomes
\begin{equation} \label{partial-on-3}
u(x) = \inf_{y \in \widetilde{S}} \max_{1 \leq j \leq N^*} (x-y) \cdot a_j^* +u(y), x \notin \widetilde{S},
\end{equation}
where we used $u=v$ on $\widetilde{S}$.

Let $S$ be the polygon with vertices $(a_1, u_1), \ldots, (a_m, u_m)$ in $\R^{d+1}$. The projection $\widetilde{S}$ of $S$ onto $\R^d$ is the convex hull of $\{ \, a_1,\ldots,a_m \, \}$. Let us assume that $\{ \, a_1,\ldots,a_p \, \}$ for $p \leq m$
consist of the vertices which are on the boundary of $\widetilde{S}$.
It is assumed that $(\widetilde{S})^\circ \neq \emptyset$. 
The purpose of the next theorem is to show that 
the infimum in \eqref{partial-on-3} can be restricted to the boundary of $\widetilde{S}$. 
Such a formula is of interest for computational purposes, since the minimization in the extension formula of the next theorem is over a set much smaller than $\widetilde{S}$. This motivates the discrete extension formula \eqref{extension} where we consider the minimization over mesh points on $\partial \Omega_h$. As explained in the introduction, the discrete extension formula is needed for the discrete scheme. 


The points $a_j^*$ are not related to the points $a_i$, the same way the domain $\Omega^*$ is not related a priori to the domain $\Omega$.


\begin{figure}[tbp]
\begin{center}
\includegraphics[angle=0, height=4.5cm]{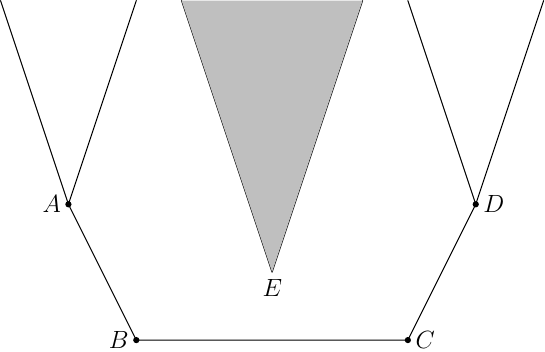}
\end{center}
\caption{ Let $S$ denote the polygon with vertices  $A(-1.5,1), B(-1,0), C(1,0)$ and $D(1.5,1)$. The polygon $S$ is the convex hull of its vertices.  The  polyhedral angle $K_{}$ associated to $\tir{\Omega^*}=[-3,3]$ is the intersection of the half-spaces $\{ \, (x_1,x_2) \in \R^2: x_2 \geq 3 x_1 \, \}$ and $\{ \, (x_1,x_2)  \in \R^2: x_2 \geq -3 x_1 \, \}$. Parallel translates $E+K$, $A+K$ and $D+K$ are shown. Put $M=\Conv(S \cup (E+K))$. To visualize $M$, note that $S \subset M$ and $E+K \subset M$. Then draw line segments connecting $A$ or $D$ to points on the boundary of $E+K$. Note that $\tir{M}$ is obtained by sweeping $K$ over $S$. 
The projection of the convex set $S$ on $\R$ is $[-1.5,1.5]$. The convex set $S$
defines a piecewise linear convex  function $u$ on $[-1.5,1.5]$, with graph the lower part of the boundary of $S$.  By 
Lemma \ref{conv-hull-union}, $\tir{M}=S+K$ is the convex hull of $S$ and $A+K$. By Theorem \ref{formula2}, 
the piecewise  linear convex  function on the real line defined by $\tir{M}$, i.e. the convex function with graph the lower part of the boundary of $\tir{M}$, is a convex extension of $u$ and is obtained by the extension formula. By Theorem \ref{as-poly} $\tir{M}$ 
has asymptotic cone $A+K_{}$. 
The ray with vertex $A$ and slope -3 and the ray with vertex $D$ and slope 3 are called extreme rays. The set $M$ is the convex hull of its vertices $A, B, C$ and $D$ and its extreme rays.
Image reproduced from \cite{Awanou-ams}. 
} 
\label{cone-fig}
\end{figure}

\begin{theorem} \label{formula}
Let $\widetilde{S}$ denote the projection on $\R^d$ of the lower part of the boundary of a polygon $S$. Let $K$ denote the polyhedral angle which is the epigraph of $\max_{1 \leq j \leq N^*} (x \cdot a_j^*)$, for given vectors $a_j^*, j=1,\ldots,N^*$, which are vertices of a non degenerate convex polygon $Y \subset \R^d$. Assume furthermore that 
$\tir{D^*} \subset Y$ where $D^*=\partial u((\widetilde{S})^\circ)$ 
and $u$ is the function defined by $S$ on $\widetilde{S}$. Assume also that 
$(\widetilde{S})^\circ \neq \emptyset$.  
The convex set $S+K$ defines a piecewise linear convex function $u$ which is given
for 
$x \notin \widetilde{S}$
by 
$$
u(x) =  \inf_{s \in \partial \widetilde{S}} \ \max_{1 \leq j \leq N^*} (x-s) \cdot a_j^*  +u(s). 
$$
\end{theorem}

\begin{proof}
The above formula is illustrated in Figure \ref{cone-fig} where the polyhedral angles (using the notation of the caption of Figure \ref{cone-fig}) $A+K$ and $D+K$ have portions of the lower part of their boundaries coincide with the graph of the extension. 

{\bf Part 1 } We show that $u$ is a piecewise linear convex function and characterize $\chi_u(x)$ for $x \notin \widetilde{S}$. Recall the representation \eqref{partial-on-3} which follows from Theorem \ref{formula-0} and $Y$ being polygonal. 
Since $S$ is the convex hull of a finite number of points, the function $u$ it defines on $\widetilde{S}$ is piecewise linear. Note that the polygon $S$ is an intersection of half-spaces, and the function defined on $\R^d$ by a half-space is a linear function. 

As in the proof of Theorem \ref{formula-0}, let for $y \in \widetilde{S}$, $k_{(y,u(y))}(x)
= \max_{1 \leq j \leq N^*} (x-y) \cdot a_j^*+u(y)$. By \cite[Chapter 4, Theorem 3]{ioffe2009theory}, for any $x \in \R^d$, $\chi_{k_{(y,u(y))}}(x)$ is the closed convex hull of a subset of $\{ \, a^*_1, \ldots, a^*_{N^*} \, \}$, i.e. $\chi_u(x)$ is a polygon with vertices in $\{ \, a^*_1, \ldots, a^*_{N^*} \, \}$. For $1\leq j \leq N^*$, $a_j^*$ is a vertex of $\chi_u(x)$ if and only if $u(x) = (x-y) \cdot a_j^*+u(y)$.
We now show that for all $x \notin \widetilde{S}$, there is $y \in \widetilde{S}$ such that $\chi_u(x) = \chi_{k_{(y,u(y))}}(x)$.

Let $x_0 \notin \widetilde{S}$ and $p \in \chi_u(x_0)$. We choose $x \in \R^d$ and have $u(x) \geq u(x_0) + p \cdot (x-x_0)$. Since $\widetilde{S}$ is compact, we can find $y_0 \in \widetilde{S}$ such that $u(x_0) = k_{(y_0,u(y_0))}(x_0)$. Recall that the graph of $u$ is the lower part of the boundary of $M=S+K$ and $M=S+K$ has asymptotic cone $K$ by Theorem \ref{as-poly}. This means that $(y_0,u(y_0)) +K \subset M$. Thus, for $x \in \R^d$, $(x, k_{(y_0,u(y_0))}(x))$ is in $(y_0,u(y_0)) +K \subset M$ and thus $ k_{(y_0,u(y_0))}(x) \geq u(x) 
\geq u(x_0) + p \cdot (x-x_0) =  k_{(y_0,u(y_0))}(x_0)  + p \cdot (x-x_0)$, i.e. $p \in  \chi_{k_{(y_0,u(y_0))}}(x_0)$.

Conversely, if $p \in  \chi_{k_{(y_0,u(y_0))}}(x_0)$, $p$ is in the convex hull of the vectors $a_{j_0}^*$ for which 
$u(x_0) = (x_0-y) \cdot a_{j_0}^*+u(y)$. It can be readily checked that $\chi_u(x_0)$ is convex. We show that any of the vectors $a_{j_0}^*$ is in $\chi_u(x_0)$ and thus $\chi_{k_{(y_0,u(y_0))}}(x_0) \subset \chi_u(x_0)$.

Let $x\in \R^d$. We have by \eqref{partial-on-3} $u(x) \geq \max_{1 \leq j \leq N^*} (x-y) \cdot a_j^* +u(y)
\geq (x-y) \cdot a_{j_0}^* +u(y)=(x-x_0) \cdot a_{j_0}^* + (x_0-y) \cdot a_{j_0}^* +u(y)=(x-x_0) \cdot a_{j_0}^*+u(y)$. This proves that $a_{j_0}^* \in  \chi_u(x_0)$ and completes the proof.

We conclude that $\chi_u(x)$ is a polygon with vertices in $\{ \, a^*_1, \ldots, a^*_{N^*} \, \}$ for any $x \notin \widetilde{S}$. This also shows with \eqref{partial-on-3} that $u$ is also piecewise linear on $\R^d \setminus \widetilde{S}$.

{\bf Part 2 } We show that the minimum in \eqref{partial-on-3} is actually on $\partial \widetilde{S}$. Let $x_0 \notin \widetilde{S}$. We can then find an index $k_0$ such that $a^*_{k_0} \in \chi_u(x_0)$. Define
$$
V_{0} = \{ \, x \in \R^d, a_{k_0}^* \in \chi_u(x) \, \}. 
$$
We first show that the non empty set $V_0$ is convex with $V_0 \cap \widetilde{S} \neq \emptyset$. Then we choose $s_1 \in V_0 \cap \widetilde{S}$. Next, we denote by $s_0$ the point of intersection  with $\partial \widetilde{S}$ of the line through $x_0$ and $s_1$. Finally, we show that $s_0$ is a point where the infimum in \eqref{partial-on-3} is realized when $x=x_0$.

Since $x_0 \in V_0$, $V_0 \neq \emptyset$. The convexity of $V_0$ follows immediately from the definitions. Let $x_1, x_2 \in V_0$ and $\lambda \in [0,1]$. For $y \in \R^d$, we have $u(y) \geq u(x_1) + (y-x_1)\cdot a_{k_0}^*$ and  $u(y) \geq u(x_2) + (y-x_2)\cdot a_{k_0}^*$. Thus $u(y) \geq \lambda  u(x_1) + (1-\lambda) u(x_2) + 
(y- \lambda x_1-(1-\lambda) x_2)\cdot a_{k_0}^*$, which shows by the convexity of $u$ that 
$ a_{k_0}^* \in \chi_u( \lambda x_1 +(1- \lambda) x_2)$. We conclude that $V_0$ is convex.

Next, we show that $V_0 \cap \widetilde{S} \neq \emptyset$. Using \eqref{partial-on-3}, since $\widetilde{S}$ is compact, we can find $s_1 \in \widetilde{S}$ such that $u(x_0) =  u(s_1) +  \max_{1 \leq j \leq N^*} (x_0-s_1) \cdot a_j^*$. Using $a_{k_0}^* \in \chi_u(x_0)$, we have for $y \in \R^d$, $u(y) \geq u(x_0) +  (y-x_0)\cdot a_{k_0}^* $. Thus
\begin{align*}
u(s_1) \geq u(x_0) +  (s_1-x_0)\cdot a_{k_0}^* = u(s_1) +  (s_1-x_0)\cdot a_{k_0}^*+  \max_{1 \leq j \leq N^*} (x_0-s_1) \cdot a_j^*.
\end{align*}
It follows that $\max_{1 \leq j \leq N^*} (x_0-s_1) \cdot a_j^* \leq (x_0-s_1)\cdot a_{k_0}^*$ and hence
$\max_{1 \leq j \leq N^*} (x_0-s_1) \cdot a_j^* = (x_0-s_1)\cdot a_{k_0}^*$. We conclude that
\begin{equation} \label{ux0}
u(x_0)=u(s_1)+(x_0-s_1)\cdot a_{k_0}^*. 
\end{equation} 
Since $a_{k_0}^* \in \chi_u(x_0)$, we have for $y \in \R^d$, $u(y) \geq u(x_0) +  (y-x_0)\cdot a_{k_0}^* =u(s_1)+(y-s_1)\cdot a_{k_0}^*$. This gives $a_{k_0}^* \in \chi_u(s_1)$ and hence $s_1 \in V_0 \cap \widetilde{S}$. 

Let now $s_0$ be the point on $\partial \widetilde{S}$ such that $x_0, s_0$ and $s_1$ are colinear. By the convexity of $V_0$ and since both $x_0$ and $s_1$ are in $V_0$, $s_0$ exists and is in $V_0$. 
Since $u$ is a piecewise linear convex function, it must be that on $V_0$, $u$ is a linear function, i.e.
for all $x \in V_0$, $u(x)=u(s_1)+(x-s_1)\cdot a_{k_0}^*$. In particular, $u(s_0)=u(s_1)+(s_0-s_1)\cdot a_{k_0}^*$ and by \eqref{ux0} we get
$u(x_0)=u(s_0) + (x_0-s_0)\cdot a_{k_0}^*$. Using $y=s_0$ in \eqref{partial-on-3}, we have
\begin{multline*}
u(x_0)=u(s_0) + (x_0-s_0)\cdot a_{k_0}^* \geq u(s_0) +  \max_{1 \leq j \leq N^*} (x_0-s_0) \cdot a_j^* \geq u(s_0) + (x_0-s_0)\cdot a_{k_0}^*
\\
=u(x_0)
\end{multline*}
and thus $u(x_0)=u(s_0) +  \max_{1 \leq j \leq N^*} (x_0-s_0) \cdot a_j^*$ for $s_0 \in \partial \widetilde{S}$. We conclude that for $x \notin \widetilde{S}$
$$
u(x) = \inf_{s \in \partial \widetilde{S}} \ \max_{1 \leq j \leq N^*} (x-s) \cdot a_j^*  +u(s).
$$
The proof is complete. \Qed \end{proof}

Theorems \ref{as-poly} and \ref{formula} provide the formula for the extension of a convex function, defined by the lower part of the convex hull of a finite set of points, to have a given asymptotic cone. The notation for the domain of the function in the following theorem was chosen so that its statement is similar to the one of Theorem \ref{formula}. Recall the notation $k_Y$ for the support function of the convex set $Y$.
\begin{theorem}  \label{formula2}
Let $u$ be a piecewise linear convex function on $\R^d$. Assume that the convex hull $\widetilde{S}$ of the vertices of $u$ is a bounded set. 
If $\partial u(\R^d) = Y$, then for all $x \notin \widetilde{S}$
$$
u(x) = \min_{ s \in \partial \widetilde{S} } u(s) + k_Y(x-s).
$$
\end{theorem}

\begin{proof}
The proof is the same as the proof of Theorem \ref{formula}.  \Qed
\end{proof}

We have the following generalization of Theorem \ref{formula-0} where the infimum in \eqref{partial-on-1} is replaced by an infimum on the boundary of $\widetilde{S}$.

\begin{theorem}  \label{formula-00}
Let $S$ be a closed bounded convex set which defines a convex function $v$ on the projection $\widetilde{S}$ of $S$
onto $\R^d$. Let $D^* = \partial v((\widetilde{S})^\circ)$ and assume that $\tir{D^*} \subset \tir{\Omega^*}$. 
The convex set $K_{\Omega^*}+S$ defines a convex function $u$ on $\R^d$ which extends $v$ to $\R^d$ by 
\begin{equation} \label{partial-on-11}
u(x) = \inf_{y \in \partial \widetilde{S}} v(y) +  \sup_{p \in  \tir{\Omega^*} } \, p \cdot (x-y), x \notin \widetilde{S}.
\end{equation}
\end{theorem}

\begin{proof}
We first note that  \eqref{partial-on-1} also holds for $x \in \widetilde{S}$ as by \eqref{ext-part0}, for all $x \in \R^d$,
$u(x) \leq  \inf_{y \in \widetilde{S}} v(y) +  \sup_{p \in  \tir{\Omega^*} } \, p \cdot (x-y)$. Next, let $x \notin \widetilde{S}$ and suppose that $u(x) = v(y_1) +  \sup_{p \in  \tir{\Omega^*} } \, p \cdot (x-y_1)$ for $y_1 \in \widetilde{S}$ and furthermore
$ \sup_{p \in  \tir{\Omega^*} } \, p \cdot (x-y_1) = p_1 \cdot (x-y_1)$ where we used the compactness of $\widetilde{S}$ and $ \tir{\Omega^*}$. That is, $u(x) = v(y_1) + p_1 \cdot (x-y_1)$. Define
$$
V_0 = \{ \, y \in \R^d,  \sup_{p \in  \tir{\Omega^*} } \, p \cdot (y-y_1) = p_1 \cdot (y-y_1) 
\, \}.
$$
It can be readily checked that $V_0$ is convex and contains both $x$ and $y_1$. Let $y_1'$ denote the point of intersection with $\partial  \widetilde{S}$ of the half-line through $x$ starting at $y_1$. Since $V_0$ is convex, $y_1' \in V_0$ and thus $\sup_{p \in  \tir{\Omega^*} } \, p \cdot (y_1'-y_1) = p_1 \cdot (y_1'-y_1)$. So, by \eqref{partial-on-1}
\begin{equation} \label{cv-bd01}
u(y_1') \leq u(y_1) + p_1 \cdot (y_1'-y_1). 
\end{equation} 
Similarly, the set
$$
V_1 = \{ \, y \in \R^d, p_1 \cdot (x-y) \geq p \cdot (x-y) \, \forall p \in    \tir{\Omega^*} \, \},
$$
 is convex and contains both $x$ and $y_1$. 
Thus $\sup_{p \in  \tir{\Omega^*} } \, p \cdot (x-y_1')= p_1 \cdot (x-y_1')$. Since $v(y_1)=u(y_1)$, it thus follows from \eqref{partial-on-1} and \eqref{cv-bd01}
$$
u(x) \leq u(y_1') + p_1 \cdot (x-y_1') \leq u(y_1) + p_1 \cdot (x-y_1) = u(x),
$$
which shows that the minimum is reached at $y_1' \in \partial  \widetilde{S}$.
\Qed
\end{proof}
The above result can be used to simplify the proof of Theorem \ref{formula}. However, the proof of Theorem \ref{formula} illustrates the structure of piecewise linear convex functions. 
The following result was mentioned in the introduction. 
\begin{lemma} \label{subd-piecewise-c}
Let $u(x) = \max_{i=1,\ldots,k} x\cdot p_i + h_i$, for $p_i \in \R^d$ distinct and $h_i \in \R$
be a piecewise linear convex function on $\R^d$. Then $\partial u(\R^d) = \Conv\{ \, p_1,\ldots, p_k\, \} $. 
\end{lemma}

\begin{proof}

For $x \in \R^d$, $\partial u(x) = \Conv\{ \, p_i, i \in C_x\, \}$, 
where 
$$C_x = \{ \, i, 1\le i \le k, u(x) = x \cdot p_i + h_i\, \},
$$ 
c.f. for example \cite[Chapter 4, Theorem 3]{ioffe2009theory}. It follows that
\begin{equation} \label{sub-conv-p1k}
\partial u (\R^d) \subset \Conv\{ \, p_1,\ldots, p_k\, \}.
\end{equation}
Given a function $\phi$ on  $\R^d$, recall its Legendre transform defined on $\R^d$ by $\phi^*(y)=\sup_{x \in \R^d} x \cdot y - \phi(x)$. 
Let $y \in \Conv\{ \, p_1,\ldots, p_k\, \}$. We have $u^*(y) < \infty$, c.f. \cite[p. 387]{Yau2013} or \cite[Theorem 2.2.7 ]{hormander2007notions} for an explicit expression. Given $x \in \partial u^*(y)$ we have by \cite[Proposition 2.4]{Villani03} $y \in \partial u(x)$. Thus $\Conv\{ \, p_1,\ldots, p_k\, \} \subset 
\partial u(\R^d)$. We conclude that 
$\partial u(\R^d) = \Conv\{ \, p_1,\ldots, p_k\, \}$. \Qed
\end{proof}

\subsection{The second boundary condition in terms of an asymptotic cone}

Let $\nu$ be a Borel measure on $\R^d$. 

\begin{theorem} \cite{Bakelman1994} \label{exi-cone}
Assume that $\int_{\Omega^*} R(p) d p = \nu(\Omega)$. 
There exists a convex function $v$ on $\R^d$ with asymptotic cone $K_{\Omega^*}$ such that
$$
\omega(R,v,E) = \nu(E) 
\text{ for all Borel sets } E \subset \tir{\Omega}.
$$
Such a function is unique up to an additive constant.
\end{theorem}

\begin{corollary} \cite[p. 23]{Oliker03} \label{Oliker-cor}
The function $v$ given by Theorem \ref{exi-cone} satisfies $\chi_v (\tir{\Omega})  = \tir{\Omega^*}$.
\end{corollary}

Extending the function $v$ from Corollary \ref{Oliker-cor} to $\R^d$ using any of the procedures \eqref{ext-asc2} or \eqref{ext-sh2} below results in a function $\hat{v}$ on $\R^d$ which solves $\chi_{\hat{v}} (\R^d) = \chi_{\hat{v}} (\tir{\Omega})  = \tir{\Omega^*}$ by Lemma \ref{asw-34-weak} below, and hence $\hat{v}$ has asymptotic cone $K_{\Omega^*}$ by Lemma \ref{as-2-bvp}. Thus $\hat{v}=v$ and so $\chi_v (\R^d)=\chi_v (\tir{\Omega})  = \tir{\Omega^*}$.

Theorem \ref{exi-cone} and Corollary \ref{Oliker-cor} give existence of a convex solution $v$ on $\R^d$ which solves \eqref{m1m}. Its unicity up to a constant follows from Theorem \ref{exi-cone} and 
Lemma \ref{as-2-bvp}.

The second boundary value problem is often presented as the problem of finding a convex function $u$ on $\Omega$ such that
\begin{align} \label{m1m-new}
\begin{split}
\omega(R,u,E) &= \int_E f(x) d x \text{ for all Borel sets } E \subset \Omega \\ 
\partial u (\Omega)  &= \Omega^*.
\end{split}
\end{align}
The extension $\tir{u}$ based on \eqref{ext-sh} of a solution $u$ of \eqref{m1m-new}  solves \eqref{m1m}, c.f. Lemma \ref{lem-2-ext} below. Since solutions of \eqref{m1m} are unique up to a constant, a solution of \eqref{m1m} must be the extension of a solution of \eqref{m1m-new}. 

\subsection{Convex extensions revisited} \label{revisited}

Recall that $\Omega$ and $\Omega^*$ are assumed to be convex. We prove 
below that the two extensions $\tilde{u}$ and $\tir{u}$ given by \eqref{ext-asc} and \eqref{ext-sh} are equal. For that we will need the following lemma

\begin{lemma} \label{lem-2-ext}
Let $u_0 \in C(\tir{\Omega})$ such that $\partial u_0(\Omega)=\Omega^*$. For the convex extensions $\tilde{u}$ and $\tir{u}$ given respectively by \eqref{ext-asc} and \eqref{ext-sh}, the epigraph $M$ of $\tilde{u}$ has asymptotic cone $A+K_{\Omega^*}$ for $A \in M$ and
$\chi_{\tir{u}}(\tir{\Omega}) =\chi_{\tir{u}}(\R^d) = \tir{\Omega^*}$.
\end{lemma}

\begin{proof}
Put $\mu_{max} = \max_{x \in \tir{\Omega}} u_0(x)$. By Theorem \ref{formula-0}, the epigraph of $\tilde{u}$ is equal to $S+K_{\Omega^*}$ where $S$ is the closed bounded convex set
$\{ \, (x,\mu), x \in \tir{\Omega}, u_0(x) \leq \mu \leq \mu_{max}
\, \}$. By Theorem \ref{as-poly}, $S+K_{\Omega^*}$ has asymptotic cone $A+K_{\Omega^*}$ for $A \in M$. Note that by construction, $\tilde{u}=u_0$ on $\tir{\Omega}$ and  \eqref{ext-asc} gives the values of $\tilde{u}$ outside of $\tir{\Omega}$.



The claim that $\tir{u}$ is a convex extension of $u$ with $\chi_{\tir{u}}(\tir{\Omega}) = \tir{\Omega^*}$ follows from \cite[Lemma 4]{awanou2019uweakcvg}. \Qed
\end{proof}

\begin{lemma} \label{lem-2-ext-proof}
Let $u_0 \in C(\tir{\Omega})$ such that $\partial u_0(\Omega)=\Omega^*$. The convex extensions $\tilde{u}$ and $\tir{u}$ given respectively by \eqref{ext-asc} and \eqref{ext-sh} are equal.
\end{lemma}

\begin{proof} 
For a Borel set $E \subset \tir{\Omega}$ we define
$\omega(R,\tilde{u},E) := \omega(R,u_0,E \cap \Omega)$ and 
$ \omega(R,\tir{u},E)$ $:=\omega(R,u_0,E \cap \Omega)$, that is, $\omega(R,\tilde{u},E) =
 \omega(R,\tir{u},E)$ for all Borel sets $E \subset  \tir{\Omega}$. By Lemma \ref{lem-2-ext} the epigraph $M$ of $\tilde{u}$ has asymptotic cone $A+K_{\Omega^*}$ for $A \in M$ and
$\chi_{\tir{u}}(\tir{\Omega}) = \tir{\Omega^*}$. Thus $\tilde{u}$ has asymptotic cone $K_{\Omega^*}$ and by Lemma \ref{as-2-bvp}, $\tir{u}$ also has asymptotic cone $K_{\Omega^*}$. We conclude from Theorem \ref{exi-cone} that $\tilde{u}=\tir{u}$ since $\tilde{u}=\tir{u}=u_0$ on $\Omega$. \Qed
\end{proof}

The results we now prove were used in the proof of the equivalence of \eqref{m1m} and \eqref{m1m-new} in section \ref{R-curvature}. Let $E \subset \Omega$ 
and let $u$ be a 
convex function on $\Omega$. To extend $u|_E$, one may want to take into account $\partial u(\partial E)$. We thus consider the following variant of \eqref{ext-sh}
\begin{equation} \label{ext-sh2}
\hat{u}(x) = \sup \{ \, u(y) +  (x-y) \cdot z, y \in \tir{E}, z \in \partial u(y)  \, \}.
\end{equation}
First, we note
\begin{lemma} \label{subd-closed}
Let $E \subset \Omega$, $E$ bounded, $\Omega$ open and $u \in C(\Omega)$. Then $\partial u(\tir{E})$ is closed.
\end{lemma}

\begin{proof}
Let $p_n \in \partial u(\tir{E})$ and assume that $p_n \to p, p \in \R^d$. Let $a_n \in \tir{E}$ such that $p_n \in \partial u(a_n)$. For all $x \in \Omega$
$
u(x) \geq u(a_n) + p_n \cdot (x-a_n)
$.
Since $E$ is bounded, we may assume that $a_n \to a$ for $a \in \tir{E}$. We thus obtain $u(x) \geq u(a) + p \cdot (x-a)$ for all $x \in \Omega$. It follows that $p \in \partial u(a)$ and $\partial u(\tir{E})$ is closed. \Qed
\end{proof}

As with \cite[Lemmas 3 and 4]{awanou2019uweakcvg} we have
\begin{lemma} \label{asw-34-weak}
Let $E \subset \Omega$, $E$ bounded and $u$ a bounded convex function on $\Omega$. The extension $\hat{u}$ of $u|_{\tir{E}}$ given by \eqref{ext-sh2} is convex on $\R^d$ and if $\partial u(\tir{E})$ is bounded, for all $x \in \tir{E}$ we have $\chi_{\hat{u}} (x) = \partial u(x)$. Moreover
$$
\partial u(\tir{E}) = \chi_{\hat{u}}(\tir{E}) \subset  \chi_{\hat{u}} (\R^d) \subset \Conv(\partial u(\tir{E})). 
$$
\end{lemma}

\begin{proof} We only need to prove that for all $x \in \tir{E}$, $\chi_{\hat{u}} (x) \subset \partial u(x)$.
The other statements are proved as for  \cite[Lemmas 3 and 4]{awanou2019uweakcvg}, using the observation from Lemma \ref{subd-closed} that $\partial u(\tir{E})$ is closed.

Let $x \in \tir{E}$ and $p \in \chi_{\hat{u}} (x)$. Let $y \in \R^d$. We have
$\hat{u}(y) \geq \hat{u}(x) + p \cdot (y-x) = u(x) +  p \cdot (y-x)$. As $\tir{E}$ and  $\partial u(\tir{E})$ are bounded, we can find $y_0 \in \tir{E}$ and $z_0$ in $\partial u(y_0)$ such that $\hat{u}(y) = u(y_0)+ z_0 \cdot (y-y_0)$. If $y \in \Omega$, we have $u(y) \geq u(y_0)+ z_0 \cdot (y-y_0) = \hat{u}(y)  \geq u(x) +  p \cdot (y-x)$ which shows that $p \in \partial u(x)$. This completes the proof. \Qed
\end{proof}

We note that $\tir{\partial u(E)} \subset \partial u(\tir{E})$ and $\partial u(\tir{E})$ can be larger than $\tir{\partial u(E)}$. However, if $\partial u(E)$ is convex, it follows from \cite[Lemma 4]{awanou2019uweakcvg} that $\tir{\partial u(E)} =\chi_{\tir{u}}(\tir{E})$ where we recall that $\tir{u}$ is the extension of $u$ based on \eqref{ext-sh} which does not take into account $\partial E$.

The extension $\tilde{u}$ of $u|_E$ given by \eqref{ext-asc} would take into account only $\tir{\partial u(E)}$. We therefore consider the following variant
\begin{equation} \label{ext-asc2}
\check{u}(x) = \inf \{ \, u(y) + \sup_{z \in \partial u (\tir{E}) } (x-y) \cdot z,  y \in \tir{E} \, \}.
\end{equation}
Analogous to Lemma \ref{lem-2-ext-proof}, we have

\begin{lemma} \label{equal-2-ext-closed}
Let $E \subset \Omega$, $E$ bounded, convex and $u$ a bounded convex function on $\Omega$. Assume also that 
$\partial u(\tir{E})$ is bounded and convex. The  extensions $\hat{u}$ and $\check{u}$ of $u|_{\tir{E}}$ given by \eqref{ext-sh2} and \eqref{ext-asc2} are equal. 
\end{lemma}

\begin{proof} The proof is the same as for Lemma \ref{lem-2-ext-proof}. Put $D^*=\partial u(\tir{E})=\chi_{\hat{u}}(\tir{E}) $ and let $K_{D^*}$ denote 
the convex set associated with $D^*$ following \eqref{cone-def}. Then both $\hat{u}$ and $\check{u}$ have the same asymptotic cone $K_{D^*}$ and satisfy the same Monge-Amp\`ere equation on $\tir{E}$. \Qed
\end{proof}

We finish this subsection with an observation on the convex extensions of a piecewise linear convex function $u$ on $\Omega$. The result is used in the proof of Lemma \ref{gamma-lem}  below.
Let now $E \subset \Omega$ be a bounded convex polygonal domain. 

We may write for $x \in \Omega$, $u(x) = \max_{i=1,\ldots,k} x\cdot p_i + h_i$, for $p_i \in \R^d$ distinct and $h_i \in \R$. We assume 
that this expression also holds on $E$, or equivalently, all vertices of $u$ on $\Omega$ are vertices in $E$. The expression $\max_{i=1,\ldots,k} x\cdot p_i + h_i$ defines a convex extension of $u$ to $\R^d$ which we also denote by $u$.

It is known that $Y=\partial u(\R^d)$ is the convex polygonal domain $\Conv\{ \, p_1,\ldots, p_k\, \}$, c.f. Lemma \ref{subd-piecewise-c}. Let $p \in \Conv\{ \, p_1,\ldots, p_k\, \}$ and $x \in \R^d$ such that $p \in \partial u (x)= \Conv\{ \, p_i, i \in C_x\, \}$. This means that the hyperplanes $\{ \, (x,z) \in \R^{d+1}, z= x \cdot p_i + h_i, i \in C_x 
\, \} $ have a non-empty intersection and since $u(x)=\max_{i=1,\ldots,k} x\cdot p_i + h_i$ on $E$ as well, there is $y \in E$ such that $u(y)= x\cdot p_i + h_i, i \in C_x$, i.e. $p \in \partial u(y)$. Thus $\Conv\{ \, p_1,\ldots, p_k\, \}=Y=\partial u(\R^d) \subset \partial u(E) \subset \partial u (\R^d) \subset \Conv\{ \, p_1,\ldots, p_k\, \}$ where we used \eqref{sub-conv-p1k}. We conclude that $Y=\partial u (E)$ is a convex polygonal domain.


\begin{lemma} \label{ext-plc}
Let $E \subset \Omega$ be a bounded convex polygonal domain 
and let $u$ be a piecewise linear convex function on $\Omega$. Assume that all the vertices of $u$ in $\Omega$ are in $E$. Then the extensions $\check{u}$ and $\hat{u}$ of $u|_E$ based respectively on asymptotic cones and supporting hyperplanes, i.e. \eqref{ext-asc2} and \eqref{ext-sh2} are equal to $u$ on $\Omega$.
\end{lemma}

\begin{proof}
Note that $E$ is closed and $\partial u(E)$ is bounded and convex. By Lemma \ref{equal-2-ext-closed}, $\check{u} = \hat{u}$ on $\R^d$. We show that $\hat{u}=u$ on $\Omega$. Let us assume that on $\Omega$, $u(x) = \max_{i=1,\ldots,k} x\cdot p_i + h_i$, for $p_i \in \R^d$ distinct and $h_i \in \R$. We define
$$
v(x) = \sup \{ \, u(y) + p_i \cdot (x-y), y \in E, p_i, i \in C_y\, \}.
$$
By definition, for all $x \in \R^d$, $\hat{u}(x) \geq v(x)$. Let $y \in E$ and $z \in \partial u(y)$. Put $z=\sum_{i \in C_y} \lambda_i p_i$, $0 \leq \lambda_i \leq 1$ and $\sum_{i \in C_y} \lambda_i=1$. Since $v(x) \geq u(y) + p_i \cdot (x-y)$ for all $i \in C_y$, we obtain $v(x) \geq u(y) + z \cdot (x-y)$ and thus $v(x) \geq \hat{u}(x)$. We conclude that $v=\hat{u}$. 

Next, recall that by definition of $C_y$, if $p_i \in \partial u(y)$ and $i \in C_y$, we have $u(y)= y \cdot p_i + h_i$. It follows that $v(x) = \sup \{ \, x \cdot p_i + h_i, i \in C_y, y \in E\, \}=\max_{i=1,\ldots,k} x\cdot p_i + h_i$. We conclude that
$\hat{u}=v=u$ on $\Omega$. \Qed
\end{proof}


\section{Weak convergence of Monge-Amp\`ere measures for discrete convex functions} \label{wcg}

\begin{definition} \label{def-cv-mesh}
We say that $u_h$ converges to a function $u$ uniformly on $\tir{\Omega}$ in the sense of \cite{awanou2019uweakcvg}  if and only if for each sequence $h_k \to 0$ and for all $\epsilon >0$, there exists $h_{-1} >0$ such that for all $h_k$, $0< h_k < h_{-1}$, we have
$$
\max_{x \in \mathcal{N}_{h_k}^1} |u_{h_k}(x) - u(x)| < \epsilon.
$$
\end{definition}

\begin{theorem} \cite[Theorem 7]{awanou2019uweakcvg}  \label{old-02}
 Let $u_h$ 
 converge to a convex function $u$ uniformly on $\tir{\Omega}$ in the sense of \cite{awanou2019uweakcvg}. 
Assume also that $u$ is bounded. 
Then
$\omega(R,\Gamma_1(u_h),.)$ weakly converges to $\omega(R,u,.)$.
\end{theorem}


\begin{theorem}  \cite[Lemma 6]{awanou2019uweakcvg}  \label{old-03}
Let $u_{h}$ be discrete convex. If $u_h$ converges uniformly on compact subsets of $\Omega$ to a function $u \in C(\Omega)$ in the sense of \cite{awanou2019uweakcvg}, $u$ is convex on $\Omega$. 
\end{theorem}

\begin{theorem}  \cite[Theorem 12]{awanou2019uweakcvg}\label{old-04}
 Let $u_h$ be a family of discrete convex functions in the sense of \cite{awanou2019uweakcvg} 
 such that
$|u_h| \leq C$ for a constant $C$ independent of $h$ and $\chi_{\Gamma_1(u_h)}(\mathcal{N}_h^1)$ is uniformly bounded. Assume furthermore that $u_h$ is uniformly Lipschitz on $\tir{\Omega}$ and $u_h=\Gamma_1(u_h)$ on 
$\partial \Conv(\mathcal{N}_h^1)$.  
Then  there is a subsequence $h_k$ such that $u_{h_k}$ converges uniformly in the sense of \cite{awanou2019uweakcvg} to a convex function $v$ on $\tir{\Omega}$.
\end{theorem}
The above theorem gives not only the convergence of a subsequence of $\Gamma_1(u_h)$ but also the convergence of a subsequence of $u_h$. For the latter, we used a piecewise linear interpolant which is defined on a domain containing $\tir{\Omega}$, and is equal to $\Gamma_1(u_h)$ outside of $ \Conv(\mathcal{N}_h^1)$. The assumption $u_h=\Gamma_1(u_h)$ on 
$\partial \Conv(\mathcal{N}_h^1)$ is needed to make the interpolant globally Lipschitz. The latter assumption holds for the Dirichlet problem \cite[Lemma 5]{awanou2019uweakcvg}. 







Recall that for $V=V_{max}$, $\omega_V :=\omega_a$. The results in \cite{awanou2019uweakcvg}  are essentially for mesh functions and their convex envelopes. Theorems \ref{old-02}--\ref{old-04} hold for $\Gamma_2(u_h)$, $\partial_{V_{max}} u_h$ with the following definition of uniform convergence on $\tir{\Omega}$ which uses $\mathcal{N}_{h}^2$ whereas Definition \ref{def-cv-mesh} uses $\mathcal{N}_{h}^1$. Discrete convexity was defined in section \ref{disc-R-curvature}, Definition \ref{def-dc}.

\begin{definition} \label{def-cv-mesh2}
We say that $u_h$ converges to a function $u$ uniformly on $\tir{\Omega}$ if and only if for each sequence $h_k \to 0$ and for all $\epsilon >0$, there exists $h_{-1} >0$ such that for all $h_k$, $0< h_k < h_{-1}$, we have
$$
\max_{x \in \mathcal{N}_{h_k}^2} |u_{h_k}(x) - u(x)| < \epsilon.
$$
\end{definition}

\begin{theorem} \label{main0-wcg}
 Let $u_h$ be a family of discrete convex functions  such that $u_h$ 
 converges to a convex function $u$ uniformly on $\tir{\Omega}$. 
Assume also that $u$ is bounded. 
Then
$\omega_a(R,u_h,.)$ weakly converges to $\omega(R,u,.)$.
\end{theorem}

\begin{theorem} \label{new-03}
Let $u_{h}$ be discrete convex. If $u_h$ converges uniformly on compact subsets of $\Omega$ to a function $u \in C(\Omega)$, $u$ is convex on $\Omega$. 
\end{theorem}

For the analogue of Theorem \ref{old-04}, note that $\mathcal{N}_h^1  \subset \tir{\Omega} $ and the convex extension to $\R^d$ of $\Gamma_1(u_h)$ is used in \cite{awanou2019uweakcvg} to have an interpolant defined on $\tir{\Omega}$. Lemma \ref{pre-lip-lem} gives the Lipschitz continuity on $\tir{\Omega} \cap \mathbb{Z}_h^d$ of a discrete convex function with asymptotic cone $K$. However $\partial \Omega \cap \mathbb{Z}_h^d$ may be empty. But we can use the Lipschitz continuity of $u_h$ on $\Omega_h$. 
An interpolant of $u_h$ equal to $\Gamma_2(u_h)$ outside of $\Conv(\Omega_h)$ can be constructed.

\begin{theorem} \label{old-14}
Let $u_h$ be a family of discrete convex functions 
 such that
$|u_h| \leq C$ for a constant $C$ independent of $h$ and $\chi_{\Gamma_2(u_h)}(\mathcal{N}_h^2)$ is uniformly bounded. Assume furthermore that $u_h$ is uniformly Lipschitz on $\tir{\Omega}$ and $u_h=\Gamma_2(u_h)$ on 
$\partial \Conv( \Omega_h  )$.  Then  there is a subsequence $h_k$ such that $u_{h_k}$ converges uniformly to a convex function $v$ on $\tir{\Omega}$.
\end{theorem}

If $\Omega$ is a rectangle, and $u_h$ is discrete convex with asymptotic cone $K$, by Lemma \ref{pre-lip-lem}, $u_h$ is Lipschitz on $\tir{\Omega} \cap \mathbb{Z}_h^d$ and a piecewise linear interpolant $I(u_h)$ of $u_h$ on $\Conv(\tir{\Omega} \cap \mathbb{Z}_h^d)$ is uniformly Lipschitz on $\tir{\Omega}$ and uniformly bounded. By the Arzela-Ascoli theorem, there is a subsequence $h_k$ such that $u_{h_k}$ converges uniformly to a function $v$ on $\tir{\Omega}$ which is convex by Theorem \ref{new-03}. We therefore have the following theorem.

\begin{theorem} \label{boom-thm} 
Assume that $\Omega$ is a rectangle and $u_h$ is discrete convex with asymptotic cone $K$. There is a subsequence $h_k$ such that $u_{h_k}$ converges uniformly to a convex function $v$ on $\tir{\Omega}$.
\end{theorem}

We will use the above theorem in section \ref{cvg-visc} for stencils $V=V_{\kappa} \cap V_{max}$ with size uniformly bounded and allow $\kappa \to \infty$. 

\begin{lemma} \label{full-subd}
If a mesh function $u_h$ solves \eqref{m2d3} for $f>0$ on $\Omega$, then $\partial_V u_h (\Omega_h)=Y$ for $V=V_{max}$. 
\end{lemma}

\begin{proof}
By assumption, a solution of  \eqref{m2d3} has asymptotic cone $K$. Since $f>0$ on $\Omega$,  $\partial_V u_h(x) \neq \emptyset$ for $x \in \Omega_h$ and $u_h$ is discrete convex by Lemma \ref{auto-dc}. By Theorem \ref{from-out} $\partial \Gamma_2(u_h)(x) =  \partial_V u_h(x)$. But for $x \neq y$,
$\partial \Gamma_2(u_h)(x) \cap \partial \Gamma_2(u_h)(y)$ is a set of measure 0 by \cite[Lemma 1.1.8]{Guti'errez2001}. We conclude that $\int_{\cup_{x \in \Omega_h}  \partial_V u_h (x)} R(p) dp = \sum_{x \in \Omega_h} \omega_a(R,u_h,\{\, x \,\})  = \int_{Y} R(p) d p$ where we used \eqref{mass-conservation}. Since by Lemma \ref{inc-sub-lem} we have $\partial_V u_h(\Omega_h) \subset Y$ 
we get $\cup_{x \in \Omega_h}  \partial_V u_h (x) = Y$ up to a set of measure 0. Since $Y$ is a polygon and for each $x \in \Omega_h$, $\partial_V u_h (x)$ is also a polygon, we obtain $\partial_V u_h (\Omega_h)=Y$. \Qed
\end{proof}

Theorem \ref{boom-thm} is enough to extract a converging subsequence for solutions of \eqref{m2d3}. In addition, by \cite[Lemma 10]{awanou2019uweakcvg}, the uniform convergence of $u_h$ implies the uniform convergence of the convex envelopes $\Gamma_2(u_h)$. 
The following lemma gives conditions under which $\chi_{\Gamma_2(u_h)}(\mathcal{N}_h^2)$ is uniformly bounded. 
It can be used to extract a convergent subsequence from $\Gamma_2(u_h)$ when $V=V_{max}$. 

\begin{lemma} \label{key-lem}
Assume that  
$u_h$ is discrete convex with asymptotic cone $K$. 
Then $\chi_{\Gamma_2(u_h)}(\Conv(\Omega_h)) \subset Y$ 
and $\chi_{\Gamma_2(u_h)}(\mathcal{N}_h^2)$ $
=\chi_{\Gamma_2(u_h)}(\R^d)$ is uniformly bounded. 

\end{lemma}

\begin{proof}

{\bf Part 1 } We first  prove that if $z \in \Omega_h$ and 
$\Gamma_2(u_h)(z) = u_h(z)$, then $\chi_{\Gamma_2(u_h)}(z)$ $\subset \partial_V u_h(z) \subset Y$. 

Let then $p \in \chi_{\Gamma_2(u_h)}(z)$. We have for all $s \in \R^d$, $\Gamma_2(u_h)(s) \geq 
\Gamma_2(u_h) (z) + p \cdot (s-z)$. If $s \in  \mathcal{N}_h^2$, we get $u_h(s) \geq \Gamma_2(u_h)(s) \geq 
u_h(z) + p \cdot (s-z)$. In particular, for $e \in V(z) \subset V_{max}(z)$ and $s=z+ h e$ we obtain $u_h(z+ h e) \geq 
u_h(z) + p \cdot (h e)$. This proves that $\chi_{\Gamma_2(u_h)}(z) \subset \partial_V u_h(z)$. By Lemma \ref{inc-sub-lem}, 
$ \partial_V u_h(z) \subset Y$.

{\bf Part 2 } We prove that $\chi_{\Gamma_2(u_h)}(\Conv(\Omega_h)) \subset Y$. We use notions of faces of polyhedra reviewed in section \ref{comp-section}. 
Recall from Definition \ref{def-conv-subd} the convex subdivision $\mathcal{T}_h$ associated with the piecewise linear convex function $\Gamma_2(u_h)$ on $\Conv(\mathcal{N}_h^2)$. If $\sigma \in \mathcal{T}_h$, $\sigma$ is a convex polyhedron in $\R^d$, $\Conv(\mathcal{N}_h^2) = \cup_{\sigma \in \mathcal{T}_h} \sigma$, if $\sigma, \tau \in \mathcal{T}_h$, then $\sigma \cap \tau \in \mathcal{T}_h$, and if $\sigma \in \mathcal{T}_h$ and $\tau \subset \sigma$, $\tau \in \mathcal{T}_h$ if and only if $\tau$ is a face of $\sigma$. On each $d$-dimensional cell $\sigma \in \mathcal{T}_h$, $\Gamma_2(u_h)$ is a linear function. 

Recall that for a  vertex $x$ of $\mathcal{T}_h$, we have $\Gamma_2(u_h)(x) = u_h(x)$, c.f. for example \cite{awanou2019uweakcvg}. For $x$ in the interior of $\Conv(\mathcal{N}_h^2)$, let $\omega(x)$ denote the collection of the $d$-dimensional cells $\sigma \in \mathcal{T}_h$ such that $x \in \sigma$. It is known, using for example \cite[Theorem 5]{awanou2019uweakcvg} that $\partial \Gamma_2(u_h)(x)$ is the convex hull of the constant gradients of $ \Gamma_2(u_h)$ on elements $\sigma \in \omega(x)$. 

Let $z \in \Conv(\Omega_h)$ and let $\tau$ denote a $d$-dimensional cell in $\mathcal{T}_h$ such that $z \in \tau$. If all vertices of $\tau$ are in $\R^d \setminus \Conv(\Omega_h)$, then $z \notin \Conv(\Omega_h)$. Thus, at least one vertex $x$ of $\tau$ is in $\Omega_h$.  

If $z \in \tau^\circ$, then $\partial \Gamma_2(u_h)(z) = \{ \, p \, \}$ where $p$ is the gradient of $ \Gamma_2(u_h)$ at $z$. Thus $\partial \Gamma_2(u_h)(z)  \subset \partial \Gamma_2(u_h)(x)$, and since $\Gamma_2(u_h)(x) = u_h(x)$ we get $\partial \Gamma_2(u_h)(z)  \subset Y$.

If $z \in \partial \tau$ and $z$ is a vertex of $\tau$, we must have $z \in \Omega_h$ since $z \in \mathcal{N}_h^2$ and $z \in \Conv(\Omega_h)$. Also, $ \Gamma_2(u_h)(z) = u_h(z)$. We then have $\partial \Gamma_2(u_h)(z)  \subset Y$.

Suppose $z \in \partial \tau$ and $z$ is not a vertex of $\tau$. Let $\gamma$ be a lowest dimensional cell such that $z \in \gamma$. At least one vertex $x$ of $\gamma$ must be in $\Omega_h$. For $\sigma \in \omega(z)$, $\sigma \cap \gamma$ is a cell of $\mathcal{T}_h$ which must be a face of $\gamma$ and contains $z$. By the assumption on $\gamma$, we have $\sigma \cap \gamma=\gamma$ and hence $x \in \sigma$, i.e. $\sigma \in \omega(x)$. We conclude that $\omega(z) \subset \omega(x)$ and hence $\partial \Gamma_2(u_h)(z)  \subset \partial \Gamma_2(u_h)(x)$.  As above, we obtain $\partial \Gamma_2(u_h)(z)  \subset Y$.

{\bf Part 3 } 
Put $D^* = \partial \Gamma_2(u_h)(\Conv(\Omega_h)^\circ)$. Let $S$ be a closed convex set the projection of which on $\R^d$ is equal to $\Conv(\Omega_h)$. We have $\tir{D^*} \subset Y$. By Theorem \ref{formula}, the convex set $S+K$ defines a convex function $v$ on $\R^d$ which extends $\Gamma_2(u_h)|_{\Conv(\Omega_h)}$ and such that $v(z)$ for $z \in \R^d \setminus \Conv(\Omega_h)$ is given by Theorem \ref{formula}, i.e.
\begin{equation} \label{temp-ext}
v(z) = \inf_{y \in \partial \Conv(\Omega_h) } \Gamma_2(u_h)(y) + k_Y(z-y).
\end{equation}
By Lemma \ref{lem-2-ext}, $\chi_v(\R^d) = Y$. Thus, there exists a constant $C$ independent of $h$ such that for all , 
\begin{equation} \label{punch01}
|v(x) - v(y)| \leq C | x-y|, \forall x, y \in  \mathcal{N}_h^2,
\end{equation}
where $|x|^2= x \cdot x$.   

Moreover, for $x \in  \mathcal{N}_h^2 \setminus \Omega_h$, $u_h(x) =  \inf_{y \in \partial \Omega_h} u_h(y) + k_Y(x-y)$. 
Therefore, by \eqref{temp-ext}, 
 $v(x) \leq u_h(x)$ for all $x \in  \mathcal{N}_h^2 \setminus \Omega_h$. Since by construction $v=\Gamma_2(u_h)$ on $\Conv(\Omega_h)$, we obtain $v(x) \leq u_h(x)$ for all $x \in  \mathcal{N}_h^2$. As $\Gamma_2(u_h)$ is the largest convex function majorized by $u_h$ on 
 $\mathcal{N}_h^2$, we obtain 
 \begin{equation} \label{punch02}
 v(x) \leq \Gamma_2(u_h)(x) \, \text{for all} \, x \in  \mathcal{N}_h^2, 
 \end{equation}
which can also be seen by taking a supporting hyperplane to the graph of $v$ and the definition of $\Gamma_2(u_h)$. 


Let now $x \in \mathcal{N}_h^2 \setminus \Conv(\Omega_h)$ and $q \in \chi_{\Gamma_2(u_h)}(x)$. We have 
$q \cdot (z-x) \leq \Gamma_2(u_h)(z) - \Gamma_2(u_h)(x)$ for all $z \in \R^d$. Let $e_i, i=1,\ldots, d$ be a set of independent vectors such that $z_i=x + c_i e_i$ is in $\Omega_h$ for $c_i >0$. Using $z_i \in \Omega_h$, \eqref{punch02} and \eqref{punch01}, we obtain
$$
q \cdot (c_i e_i) \leq v(z_i) - v(x) \leq C |z_i-x| = c_i  C |e_i|. 
$$
We conclude that $q \cdot e_i/|e_i| \leq C, i=1,\ldots,d$. 

Next, let $l_i >c_i > 0$ such that $s_i=x - l_i e_i, i=1,\ldots, d$ is not in $\Conv(\mathcal{N}_h^2)$. We have using Theorem \ref{formula-0}, 
$\Gamma_2(u_h)(s_i) \leq \Gamma_2(u_h)(z_i) + k_Y(s_i-z_i)$. Thus 
\begin{multline*}
l_i q \cdot (- e_i) = q \cdot (s_i-x) \leq \Gamma_2(u_h)(s_i) - \Gamma_2(u_h)(x) \leq
 \Gamma_2(u_h)(s_i) - v(x) \\
  \leq  \Gamma_2(u_h)(z_i) + k_Y(s_i-z_i) - v(x) = v(z_i) - v(x) + k_Y(s_i-z_i) \\
  \leq  c_i  C |e_i| + 2 c_i k_Y(-e_i).
\end{multline*}
We conclude that $q \cdot (-e_i)/|e_i| \leq C + 2 k_Y(-e_i/|e_i|)$. Since $Y$ is bounded, it follows that $q \cdot (\pm e_i)/|e_i| \leq C$ for a constant $C$ independent of $h$. This proves that $\chi_{\Gamma_2(u_h)}(\mathcal{N}_h^2)$ is uniformly bounded. By Lemma \ref{lem-2-ext} $\chi_{\Gamma_2(u_h)}(\mathcal{N}_h^2)
=\chi_{\Gamma_2(u_h)}(\R^d)$. 
\Qed.
\end{proof}



For $f>0$ on $\Omega$, by Theorem \ref{from-out} and Lemma \ref{full-subd}, as we will see, convergence of the discretization \eqref{m2d3} for $V=V_{max}$ reduces to proving convergence results for the convex envelope $\Gamma_2(u_h)$. 
Analogous to Lemma \ref{key-lem}, we have

\begin{lemma} \label{gamma-lem}
Assume that $u_h$ is discrete convex with asymptotic cone $K$. Then $\chi_{\Gamma_1(u_h)}(\Conv(\Omega_h)) \subset Y$ 
and $\chi_{\Gamma_1(u_h)}(\mathcal{N}_h^1)$ $
=\chi_{\Gamma_1(u_h)}(\R^d)$ is uniformly bounded. 
\end{lemma}







\section{Convergence of the discretization} \label{cvg}


Recall the truncation $\tilde{f}$ of $f$ defined by \eqref{modified-f}.
Set
$$\tilde{f}(t)=0 \text{ outside } \Omega.$$

Given a Borel set $E \subset \tir{\Omega}$ we define 
$$
\nu_{h} (E) = \sum_{x \in B \cap \Omega_h}   \int_{E_{x}} \tilde{f}(t) dt.
$$

We recall that a sequence $\mu_n$ of Borel measures converges to a Borel measure $\mu$ if and only if $\mu_n(B) \to \mu(B)$ for any Borel set $B$ with $\mu(\partial B)=0$. Let $h_k$ be a sequence converging to 0. Then $\nu_{h_k}$ weakly converges to the measure $\nu$ defined by $\nu(B)=\int_B \tilde{f}(t) dt$.



In this section, we first give the convergence of the discretization for $V=V_{max}$. 
We then consider the case $V$ not necessarily equal to $V_{max}$ and $f \in C(\Omega)$. We finish with a result about convergence of approximations when $\Omega^*$ is approximated by polygons.

\subsection{Convergence when $\partial_V u_h(\Omega_h) = Y$}

When $V=V_{max}$, by Lemma \ref{full-subd} $\partial_V u_h(\Omega_h) = Y$ for a solution of \eqref{m2d3}. Recall from Theorem \ref{stability-thm} that solutions $u_h$ of \eqref{m2d3} with $u_h(x^1_h)=\alpha$ for an arbitrary number $\alpha$ and $x^1_h \in \Omega_h$, are uniformly bounded in $h$. 

\begin{theorem} \label{cvg-thm}
For $f>0$ on $\Omega$ and $V=V_{max}$, solutions $u_h$ of \eqref{m2d3} with $u_h(x^1_h)=\alpha$ for 
$x^1_h$ in $\Omega_h$ 
and $x^1_h \to x^1 \in \tir{\Omega}$, converge uniformly on $\tir{\Omega}$ to the unique solution $u$ of \eqref{m1m} with $u(x^1)=\alpha$.
\end{theorem}

\begin{proof} 
{\bf Part 1 } Existence of a converging subsequence with converging measures.

By Remark \ref{dc-in-the-sense} of section \ref{wcg}, a discrete convex function is discrete convex as defined in \cite{awanou2019uweakcvg}. Since $V=V_{max}$, $u_h=\Gamma_1(u_h)$ on $\Omega_h$. 
By Lemma \ref{gamma-lem} $\chi_{\Gamma_1(u_h)}( \mathcal{N}_h^1 )  \subset \chi_{\Gamma_1(u_h)}( \R^d)$ is uniformly bounded. Thus, by \cite[Lemma 15]{awanou2019uweakcvg} we have 
$$
|u_h(x) - u_h(y)| \leq C ||x-y||_1, \forall x, y \in \mathcal{N}_h^1,
$$
i.e. the discrete convex mesh functions $u_h$ are uniformly Lipschitz on $\tir{\Omega}$.   As $u_h(x^1_h)=\alpha$ we have $|u_h|\leq C$ with $C$ independent of $h$. Therefore, by Theorem \ref{old-04}, there exists a subsequence $h_k$ such that $u_{h_k}$ converges uniformly on $\tir{\Omega}$, as defined in Definition \ref{def-cv-mesh}, to a convex function $v$ on $\tir{\Omega}$, which is necessarily bounded. By 
Lemma \ref{gamma-lem}, Theorems \ref{from-out}, \ref{old-01} and \ref{old-02}, 
$\omega_a(R,u_{h_k},.)=\omega(R,\Gamma_2(u_{h_k}),.) = \omega(R,\Gamma_1(u_{h_k}),.)$ 
weakly converges to $\omega(R,v,.)$. We conclude that
$$
\omega(R,v, E ) = \int_E \tilde{f}(t) dt = \omega(R,u, E ),
$$
since from \eqref{m2d3}, $\omega_a(R,u_{h},E)=\nu_{h} (E)$ for all Borel sets $E \subset \Omega$.

{\bf Part 2 } The limit function has asymptotic cone $K$.

We claim that $u_{h_k}$ converges pointwise, up to a subsequence, to $v$ on $\R^d\setminus \tir{\Omega}$ with $v$ given for $x \notin \tir{\Omega}$ by
\begin{equation} \label{limit-ext}
v(x) = \inf_{s \in \partial \tir{\Omega}} \ v(s) + \max_{j=1,\ldots,N^*} (x-s) \cdot a_j^*.
\end{equation}
Let $x_h \to x$ as $h \to 0$. We may assume that $x_h \notin \Omega_h$. Therefore
$u_h(x_h) = u_h(y_h) + \max_{j=1,\ldots,N^*} (x_h-y_h) \cdot a^*_j$ for $y_h \in \partial \Omega_h$. Let $y_{h_k}$ be a subsequence such that $y_{h_k} \to y \in \tir{\Omega}$. Since $y_h \in \partial \Omega_h$, we have $y \in \partial \tir{\Omega}$. If necessary, by taking a further subsequence, we use the uniform convergence of $u_{h_k}$ to $v$ on $\tir{\Omega}$ to conclude that $u_{h_k}(y_{h_k}) \to v(y)$. We may write
$ \max_{j=1,\ldots,N^*} (x_{h_k}-y_{h_k}) \cdot a^*_j = (x_{h_k}-y_{h_k}) \cdot a^*_{j_k}$, and again up to a subsequence, this converges to
$(x-y) \cdot a^*_l$ for some $l \in \{ \,1,\ldots,N^* \, \}$. Since $ (x_{h_k}-y_{h_k}) \cdot a^*_{j_k} \geq 
(x_{h_k}-y_{h_k}) \cdot a^*_j$ for all $j$, we get $(x-y) \cdot a^*_l =  \max_{j=1,\ldots,N^*} (x-y_{}) \cdot a^*_j$. We conclude that $u_{h_k}(x_{h_k})$ converges to 
$$
v(y) +  \max_{j=1,\ldots,N^*} (x-y_{}) \cdot a^*_j, \text{ for } y \in \partial \tir{\Omega}. 
$$
Next, if $z \in \partial \Omega$ and $z_h \to z$, $z_h \in  \partial \Omega_h$, we have $u_h(x_h) \leq u_h(z_h) + \max_{j=1,\ldots,N^*} (x_h-z_h) \cdot a^*_j$
and repeating the same argument, we obtain for all $z \in \partial \Omega$
$$
v(y) +  \max_{j=1,\ldots,N^*} (x-y_{}) \cdot a^*_j \leq v(z) +  \max_{j=1,\ldots,N^*} (x-z_{}) \cdot a^*_j.
$$
This proves \eqref{limit-ext}. As a consequence, 
by Theorem \ref{formula-00}, the limit function $v$ coincides with a function on $\R^d$ with asymptotic cone $K$, i.e.
$v$ has asymptotic cone $K$. 
We conclude by Corollary \ref{Oliker-cor} that 
$$
\chi_v (\tir{\Omega})  = Y.
$$ As a consequence
\begin{equation} \label{conserve}
\omega(R,v,  \tir{\Omega} ) = \int_{Y} R(p) dp = \omega(R,u,  \tir{\Omega} ).
\end{equation}

{\bf Part 3 } The limit function solves \eqref{m1m}.

Since $u_{h_k}$ converges uniformly to $v$ on $\tir{\Omega}$, by  \cite[Lemma 10]{awanou2019uweakcvg} $\Gamma_1(u_{h_k})$ converges uniformly on compact subsets of $\Omega$ to $v$. By \cite[Lemma 1.2.2]{Guti'errez2001}, for each compact set $K \subset U \subset \tir{U} \subset \Omega$ for an open set $U$, $\partial v(K) \subset \liminf_{h_k \to 0} \partial \Gamma_1(u_{h_k})(U)= \liminf_{h_k \to 0} \partial_V u_{h_k}(U)$ up to a set of measure 0. 
Here, we also used Lemma \ref{gamma-lem}. We recall from Lemma \ref{inc-sub-lem} that $\partial_V u_{h_k}(\Omega_{h_k}) \subset Y$. Thus $\chi_v(\Omega) \subset Y$.

Next, we recall that the set of points which are in the normal image of more than one point is contained in a set of measure 0,  \cite[Lemma 1.1.12]{Guti'errez2001}. As $\chi_v (\tir{\Omega})  = Y$ and $\chi_v(\Omega) \subset Y$, we have
$\chi_v(\partial \Omega) \subset \partial Y$ up to a set of measure 0. In other words, $|\chi_v(\partial \Omega)|=0$.
We conclude that 
\begin{multline*}
\omega(R,v, E ) = \omega(R,v, E\cap \Omega )+ \omega(R,v, E\cap \partial \Omega ) =
 \omega(R,v, E\cap \Omega ) \\ =  \omega(R,u, E\cap \Omega ) \leq \omega(R,u, E ),
\end{multline*}
for all Borel sets $E \subset \tir{\Omega}$. 
Thus, it is not possible to have $\omega(R,v, E )  < \omega(R,u, E )$ for a Borel set $E$ since that would give
\begin{multline*}
\omega(R,v,  \tir{\Omega} ) = \omega(R,v,  E  ) + \omega(R,v,  \tir{\Omega} \setminus E )
<   \omega(R,u,  E  ) + \omega(R,u,  \tir{\Omega} \setminus E ) = \omega(R,u,  \tir{\Omega} ),
\end{multline*}
contradicting \eqref{conserve}. We conclude that $\omega(R,v, E )  = \omega(R,u, E ) $ for all Borel sets $E \subset \tir{\Omega}$. 

As $u_{h_k}$ converges uniformly to $v$ on $\tir{\Omega}$ and $x^1_h \to x^1$, $u_{h_k}(x^1_{h_k}) \to v(x^1)$. Thus $v(x^1)=\alpha$. Since
\eqref{m1m} has a unique solution with $u(x^1)=\alpha$ and $v(x^1)=\alpha$, we have $u=v$ and hence $u_h$ converges uniformly on $\tir{\Omega}$ to $u$. \Qed
\end{proof}

\subsection{Convergence when $\partial_V u_h(\Omega_h)$ is not necessarily equal to $Y$} \label{cvg-visc}

In this section we consider the case $V_{min} \subset V \subset V_{max}$. For a solution of \eqref{m2d3}, we have 
$\partial_V u_h(\Omega_h) \subset Y$, but we may have $\partial_V u_h(\Omega_h) \neq Y$. Thus arguments for convex functions no longer apply. We will use arguments for convergence to viscosity solutions. 
But we will also use the Lipschitz continuity of mesh functions to extract subsequences, c.f. Theorem \ref{boom-thm}. Our convergence results are thus for $\Omega$ a rectangle. There is no loss of generality as Problem \ref{m1} has an equivalent formulation on a larger rectangular domain $\widetilde{\Omega}$ by setting $f=0$ on $\widetilde{\Omega} \setminus \Omega$. Recall that for a solution $u$ of \eqref{m1}, we have $\chi_u(\tir{\Omega}) = \chi_u(\R^d)=\tir{\Omega^*}$. The existence of solution to \eqref{m2d3} in the degenerate case $f \geq 0$ is discussed in section \ref{degenrate-case}. If $V(x)=V_{max}(x)$ for all $x \in \partial \Omega_h$, then convergence on a bounded convex domain can be proven based on Theorem \ref{old-14}. 

We denote by $|.|$ the matrix norm induced by the Euclidean norm $|.|$ on $\R^d$. Let $M$ be a symmetric positive definite $d \times d$ matrix and $p(x)=1/2 \, x^T M x$ be a strictly convex quadratic polynomial. Recall that the condition number of $M$ is given by $\sqrt{|M| \, |M^{-1}|}$. Let $\lambda$ and $\Lambda$ denote the smallest and largest eigenvalues of $M$. It is known that $|M|=\Lambda$ and thus similarly $|M^{-1}| = 1/\lambda$. So the condition number of $M$ is $\sqrt{\Lambda/\lambda}$. 

If $p(x)=1/2 \, x^T M x$ and $M$ has condition number less than $\kappa$, we say that $p$ is a quadratic polynomial with condition number less than $\kappa$. 

\begin{definition}
A convex function $u \in C(\tir{\Omega})$ is a viscosity solution of 
\begin{equation} \label{monge}
R(D u(x)) \det D^2 u(x) =f(x),
\end{equation}
 in $\Omega$ if
for all $\phi \in C^2(\Omega)$ the following holds
\begin{itemize}
\item[-] at each local maximum point $x_0$ of $u-\phi$, $f(x_0) \leq R(D \phi (x_0)) \det D^2 \phi(x_0)$
\item[-] at each local minimum point $x_0$ of $u-\phi$, $f(x_0) \geq R(D \phi (x_0)) \det D^2 \phi(x_0)$, if $D^2 \phi(x_0) \geq 0$, i.e. $D^2 \phi(x_0)$ has positive eigenvalues.
\end{itemize}
\end{definition}

As explained in \cite{Ishii1990}, the requirement $D^2 \phi(x_0) \geq 0$ in the second condition above is natural for the two dimensional case. 
The space of test functions in the definition above can be restricted to the space of strictly convex quadratic polynomials \cite[Remark 1.3.3]{Guti'errez2001}. We will refer to the conditions above as the conditions in the definition of viscosity solution for the test function $\phi$.


\begin{definition}
A convex function $u \in C(\tir{\Omega})$ is a $\kappa$-viscosity solution of \eqref{monge} if the conditions in the definition of viscosity solution hold for all strictly convex quadratic polynomials with condition number less than $\kappa$. 
\end{definition}
A viscosity solution of \eqref{monge} is a $\kappa$-viscosity solution for all $\kappa >0$. 

\subsubsection{ Equivalence with Aleksandrov solutions} \label{equi-A-visc}

We recall that an Aleksandrov solution of \eqref{monge} is a convex function $u \in C(\tir{\Omega})$ such that $\omega(R,u,E) = \int_E f(x) dx$ for all Borel sets $E \subset \Omega$.

For $f>0$ and $f \in C(\tir{\Omega})$, one proves as with \cite[ Propositions 1.3.4 and 1.7.1]{Guti'errez2001} that a convex function $u \in  C(\tir{\Omega})$  is an Aleksandrov solution of \eqref{monge} if and only if it is a viscosity solution of \eqref{monge}. 

\subsubsection{ Convergence to the viscosity solution}

The scheme \eqref{m2d3} is said to be monotone if for $z_h$ and $w_h$ in  $\mathcal{C}_h$, $z_h(y) \geq w_h(y), y \neq x$ with $z_h(x) = w_h(x)$, we have $\omega(R,z_h, \{ \, x\, \})  \geq \omega(R,w_h, \{ \, x\, \}) $. One proves as with \cite[Lemma 3.7]{DiscreteAlex2} that the scheme \eqref{m2d3} is monotone. 

We say that the scheme \eqref{m2d3} is consistent if for all $C^2$ convex functions $\phi$, a sequence $x_h \to x \in \Omega$ 
$$
\lim_{h \to 0} \frac{1}{h^d} \omega(R, \phi , \{ \, x_h\, \})  = \det D^2 \phi(x).
$$ 
We will also use the terminology of consistent with a class of smooth functions. 

Analogous to \cite[Theorem 3.9]{DiscreteAlex2} and similarly to the end of Part 3 of the proof of Theorem \ref{cvg-thm}, we have
\begin{theorem}
Assume that $V=V_{max}$ and the scheme \eqref{m2d3} is consistent. If the solution $u_h$ of \eqref{m2d3}, with $u_h(x^1_h)=\alpha$ for 
$x^1_h$ in $\Omega_h$ 
and $x^1_h \to x^1 \in \tir{\Omega}$, converges uniformly on $\tir{\Omega}$ to a convex function $v$, then $v$ is a viscosity solution of \eqref{monge} with $v(x^1)=\alpha$.
\end{theorem}

Recall the definition of the stencil $V_{\kappa}$ from section \ref{disc-R-curvature}, i.e. $V_{\kappa}$ consists of all vectors $e \in \mathbb{Z}^d \setminus \{ \, 0 \, \}$ with co-prime coordinates such that $|e| \leq 1/2 \sqrt{d} \kappa$.  Analogous to the above theorem we have

\begin{theorem} \label{final-thm-01}
Assume that $V=V_{\kappa} \cap V_{max}$ and the scheme \eqref{m2d3} is consistent for strictly convex quadratic polynomials with condition number less than $\kappa$. If the solution $u_{h,\kappa}$ of \eqref{m2d3}, with $u_{h,\kappa}(x^1_h)=\alpha$ for 
$x^1_h$ in $\Omega_h$ 
and $x^1_h \to x^1 \in \tir{\Omega}$, converges uniformly on $\tir{\Omega}$ to a convex function $v_{\kappa}$, then $v_{\kappa}$ is a $\kappa$-viscosity solution of \eqref{monge} with $v_{\kappa}(x^1)=\alpha$.

\end{theorem}

We establish below the consistency of \eqref{m2d3} for  $V=V_{\kappa}  \cap V_{max}$, for strictly convex quadratic polynomials, at interior points at a distance $C h$ of $\partial \Omega$. To check the conditions in the definition of viscosity solution at a point $x \in \Omega$, one first take $h$ sufficiently small and check the conditions at mesh points $x_h$ close to $x$. See 
the proof of Theorem \ref{final-thm-01} in section \ref{proof-of-final-thm-01} below. 

\begin{theorem} \label{final-thm-02} Let $\Omega$ be a rectangle. 
Assume that $u_{h,\kappa}$ is discrete convex and solves \eqref{m2d3} for $V=V_{\kappa}  \cap V_{max}$  with $u_{h,\kappa}(x^1_h)=\alpha$ for 
$x^1_h$ in $\Omega_h$ 
and $x^1_h \to x^1 \in \tir{\Omega}$. There is a subsequence $h_k$ such that  $u_{h_k,\kappa}$ converges uniformly on $\tir{\Omega}$ to a continuous convex function $v_{\kappa}$ with $v_{\kappa}(x^1)=\alpha$.
\end{theorem}

\begin{proof}
By Theorem \ref{boom-thm}, there is a subsequence $h_k$ such that $u_{h_k,\kappa}$ converges uniformly on $\tir{\Omega}$ to a continuous function $v_{\kappa}$. The latter is convex by Lemma \ref{new-03}. \Qed
\end{proof}


As the family $v_{\kappa}$ consists of convex functions with uniformly bounded gradient, we can extract a subsequence which converges uniformly on $\tir{\Omega}$ to a convex function $v$ as $\kappa \to + \infty$. 

\begin{theorem} \label{final-thm-03}
Let $\kappa=n$ and assume that
 $v_n$ is a $\kappa$-viscosity solution of \eqref{monge} which converges uniformly on $\tir{\Omega}$ to a convex function $v$ as $n \to +\infty$. Then $v$ is a viscosity solution of \eqref{monge}.
\end{theorem}

\begin{proof} The proof is the same as the proof of stability of viscosity solutions under uniform convergence. Let $\phi$ be a strictly convex quadratic polynomial. We may assume that $\phi(x) = 1/2 \, x^T M x$ for a symmetric positive definite matrix $M$, since for a linear function $L(x)$, $\det D^2 L(x)=0$. Assume that $M$ has condition number $n_0$.

Let $x_0 \in \Omega$ and assume that $v-\phi$ has a maximum in the closed ball $B(x_0,\delta)$. Using $\phi(x) + |x-x_0|^2$, we may assume that $v-\phi$ has a strict local maximum in $B(x_0,\delta)$. By \cite[Lemma 2.4]{Bardi97}, since $v_n-\phi$ converges uniformly on $\tir{\Omega}$ to $v-\phi$, there exists a sequence $x_n \in \Omega$ such that $x_n \to x_0$ and $v_n(x_n) -\phi(x_n)\geq v_n(x) -\phi(x)$ for all $x$ in $B(x_0,\delta)$. 

We get $R(D \phi(x_n)) \det D^2 \phi (x_n) \geq f(x_n)$  for $n \geq n_0$ and thus 
$R(D \phi(x_0))$  $\det D^2 \phi (x_0) \geq f(x_0)$.

The other condition in the definition of viscosity solution is proved similarly. \Qed
\end{proof}

We now summarize Theorems \ref{final-thm-01}--\ref{final-thm-03}.
\begin{theorem} Let $\Omega$ be a rectangle. 
Assume that $V=V_{\kappa} \cap V_{max}$ and the scheme \eqref{m2d3} is consistent for strictly convex quadratic polynomials with condition number less than $\kappa$. There is a subsequence $h_k$ such that the solution $u_{h_k,\kappa}$ of \eqref{m2d3}, with $u_{h_k}(x^1_{h_k})=\alpha$ for 
$x^1_{h_k}$ in $\Omega_{h_k}$ 
and $x^1_{h_k} \to x^1 \in \tir{\Omega}$, converges uniformly on $\tir{\Omega}$ to a convex function $v_{\kappa}$. Moreover, as $\kappa \to + \infty$,  $v_{\kappa}$ converges uniformly on $\tir{\Omega}$ to the unique convex solution $u$ of \eqref{m1m} with $u(x^1)=\alpha$.
\end{theorem}

\begin{proof}
By Theorem \ref{final-thm-02}, there is a subsequence $h_k$ such that  $u_{h_k,\kappa}$ converges uniformly on $\tir{\Omega}$ to a continuous convex function $v_{\kappa}$ with $v_{\kappa}(x^1)=\alpha$. By Theorem \ref{final-thm-01}, $v_{\kappa}$ is a $\kappa$-viscosity solution of \eqref{monge} with $v_{\kappa}(x^1)=\alpha$. By Theorem \ref{final-thm-03}, as $\kappa \to + \infty$, $v_{\kappa}$ converges uniformly on $\tir{\Omega}$ to a convex function $v$ which is a viscosity solution of \eqref{monge} with $v(x^1)=\alpha$. Arguing as in Part 2 of the proof of Theorem \ref{cvg-thm}, the convex function $v$ has asymptotic cone $K$. Recall the equivalence of viscosity and Aleksandrov solutions from section \ref{equi-A-visc}. The convex function $v$ is then equal to the unique solution of \eqref{m1m}. All subsequences thus converge to the latter. This completes the proof. \Qed
\end{proof}

\begin{remark}
The requirement for convergence that solutions of \eqref{m2d3} are discrete convex can be removed when $f>0$ on $\Omega$ by using the viscosity solution reformulation of convexity.
\end{remark}

We finish this section by addressing consistency for quadratic polynomials. We first review a topic which is curiously called geometry of numbers.

\subsubsection{Geometry of numbers}
The material for this section is adapted from \cite{coppel2009number} to which the reader is referred to for additional details. 

Recall that $\{ \, r_1,\ldots,r_d \, \}$ denotes the canonical basis of $\R^d$. Here $r_i \in \mathbb{Z}^d$ for all $i$. We view $ \mathbb{Z}^d$ as a lattice, i.e.
$$
 \mathbb{Z}^d = \{ \, \zeta_1 r_1 + \ldots + \zeta_d r_d, \zeta_i \in \mathbb{Z} \text{ for all } i
 \, \}.
$$
The determinant $d(\mathbb{Z}^d)$ of the matrix with column vectors $r_i, i=1,\ldots,d$ is independent of the choice of the basis and called determinant of the lattice. We have $d(\mathbb{Z}^d)=1$.

Let $M$ be a symmetric positive definite matrix and consider the distance on $\R^d$ given by $d_M(e,e')=||e-e'||_M$ where 
$|| v ||_M :=\sqrt{v^T M v} $. For $e_0 \in  \mathbb{Z}^d$, we define the Voronoi cell
$$
\Vor (e_0) = \{ \, p \in \R^d, ||p-e_0||_M \leq ||p-e||_M, \forall e \in \Z^d 
\, \}.
$$
We denote by $\inte \Vor(e_0)$ the interior of $\Vor(e_0)$. 
It can be shown \cite[p. 343]{coppel2009number} that (recall that $\Z^d$ is infinite)
\begin{align} \label{voro}
\begin{split}
\R^d & = \cup_{e \in \Z^d} \Vor(e) \\
\inte \Vor(e) \cap \inte \Vor(e') & = \emptyset, \text{ for } e, e' \in \Z^d, e \neq e'.
\end{split}
\end{align}
We also note that for $e_0 \in \Z^d$
$$
\Vor(e_0) = \Vor(0) + e_0,
$$
i.e. $\Vor(e_0)$ is a $\Z^d$-translate of $\Vor(0)$. In the terminology of \cite[p. 337]{coppel2009number}, \eqref{voro} says that the $\Z^d$-translates of $\Vor(0)$ form a tiling of $\R^d$. By \cite[Proposition 11 Chapter VIII]{coppel2009number},
\begin{equation} \label{co-volume}
|\Vor (0)| = d(\Z^d)=1.
\end{equation}
We consider the open half-space
$$
G_e = \{ \, p \in \R^d ||p||_M < ||p-e||_M
\, \},
$$
and the hyperplane 
$$
H_e = \{ \, p \in \R^d ||p||_M = ||p-e||_M
\, \}.
$$
We have $\tir{G_e} = G_e \cup H_e$ and \cite[p. 342--343]{coppel2009number}
$$
\Vor(0) = \cap_{e \in \Z^d \setminus \{ \, 0 \, \}} \tir{G_e}. 
$$
In fact, there are a finite number of points $e_i \in \Z^d, i=1,\ldots, l$ such that
$$
\Vor(0) = \cap_{i=1}^l \tir{G_{e_i}},
$$
with the above representation irredundant, in the sense that it no longer holds if one omits one of the half-spaces $G_{e_i}$. 

Note that $\Vor(0)$ is convex, and recall that a subset $A$ of $\Vor(0)$ is a face of $\Vor(0)$ if $A$ is convex and if $y, y' \in \Vor(0)$ and the open line segment $(y,y')$ intersects $\Vor(0)$, then $y, y' \in \Vor(0)$. The $(d-1)$-dimensional faces of $\Vor(0)$ are called facets. The distinct facets of $\Vor(0)$ are given by the intersections $\Vor(0) \cap H_{e_i}, i=1, \ldots, l$
and the vectors $e_i, i=1, \ldots, l$ are the facets vectors of the lattice $\Z^d$. 

The notions introduced above are dependent on the distance $d_M$ induced by the symmetric positive definite matrix $M$. In \cite{Mirebeau15,neilan2019monge}, the facets of the Voronoi cell are called Voronoi facets and the facets vectors are called strict $M$-Voronoi vectors. $M$-Voronoi vectors are the vectors $e \in \Z^d$ for which $\Vor(0) \cap H_e \neq \emptyset$. 
Equivalently
$$
 \Vor (0) = \{ \, p \in \R^d, 2 (M p) \cdot e \leq e^T M e, \forall e \in \Z^d \, \}. 
$$ 

\subsubsection{Interior consistency for strictly convex quadratic polynomials} \label{consistency}

For a set $S$, $h S =\{ \, h \, x, x \in S\, \}$ and $M S = \{ \, M x, x \in S\, \}$. We note that the definition of $\partial_V q (x)$ uses the values of the quadratic function $q$ only when $x \in \Omega_h$. For $x \notin \Omega_h$, the discrete extension formula \eqref{extension} is used. 

In this section, we take $V=V_{\kappa} \cap V_{max}$. 
The results of this section are needed at mesh points at a distance $C h$ of $\partial \Omega$, c.f. 
the proof of Theorem \ref{final-thm-01} 
below. For those mesh points $V_{\kappa} \cap V_{max}=V_{\kappa}$. We therefore assume that the stencil $V$ is mesh independent in the statement of the results below. 

\begin{lemma} \label{pre-interior}
Let $M$ be a symmetric positive definite $d \times d$ matrix and $q(x) = 1/2 \, x^T M x$ a quadratic polynomial. We have for all $x \in \Omega_h$ such that $x+ h e \in \Omega_h$ for all $e \in V$
$$
|\partial_V q (x) | = h^d \det(M) |  \Vor (M,V)|, 
$$
where $\Vor (M,V)$ is the Voronoi cell of $M$ associated with the stencil $V$, i.e.
$$
 \Vor (M,V) = \{ \, p \in \R^d, 2 (M p) \cdot e \leq e^T M e, \forall e \in V
 \, \}. 
$$ 
\end{lemma}

\begin{proof} 
We have
$$
q(x + h e) = \frac{1}{2} (x + h e)^T M (x + h e) = q(x)+ h \, x^T M x + \frac{h^2}{2} e^T M e.
$$
Thus $\partial_V q (x)$ is equal to
\begin{multline*}
\{ \, p \in \R^d, p \cdot e \leq x^T M e +  \frac{h}{2} e^T M e, \forall e \in V
 \, \} 
  =  \{ \, p \in \R^d, (p-M x) \cdot e \leq    \frac{h}{2} e^T M e, \\ \forall e \in V
  \, \} 
  =  \{ \, h q \in \R^d, 2 (q- \frac{1}{h}M x) \cdot e \leq    e^T M e, \forall e \in V
  \, \}  
  = h \{ \, q \in \R^d, \\2 (q- \frac{1}{h}M x) \cdot e \leq    e^T M e, \forall e \in V
  \, \} 
  = h M \{ \, r \in \R^d, 2 (M r- \frac{1}{h}M x) \cdot e \leq    e^T M e, \forall e \in V
   \}.     
\end{multline*}
But the set $\{ \, r \in \R^d, 2 (M r- 1/h \, M x) \cdot e \leq    e^T M e, \forall e \in V  \, \}$ is a translate of $ \Vor (M,V)$ by $1/h \, M x$, and thus they have the same volume. The result then follows. \Qed
\end{proof}
We next give sufficient conditions on $V$ so that $| \Vor (M,V)| = 1$
so that consistency holds for strictly convex quadratic polynomials. 

\begin{lemma} \label{meas-q}
Let $M$ be a symmetric positive definite $d \times d$ matrix. 
If the stencil $V$ contains all strict $M$-Voronoi vectors, then $| \Vor (M,V)| = 1$. Therefore, for $q(x) = 1/2 \, x^T M x$ and 
$x \in \Omega_h$ such that $x+ h e \in  \Omega_h$ for all $e \in V$ we have
$$
|\partial_V q (x) | = h^d \det M.
$$
\end{lemma}

\begin{proof}
We show that under the conditions of the lemma we have $ \Vor (M,V) = \Vor(0)$. The result then follows from Lemma \ref{pre-interior} and \eqref{co-volume}.

We have from the definitions $\Vor(0) \subset  \Vor (M,V)$. Let $S$ be the set of strict $M$-Voronoi vectors. We have
$$
\Vor (0) = \{ \, p \in \R^d, 2 (M p) \cdot e \leq e^T M e, \forall e \in S \, \}. 
$$
If $S \subset V$, we get $\Vor (M,V) \subset \Vor (0)$. The result then follows. \Qed
\end{proof}

The following characterization of the set of all strict $M$-Voronoi vectors was given in \cite{Mirebeau15,neilan2019monge}.
\begin{lemma} \label{condition-number}
Let $M$ be a symmetric positive definite $d \times d$ matrix 
and let $\kappa = \sqrt{|M| \, |M^{-1}|}$. 
Then all strict $M$-Voronoi vectors are contained in the set
\begin{equation} \label{S}
S=\{ \, e \in \Z^d, |e| \leq \frac{1}{2} \sqrt{d} \kappa, e \text{ has co-prime coordinates}
\, \}.
\end{equation}
\end{lemma}


\subsubsection{Proof of Theorem \ref{final-thm-01} } \label{proof-of-final-thm-01}
Recall the half-relaxed limits defined for $x \in \tir{\Omega}$ by
\begin{align*}
v^*(x) = \limsup_{y \to x, h \to 0} u_{h,\kappa}(y) & = \lim_{\delta \to 0} \sup\{ \, u_{h,\kappa}(y), y \in \Omega_h, |y-x| \leq \delta, 0<h\leq \delta \, \} \\
 v_*(x) = \liminf_{y \to x, h \to 0} u_{h,\kappa}(y) & = \lim_{\delta \to 0} \inf \{ \, u_{h,\kappa}(y), y \in \Omega_h, |y-x| \leq \delta, 0<h\leq \delta \, \}.
\end{align*}
By construction $v_{\kappa}$ is the uniform limit of continuous functions which interpolate $u_{h,\kappa}$ and hence $v_{\kappa} \in C(\tir{\Omega})$. Since $u_{h,\kappa}$
converges uniformly on $\tir{\Omega}$ to $v_{\kappa}$, we have $v_{\kappa}=u^*  = u_*$ on $\tir{\Omega}$. At this point, it is not known yet that the limit convex function $v_{\kappa}$ is a viscosity solution of \eqref{monge}.

We show that $v_{\kappa}=u_*$ is a $\kappa$-viscosity super solution of $R(D u(x))\det D^2 u(x)=f(x)$ at every  point $x$ of $\Omega$. Let $x_0 \in \Omega$ and $\phi$ be a strictly convex quadratic polynomial with condition number less than $\kappa$
such that $v_*-\phi$ has a local minimum at $x_0$ with $(v_*-\phi)(x_0)=0$. Without loss of generality, we may assume that $x_0$ is a strict local minimum. 

Let $B_0$ denote a closed ball contained in $\Omega$ and containing $x_0$ in its interior.  We let $x_{h_l}$ be a subsequence in $B_0$ such that $x_{h_l} \to x_0$ with $u_{h_l}(x_{h_l}) \to v_*(x_0)$. As $h_l \to 0$, we may assume that for all $x \in B_0$,
$d(x,\partial \Omega) >  h_l \sqrt{d} \kappa$.  
 If $e \in V_{\kappa}$, $|e| \leq 1/2 \sqrt{d} \kappa$ by definition and thus $|h_l e| <  h_l \sqrt{d} \kappa$. We conclude that for $x \in B_0$, we have $x+h e \in \Omega$ and hence $x+h e \in \Omega_h$ for all $e \in V_{\kappa}$. Therefore $V_{\kappa} \cap V_{max}(x) = V_{\kappa}$ for all $x \in B_0$. 
 
 Let $x'_l \in B_0 \cap \Omega_{h_l}$ be defined by
$$
c_l  \coloneqq  (u_{h_l} - \phi)(x'_l) = \min_{B_0} u_{h_l} - \phi.
$$
Since the sequence $x'_l$ is bounded, it converges to some $x_1$ after possibly passing to a subsequence. Since $(u_{h_l}-\phi)(x'_l) \leq (u_{h_l}-\phi)(x_{h_l})$ we have
$$
(v_*-\phi)(x_0) = \lim_{l \to \infty} (u_{h_l}-\phi)(x_{h_l}) \geq \liminf_{l \to \infty} (u_{h_l}-\phi)(x'_l) \geq (v_*-\phi)(x_1).
$$
Since $x_0$ is a strict minimizer of the difference $v_*-\phi$, we conclude that $x_0=x_1$ and $c_l \to 0$ as $l \to \infty$. By definition
$$
u_{h_l}(x) \geq \phi(x) + c_l, \forall x \in B_0 \cap \Omega_{h_l},
$$
with equality at $x=x'_l$, and thus, by the monotonicity of the scheme
\begin{multline*}
0= \frac{1}{h_l^d} \omega(R, u_{h_l} , \{ \, x'_l\, \}) 
 - f(x'_l) 
\geq \frac{1}{h_l^d} \omega(R, \phi+c_l  , \{ \, x'_l\, \}) 
 - f(x'_l) \\
 = \frac{1}{h_l^d} \omega(R, \phi  , \{ \, x'_l\, \})  - f(x'_l),
\end{multline*}
which gives by the consistency of the scheme $R(D \phi(x_0))  \det D^2 \phi(x_0) - f(x_0) \leq 0$.

Similarly one shows that if $\phi $ is a strictly convex quadratic polynomial with condition number less than $\kappa$ such that $v^*-\phi$ has a local maximum at $x_0$ with $(v^*-\phi)(x_0)=0$, we have $R(D \phi(x_0)) \det D^2 \phi(x_0) - f(x_0) \geq 0$. It follows that $v_{\kappa}=u^*  = u_*$ on $\Omega$ is a $\kappa$-viscosity solution of $R(D u) \det D^2 u=f$. 
 

\subsection{Polygonal approximations of $\Omega^*$ } \label{poly-sec}

We now address the convergence of solutions of \eqref{m2} to the solution of \eqref{m1m} as $Y \to \tir{\Omega^*}$. 
Recall that $\tilde{f}$ as defined by \eqref{modified-f} depends on $Y$. Here we make the dependence explicit. 
Put
$
 f_{Y}(t) =\tilde{f}(t).
$ 

The distance of the point $x$ to the set $K$ is denoted $d(x,K)$. The Hausdorff distance $d(K,H)$ between two nonempty subsets $K$ and $H$ of $\R^d$ is defined as
$$
\max \{ \, \sup [d(x,K), x \in H], \sup [d(x,H), x \in K]\, \}.
$$

We say that a sequence of domains $\Omega_m$ is 
increasing to $\Omega$, if $\Omega_m \subset \Omega_{m+1} \subset  \Omega$ and $d(\partial \Omega_m, \partial \Omega) \to 0$ as $m \to \infty$. 

\begin{theorem} \label{proc2} Let $Y_m$ be bounded non degenerate convex polygonal domains 
increasing to $\tir{\Omega^*}$. 
Then the convex solution $u_m$ of
\begin{align} \label{m1m*}
\begin{split}
\omega(R,u,E) &= \int_E f_{Y_m}(x) d x \text{ for all Borel sets } E \subset \tir{\Omega}\\
\chi_u (\tir{\Omega}) & = Y_m \\
u(x^0) & = \alpha,
\end{split}
\end{align}
for $x^0 \in \Omega$ and $\alpha \in \R$ converges uniformly on 
$\tir{\Omega}$ to the solution $u$ of \eqref{m1m} with $u(x^0)  = \alpha$. 

\end{theorem}

\begin{proof}
Recall that $f_{Y_m}(x) = f(x) - \epsilon_m^* f(x)$ where 
$
\epsilon_m^* =\int_{\Omega^*\setminus Y_m} R(p) d p \bigg/ \int_{\Omega} f(x) d x
$. As $Y_m \to \tir{\Omega^*}$, $\epsilon_m^* \to 0$. Thus $\int_E f_{Y_m}(x) d x 
\to \int_E f(x) d x$ for all Borel sets $E \subset \tir{\Omega}$ with $|\partial E|=0$. 
For the purpose of using results on Monge-Amp\`ere equations stated for bounded domains in \cite{Guti'errez2001}, 
we may assume that the Borel sets $E \subset \tir{\Omega}$ are contained in a larger bounded domain $\hat{\Omega}$ such that $ \tir{\Omega} \subset U \subset \hat{\Omega}$ for an open set $U$, and set
$f_{Y_m}(x)=0$ and $f(x)=0$ outside $\Omega$. 

Recall that $\Omega^*$ is bounded. Let $C$ such that $|p| \leq C, \forall p \in \Omega^*$. We claim that the functions $u_m$ are Lipschitz continuous with the same Lipschitz constant. The proof is analogous to the one for \cite[Lemma 1.1.6]{Guti'errez2001}. Essentially because $\chi_{u_m} (\tir{\Omega}) \subset \tir{\Omega^*}$ for all $m$. Thus for all $x, y \in \tir{\Omega}$, we have for a constant $C$ independent of $m$
$$
|u_m(x) - u_m(y)| \leq C ||x-y||_1.
$$
Moreover since $u_m(x^0)=\alpha$ and $\hat{\Omega}$ is bounded, we conclude that the sequence $u_m$ is uniformly bounded and equicontinuous on $\tir{\Omega}$. By the Arzela-Ascoli theorem, there is a subsequence also denoted $u_{m}$ which converges uniformly on the compact set $\tir{\Omega}$ to a function $v$ on $\tir{\Omega}$. It is known that such a function $v$ is convex. By the weak convergence of R-curvatures \cite[Theorem 9.1]{Bakelman1994}, $\omega(R,u_m,.)$ weakly converges to $\omega(R,v,.)$ 
We conclude that $\omega(R,v,E) = \int_E f(x) d x$ for all Borel sets $E\subset \Omega$.

Next we show that $\chi_v (\tir{\Omega})  = \tir{\Omega^*}$. Let $p \in  \Omega^*$. There exists a sequence $p_m \in  Y_m$ such that $p_m \to p$ in $\R^d$, see for example \cite[Theorem 1.8.8-a]{Schneider14}. Therefore there exists $x^m \in \tir{\Omega}$ such that $p_m \in \chi_{u_{m}} (x^m)$, i.e.
$$
u_{m}(y) \geq u_{m}(x^m) + p_m \cdot (y-x^m) \, \forall y \in \R^d.
$$
The bounded sequence $x^m$ converges up to a subsequence to a point $x \in \tir{\Omega}$. We conclude that
$v(y) \geq v(x) + p \cdot(y-x)$ for all $y \in \tilde{\Omega}$. Thus $p \in \chi_v (\tir{\Omega})$ and $ \Omega^* \subset
\chi_v (\tir{\Omega})$. A similar argument shows that $\chi_v (\tir{\Omega})$ is closed. Therefore
$\tir{\Omega^*} \subset \chi_v (\tir{\Omega})$. Using \eqref{necessary}
\begin{multline*}
\int_{\chi_{v} (\tir{\Omega}) } R(p) d p = \omega(R,v,\tir{\Omega}) = \int_{\tir{\Omega}} f(x) d x
= \int_{\Omega} f(x) d x= \int_{\Omega^*} R(p) d p \\
=\int_{\tir{\Omega^*}} R(p) d p.
\end{multline*}
Therefore $|\chi_{v} (\tir{\Omega}) \setminus \tir{\Omega^*}|=0$. We conclude that $\tir{\Omega^*}$ is dense in $\chi_{v} (\tir{\Omega})$. But 
$\tir{\Omega^*}$ is closed. Thus $\chi_{v} (\tir{\Omega}) = \tir{\Omega^*}$.

Moreover, if $K$ is compact and $U$ is open such that $K \subset U \subset \tir{U} \subset \Omega$, we have up to a set of measure 0, $ \chi_v(K) \subset \liminf_{m \to \infty} \chi_{u_{m}}(U)$, by \cite[Lemma 1.2.2]{Guti'errez2001}. This implies $\chi_v(\Omega) \subset \Omega^*$. As in the proof of Part 3 of Theorem \ref{cvg-thm}, using \cite[Lemma 1.1.12]{Guti'errez2001} which says that the set of points which are in the normal image of more than one point is contained in a set of measure 0, we obtain $|\chi_v(\partial \Omega)| = 0$. So we actually have $\omega(R,v,E) = \int_E f(x) d x$ for all Borel sets $E\subset \tir{\Omega}$.

Clearly $v(x^0) = \alpha$ and so $v$ is the unique solution of \eqref{m1m} which satisfies $v(x^0) = \alpha$. It follows that the whole sequence $u_m$ converges uniformly to $u$ on $\tir{\Omega}$. \Qed
\end{proof}

\section{The degenerate case $f\geq 0$} \label{degenrate-case}

For the uniqueness of a solution, we needed the assumption $f>0$. 
In the case $f\geq 0$, from an implementation point of view, 
and for the existence of a solution, 
we may consider the approximate problem analogous to \eqref{m2d3}
\begin{align} \label{m2d4}
\begin{split}
\omega_a(R,u_h^{\epsilon},\{\, x \,\})&= \int_{E_x} \tilde{f}(t) dt + \epsilon |E_x| , x \in \Omega_h,
\end{split}
\end{align}
where $\epsilon>0$ is taken close to machine precision and for a polygon $Y$ we choose $Y_{\epsilon}$ such that $Y \subset Y_{\epsilon}$ and the compatibility condition 
\begin{align*}
\sum_{x \in \Omega_h} \omega_a(R,u_h^{\epsilon},\{\, x \,\}) = \int_{Y_{\epsilon}} R(p) d p. 
\end{align*} 
holds. Here $u_h^{\epsilon}$ is required to have asymptotic cone $K_{\epsilon}$ associated with $Y_{\epsilon}$. As $\epsilon \to 0$ $u_h^{\epsilon}$ converges to a solution $u_h$ of \eqref{m2d3} and $Y_{\epsilon} \to Y$. This proves existence of a solution to \eqref{m2d3} in the degenerate case $f \geq 0$. 

For the convergence of the discretization in the case $V=V_{max}$, i.e. the analogue of Theorem \ref{cvg-thm}, note that because of Lemmas \ref{pre-lip-lem} and \ref{inc-sub-lem}, the approximations are uniformly Lipschitz on $\tir{\Omega}$. It then remains to verify that $\chi_{\Gamma_1(u_h^{\epsilon})}( \mathcal{N}_h^1 )$ is uniformly bounded. But this is also an immediate consequence of Lemma \ref{key-lem}.

We have for all $p \in \chi_{\Gamma_1(u_h^{\epsilon})}( \mathcal{N}_h^1 )$
$$
||p|| \leq C C_{Y_{\epsilon}} C_{\Omega}.
$$
Since $Y_{\epsilon} \to Y$ as $\epsilon \to 0$, the result follows. For a subsequence $h_k$, $u_{h_k}^{\epsilon}$ converges uniformly on $\tir{\Omega}$ to a convex function $v^{\epsilon}$. The latter can be shown to converge to a solution of \eqref{m1m} using the arguments of section \ref{poly-sec}. 

We note that the convergence argument to a viscosity solution of section \ref{cvg-visc} do not require $f>0$. 



\section{Numerical experiments} \label{Numerical}

For the implementation of the numerical method \eqref{m2d3}, note that the set $\partial_V v_h(x)$, for a mesh point $x$, is a polygon defined by a finite number of inequalities. There are programs available on MATLAB Central which allow to compute the vertices of a polygon from the defining inequalities. In our MATLAB implementation, we found the vertices of $\partial_V v_h(x)$ by 
parameterizing its edges using the linear inequalities.  Numerical integration over a triangulation of the polygon can then be used to compute $\omega_V(R,v_h,\{\,x\,\})$ for $x \in \Omega_h$. Formulas for the Jacobian matrix are given in  \cite{Awanou-damped}. To deal with a possible singular Jacobian, as in \cite{Benamou2014}, we added a small constant to the diagonal elements. 
The parameters $\delta$ and $\rho$ in the damped Newton's method \cite{Awanou-damped} were taken as $\rho=1$ and $\delta=1/2$.

We give numerical experiments for $d=2$ and $\Omega=(0,1)^2$. Here $\Omega_h=\Omega \cap (a+\mathbb{Z}_h^2)$ where $a=(1/2, 1/2)$. 
For integration over edges, for the entries of the Jacobian matrix, we used a Gaussian quadrature rule with degree of precision 7. For the right hand side, a three point quadrature rule with degree of precision 2 was used. The stencil $V$ was taken as $V=-V_1 \cup V_1$ where $V_1$ consists of the vectors 
$(1,0), (0,1), (1,1), (1,-1), (2,1)$, $(-1,2), (1,2)$ and $(-2,1)$. For the imposition of the constraint $v_h(x^1)=0$, we approximate the solution of the equation
$R(D u) \det D^2 u=f+ u(x^1)$. The compatibility condition \eqref{necessary} implies that $u(x^1)=0$. In our experiment we used $x^1=a+(h,h)$. 

The discrete convexity assumption was not enforced.  Starting with an initial guess which is discrete convex, we require that subsequent iterates are $V$-discrete convex by choosing the step size in the damped Newton's method. 

Note however that since we are using in \eqref{m2d3}  the approximation $\int_{E_x} f(t) dt \approx h^2 f(x)$ and numerical integration for the evaluation of $\omega_V(R,u_h,\{\,x\,\})$ for $x \in \Omega_h$, the discrete mass conservation \eqref{mass-conservation} will not hold, i.e. $\sum_{x \in \Omega_h}  \omega_V(R,u_h,\{\,x\,\}) 
\neq \sum_{x \in \Omega_h} h^2 f(x)$. A discrete solution with some value of $u_h(x^1)$ is computed 
and we add a constant $c$ to have $u_h(x^1)+c=0$. 
Alternatively, to assure a discrete mass conservation, one could also consider, for a constant $c$ to be adjusted, 
$ \omega_V(R,u_h,\{\,x\,\})  = h^2 f(x)+ c \sum_{x \in \Omega_h} u_h (x)$. This approach naturally requires adding a small constant to the diagonal elements of the Jacobian matrix.

First we consider the exact solution $u(x,y)=x^2/2+ xy + y^2$. In this case $\tir{\Omega^*}$ is the polygon of area 1 with vertices $(0,0), (1,1), (1,2)$ and $(2,3)$. We take $R(x,y)=x+y$ with corresponding right hand side $f(x,y)$. 
As in \cite{Prins2015} we take as initial guess a function $u^0$ such that $\chi_{u^0}(\tir{\Omega})$ is a rectangle contained in $\Omega^*$. 

Table 1 
shows an asymptotic quadratic convergence rate for $u$ while the convergence rate for $Du$ is linear. 
Figures 3 and 4 
show the deformations of a grid by the gradient mapping. 
Here, the initial guess was taken as $\alpha u^0$ where $u^0$ is a function  such that $\chi_{u^0}(\tir{\Omega})$ is a rectangle contained in $\Omega^*$ and $\alpha=\int_{\Omega^*} R(p) dp$.
For this case, unlike the results in \cite{benamou2017minimal}, there is no collapse of grid points near the boundary of the circle.

\begin{table}
\begin{tabular}{c|ccccc}  
 \multicolumn{6}{c}{$h$}\\
 &  $1/2^5$ &  $1/2^6$ & $1/2^7$ & $1/2^8$& $1/2^9$ \\
Error for $u$   & 2.72 $10^{-4}$ & 8.01 $10^{-5}$ & 2.31 $10^{-5}$& 6.52 $10^{-6}$& 1.82 $10^{-6}$ \\& & & & &  \\
Rate  & &1.76 & 1.79 & 1.82 & 1.84 \\
& & & & &  \\
Error for $D u$  & 6.27 $10^{-3}$ & 3.30 $10^{-3}$ & 1.56 $10^{-3}$& 8.23 $10^{-4}$& 3.92 $10^{-4}$ \\& & & & &  \\
Rate  & &0.93 & 1.07 & 0.93 & 1.07 \\
\end{tabular}  \label{tab} 
\caption{Maximum errors for a smooth solution.}
\end{table}


\begin{figure}[tbp]
\begin{center}
\includegraphics[angle=0, height=4.5cm]{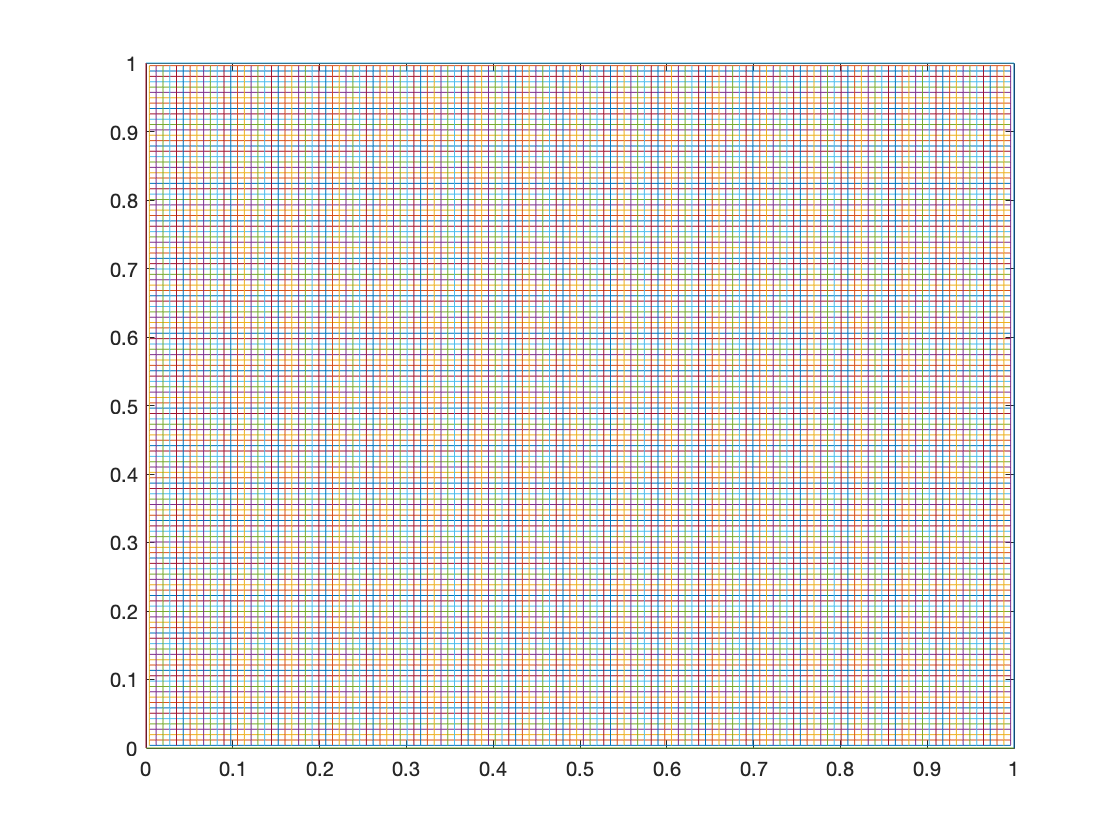}
\includegraphics[angle=0, height=4.5cm]{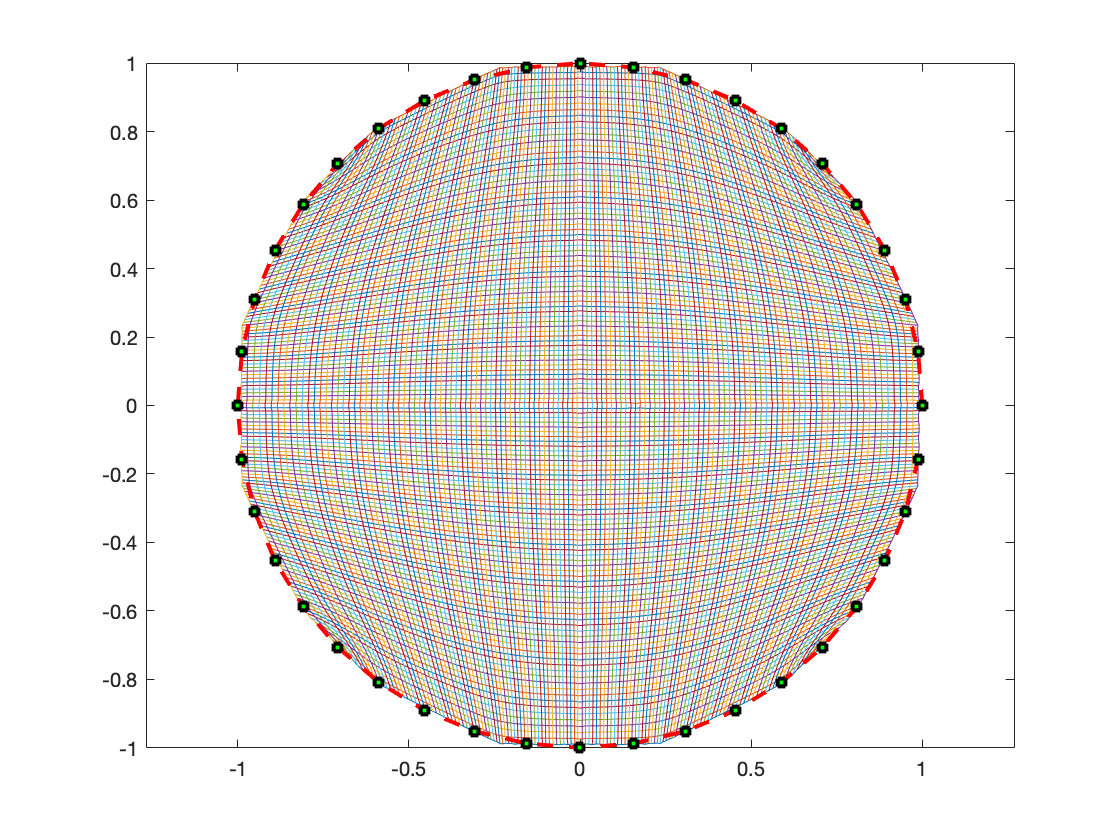}
\end{center}
\caption{Constant density on a square mapped to constant density on the unit disc $h=1/2^7$.}
\end{figure}

\begin{figure}[tbp]
\begin{center}
\includegraphics[angle=0, height=4.5cm]{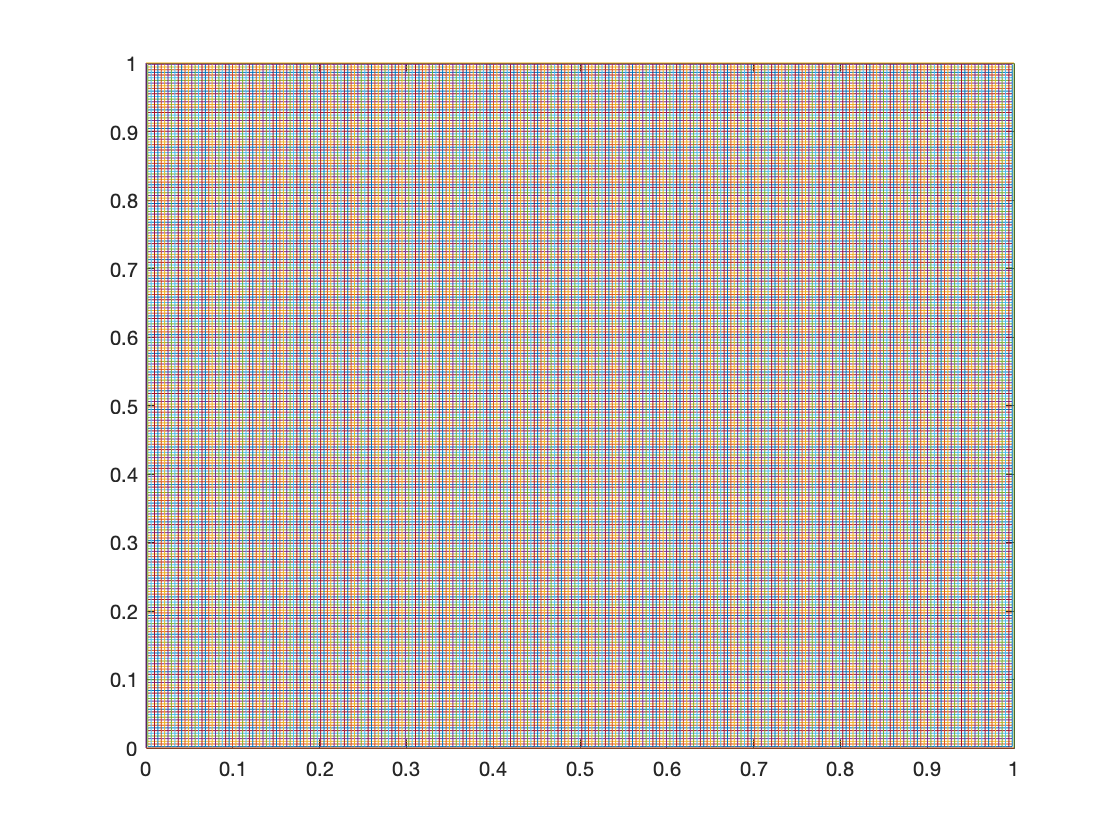}
\includegraphics[angle=0, height=4.5cm]{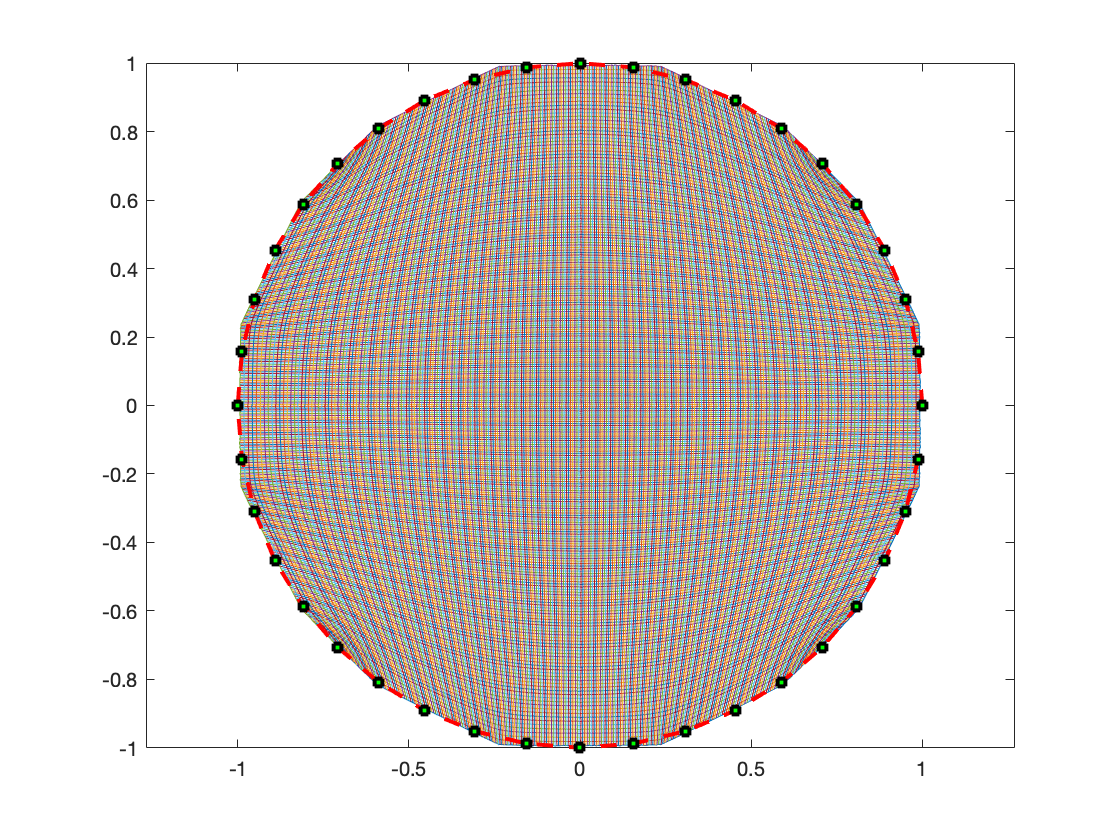}
\end{center}
\caption{Constant density on a square mapped to the Gaussian $e^{-0.5(x^2+y^2)}$ on the unit disc $h=1/2^8$.}
\end{figure}

\section{A review of polyhedral set theory}\label{comp-section}

The purpose of this section is to relate the notions introduced in section \ref{asym} to the standard polyhedral set theory. It may be skipped in a first reading. 

Any convex set which does not contain a line and consisting of the union of rays with the same common vertex is called a {\it convex cone}. The common vertex of all these rays is called the {\it vertex} of this convex cone. Formally

\begin{definition} \label{defn-cc}
A convex set $D \subset \R^{d+1}$ which does not contain a line is a convex cone with vertex $A$ if there is a subset $S$ of $\R^{d+1}$ such that
$D = \cup_{e \in S} L_{A,e}^+$.
\end{definition}

See Figures \ref{cone-fig} and \ref{pointed_cone-fig} for examples of convex cones.

\begin{lemma}
A convex set $D \subset \R^{d+1}$ which does not contain a line is a convex cone with vertex $A$, if and only if for $X \in D$, we have $A+\lambda \  \overrightarrow{A X} \in D$ for all $\lambda >0$.
\end{lemma}

\begin{proof} 
Assume that $D$ is a convex cone. 
Let $X \in D$ and $e \in \R^{d+1}$ such that $X \in L_{A,e}^+$. Let $\mu \geq 0$ such that 
$\overrightarrow{A X} = \mu e$, i.e. $X=A+ \mu e$. Then $B:=A+\lambda \  \overrightarrow{A X} = A + \lambda \mu e$
which means that $B \in L_{A,e}^+ \subset D$. 

Conversely, with $S=\{ \, e: e=  \overrightarrow{A X}, X \in D \, \}$, we have $D=\cup_{e \in S} L_{A,e}^+$. \Qed
\end{proof}

Let $M$ be a convex set which does not contain a line and let $A \in M$. The asymptotic cone $K_A(M)$ of $M$ is a convex cone. For another example, the epigraph of the function $k_{(p,\mu)}$ in \eqref{ex-cc}
is a convex cone in $\R^{d+1}$ with vertex $(p,\mu)$ (it is equal to its asymptotic cone by Lemma \ref{ex-as-cone}).

\begin{lemma} \label{onlyonev}
A convex cone has only one vertex.

\end{lemma} 

\begin{proof}
Assume that $D$ is a convex cone such that $D = \cup_{e \in S} L_{A,e}^+$ and $D = \cup_{e' \in S'} L_{B,e'}^+$ for subsets $S$ and $S'$ of $\R^{d+1}$ and vertices $A$ and $B$. Let $e' \in S'$ and $\mu \geq 0$ such that 
$A=B+\mu e'$. Let also $e \in S$ and $\lambda \geq 0$ such that $B=A + \lambda e$. We have $\lambda e + \mu e'=0$. If $\mu=0$ or $\lambda=0$, $A=B$. Otherwise $e'=-\lambda/\mu \, e$ and by assumption $ L_{B,e'}^+ \subset D$. But $ L_{B,e'}^+=  L_{B, \mu/\lambda \, e'}^+ =  L_{B,-e'}^+$. Thus $D$ contains the line with direction $e'$. Recall that by Definition \ref{defn-cc} a convex cone does not contain a lime. 
Contradiction. \Qed
\end{proof}

Let $D$ be a convex cone with vertex $A$ and put $D=A+K$ where $K$ is a convex cone with vertex at the origin. 
The condition $K \cap -K = \{ \, O \, \}$ is equivalent to requiring that $K$ does not contain a line. 
A convex cone as defined above is also refereed to as pointed convex cone \cite[p. 2]{Aliprantis}. In other words, the convex cone $A+K$ is pointed in the sense that $K \cap -K = \{ \, O \, \}$. We restrict to this class of convex cones because of the applications considered. We are interested in convex functions on $\R^{d}$ whose graphs form the boundary of the Minkowski sum of a convex cone and the convex hull of a set of points. 
Note that the epigraph of such a convex function do not contain a line. 
See Figure \ref{cone-fig} for the graph of a piecewise linear convex function which is the boundary of the Minkowski sum of a convex cone and the convex hull of a set of points.

Following \cite{Bakelman1994}, the points $X_0, X_1, \ldots, X_k$ are in general position if the vectors $\overrightarrow{X_0 X_1}, \ldots, \overrightarrow{X_0 X_k}$ are linearly independent. 
The points $X_0, X_1, \ldots, X_k$ are thus necessarily distinct. If $k>d$  they cannot be in general position.

We shall say that a set $S \subset \R^{d+1}$ is $k$-dimensional $0 \leq k \leq d+1$ if it contains $k+1$ points in general position but does not contain $k+2$ points in general position. 
A hyperplane in $\R^{d+1}$ is a $d$-dimensional set of the form 
$\{ \, x:  x \in \R^{d+1},  a^* \cdot x = b^*_< \, \}$ for $a^* \in \R^{d+1}, a^* \neq 0$ and $b^*_< \in \R$. By a closed half-space in $\R^{d+1}$, we mean a set of the form $\{ \, x:  x \in \R^{d+1},  a^* \cdot x \geq b^*_< \, \}$ for $a^* \in \R^{d+1}, a^* \neq 0$ and $b^*_< \in \R$.

A $k$-convex polyhedron $P$ is a $k$-dimensional set which is the intersection of a finite number of closed half-spaces, 
$$
P= \{ \, x: x \in \R^{d+1}, A^* x \geq b^*\, \},
$$
where $A^*$ is a $m \times(d+1)$ matrix and $b^* \in  \R^{d+1}$. 

The hyperplane $F=\{ \, x:  x \in \R^{d+1},  a^* \cdot x = b^*_< \, \}$ is a supporting hyperplane to the convex polyhedron $P$ if $P \subset \{ \, x:  x \in \R^{d+1},  a^* \cdot x \geq b^*_< \, \}$, i.e. $P$ is contained in (one of) the closed half-space with boundary $F$, and $F$ contains one or more points of $P$. 

A face of a convex polyhedron $P$ is a non-empty intersection of $P$ with one or more supporting hyperplanes. If a face of $P$ has dimension $k$, i.e. it is a $k$-dimensional set, it is called a $k$-face. The $0$-faces and 1-faces of $P$ are called vertices and edges of $P$ if they exist. 

A polyhedral angle, also called pointed polyhedral cone using the terminology of \cite{auslender2006asymptotic}, is a convex cone which is a convex polyhedron. Recall that by our convention a convex cone does not contain a line and hence has only one vertex by Lemma \ref{onlyonev}. A polyhedral angle can be written as $A+K$ where $A \in \R^{d+1}$ and 
$$
K = \{ \, x:  x \in \R^{d+1}, A^* x \geq 0\, \},
$$
for a  $m \times(d+1)$ matrix $A^*$ of rank $d+1$. If we let $a^*_i, i=1,\ldots,m$ denote the rows of $A^*$, the rank condition ensures that the origin is the only point in the intersection of the half-spaces $\{ \,x: x \in \R^{d+1}, a^*_i \cdot x \geq 0, i=1,\ldots,m \, \}$. This implies that the polyhedral angle has only one vertex $A$. See also \cite[Proposition 4.29]{paffenholz2010polyhedral}.

We now state some results of basic polyhedral theory c.f. for example \cite{paffenholz2010polyhedral}. The particular results used in this paper (Lemma \ref{conv-hull-union}  and Theorem  \ref{as-poly} ) were proved above. 

The asymptotic cone of an unbounded convex polyhedron which does not contain a line is a polyhedral angle, i.e. if $P$ is unbounded of the form $P=\{ \, x: x \in \R^{d+1}, A^* x \geq b^*\, \}$ with $A^*$ of rank $d+1$, then $P$ has asymptotic cone $A+K$ where $A \in P$ and $K=\{ \, x:  x \in \R^{d+1}, A^* x \geq 0\, \}$. The set $K$ is also known as recession cone or characteristic cone of $P$ \cite[Proposition 2.15]{paffenholz2010polyhedral}. In fact $P=S+K$ where $S$ is the convex hull of a finite number of points  \cite[Theorem 2.8 and Proposition 2.15]{paffenholz2010polyhedral}.

An extreme ray of a polyhedron $P$ is a ray which is a face of $P$. 
Klee, \cite{Klee} or \cite[Theorem 3.6.14]{Stoer}, proved that a polyhedron which does not contain a line is the convex hull of its vertices and its extreme rays. See also \cite[Theorem 1.4.3]{Schneider14}. The above decomposition $P=S+K$ of a line free polyhedron also follows  \cite[Corollary 1.4.4]{Schneider14}, using the observation that a point on an extreme ray is the sum of a vertex of $P$ and and an element of its recession cone $K$. 
A similar result is the following theorem by Bakelman who gave a simple geometric proof.

\begin{theorem}\cite[Theorem 4.2]{Bakelman1994} \label{4.2}
Every unbounded convex polyhedron which does not contain a line is the convex hull of its vertices and its asymptotic convex polyhedral angle, which is placed at one of its vertices.
\end{theorem}

In this paper we are interested in a particular kind of polyhedral angle. Let us illustrate how Lemma \ref{ex-as-cone} follows from polyhedral theory. 

Let $Y \subset \R^d$ be a $d$-convex polygon with vertices $a_1^*, a_2^*, \ldots, a_{N^*}^*$. This implies that 
$\{ \, a_1^*, a_2^*, \ldots, a_{N^*}^* \, \}$ is $d$-dimensional, i.e. it contains $d+1$ vectors in general position. Thus the matrix with columns $a_i^*-a_1^*, i=2,\ldots,N^*$ has rank $d$. It follows that the $N^* \times (d+1)$ matrix $A^*$ with rows
$\begin{pmatrix}(a_i^*)^T & -1 \end{pmatrix}$ has rank $d+1$. For the purpose of matrix multiplication, elements of $\R^d$ are column vectors. For simplicity below, if no matrix multiplication is involved, an element of $\R^d$ is a $d$-tuple.

The graph of the linear function $x \mapsto a_i^* \cdot x$ on $\R^d$, $\{ \, (x,x_{d+1}): (x,x_{d+1}) \in \R^d \times \R, 
 x_{d+1}=a_i^* \cdot x\, \}$ is a hyperplane of the form 
  $\{ \, (x,x_{d+1}): (x,x_{d+1}) \in \R^d \times \R, 
 (x,x_{d+1}) \cdot (a_i^*,-1) = 0  \, \}$. The closed half-space $\{ \, (x,x_{d+1}): (x,x_{d+1}) \in \R^d \times \R, 
 (x,x_{d+1}) \cdot (a_i^*,-1) \geq  0  \, \}$ is the epigraph of the linear function $x_{d+1}=a_i^* \cdot x$. 
 
 The convex cone $K=K_{(0,0)}$ introduced above and associated with the polygon $Y$ is the convex  cone 
 $\{ \, 
y= \begin{pmatrix}(x)^T & x_{d+1} \end{pmatrix}^T:  y \in \R^d \times \R, A^*y \geq 0
 \, \}$. It is equal to its recession cone. Thus the epigraph of $k_{(0,0)}$ is a convex cone equal to its asymptotic cone. 
 
 Lemma \ref{conv-hull-union} is just a special case of Bakelman's theorem, Theorem \ref{4.2}. To see this, recall that $S$ is the convex hull of a finite number of points. One first establishes that $P:=S+K$ is a polyhedron and hence has recession cone $K$, i.e. asymptotic cone $A+K$ for a vertex $A$ of $P$. By Theorem \ref{4.2}, $P$ is the convex hull of its vertices (the vertices of $S$) and $A+K$.
 
 A convex cone $D \subset \R^{d+1}$ is said to be finitely generated if there is a $(d+1) \times m$ matrix $B$ such that 
 $D=\{ \, B \lambda, \lambda \in \R^m, \lambda \geq 0\, \}$. By Minkowski's theorem \cite[Theorem 1.13]{paffenholz2010polyhedral}, the polyhedral cone $K = \{ \, x:  x \in \R^{d+1}, A^* x \geq 0\, \}$ is finitely generated. Thus $S+K$ is a polyhedron since $K$ is finitely generated, c.f. for example \cite[Theorem 2.8]{paffenholz2010polyhedral}.  It follows that $P:=S+K$ has recession cone $K$ and hence asymptotic cone $A+K$ for any element $A$ of $S$. 
 
 For Theorem \ref{as-poly}, by Lemma \ref{conv-hull-union}, the closure of the set $M$ is given by $S+K$ and hence has recession cone $K$.

\section{Appendix} \label{appendix}

We gave a geometric proof for Theorem \ref{formula-0} based on Lemma \ref{conv-hull-union}. Here we give an analytical proof based on infimal convolution. The epigraph of the infimal convolution illustrates with an analytical argument Lemma \ref{conv-hull-union}. 

Let $v$ be a continuous convex function on a closed convex set $\widetilde{S}$ with non empty interior. Let $S$ denote the epigraph of $v$. Here $S$ is unbounded unlike in Lemma 3. Let us consider another extension of $v$ to $\R^d$ as an extended value function
$$
v_{\infty}(x) =  \left \{ \begin{array}{ll}
v(x) & \text{ if } x \in \widetilde{S}\\
+\infty & \text{ otherwise }. 
\end{array}
\right.
$$
Recall the function $k_{\Omega^*}$ from \eqref{cone-def-cvx}. The infimal convolution of $v_{\infty}$ and $k_{\Omega^*}$ is a function $v_{\infty} \msquare k_{\Omega^*}: \R^d \to \R^d \cup \{ \, +\infty \, \}$ defined as
$$
v_{\infty} \msquare k_{\Omega^*}(x) = \inf_{y \in \R^d} v_{\infty}(y) + k_{\Omega^*} (x-y). 
$$
Since $ v_{\infty}(y) = +\infty$ for $y \notin \widetilde{S}$, we have 
$$
v_{\infty} \msquare k_{\Omega^*}(x) = \inf_{y \in \widetilde{S} } v(y) + k_{\Omega^*} (x-y). 
$$
Let $\epi u$ denotes the epigraph of a function $u$. Note that $\epi v = \epi v_{\infty}$ as $+\infty \notin \R$. 
For given functions $\phi_1$ and $\phi_2$ from $\R^d$ to $\R^d \cup \{ \, +\infty \, \}$ we have 
$\epi \phi_1 + \epi \phi_2 \subset \epi \phi_1 \msquare \phi_2$. The infimal convolution is said to be exact at $x \in \R^d$ if there exists $y \in \R^d$ such that $\phi_1 \msquare \phi_2(x) = \phi_1(y) + \phi_2(x-y)$. If $\phi_1 \msquare \phi_2$ is exact at all $x \in \R^d$, $\epi \phi_1 + \epi \phi_2 = \epi \phi_1 \msquare \phi_2$, \cite[Lemma 2.8]{ferrera2013introduction}. 

Given $x \in \R^d$, the function $y \mapsto v(y) + k_{\Omega^*}  (x-y)$ is continuous on $\widetilde{S}$ and hence has a minimum on $\widetilde{S}$. Thus $v_{\infty} \msquare k_{\Omega^*} $ is exact at all points $x \in \R^d$ and we conclude that
$$
\epi v_{\infty} \msquare k_{\Omega^*} = \epi v + \epi k_{\Omega^*},
$$
i.e. $M=S+K_{\Omega^*}$ where $M=\epi v_{\infty} \msquare k_{\Omega^*}$. This is essentially the content of Lemma \ref{conv-hull-union}. 

\begin{theorem}
A necessary and sufficient condition for $v_{\infty} \msquare k_{\Omega^*}$ to be a convex extension of $v$ is that
$\partial v((\widetilde{S})^\circ) \subset \tir{\Omega^*}$.
\end{theorem}

\begin{proof}
Recall that a function  $\phi$ defined on $\R^d$ is proper if there exists $x_0 \in \R^d$ such that $\phi(x_0)< +\infty$ and $\phi(x) > - \infty$ for all $x \in \R^d$. As $v_{\infty}$ and $k_{\Omega^*}$ are proper convex functions, $v_{\infty} \msquare k_{\Omega^*}$ is a convex function by \cite[Proposition 2.56]{dhara2011optimality}. 

Recall that $\partial k_{\Omega^*}(\R^d) = \tir{\Omega^*}$. Let us first assume that $v_{\infty} \msquare k_{\Omega^*}=v$ on $\widetilde{S}$. Then for all $x \in (\widetilde{S})^\circ$, $\partial v(x) = \partial v_{\infty} \msquare k_{\Omega^*} (x)$.  This follows from the locality of the subdifferential c.f. \cite[Exercise 1]{Guti'errezExo}. 

By \cite[Proposition 16.48 (i) ]{Bauschke2011}, we have for $x \in (\widetilde{S})^\circ$, $\partial v(x) = \partial v_{\infty} \msquare k_{\Omega^*} (x) = \partial v_{\infty}(y) \cap \partial  k_{\Omega^*} (x-y)$, where $y \in \widetilde{S}$ with $v_{\infty} \msquare k_{\Omega^*} (x) =  v_{\infty}(y) + k_{\Omega^*} (x-y)$. Here $y=x$ and $\partial k_{\Omega^*} (0) =  \tir{\Omega^*}$. We conclude that $\partial v((\widetilde{S})^\circ) \subset  \tir{\Omega^*}$.

Let us now assume that $\partial v((\widetilde{S})^\circ) \subset  \tir{\Omega^*}$. We show that $v_{\infty} \msquare k_{\Omega^*}$ is a convex extension of $v$.
Let $x \in (\widetilde{S})^\circ$. We have $v_{\infty} \msquare k_{\Omega^*}(x)\leq v(x)$. Assume by contradiction that $v_{\infty} \msquare k_{\Omega^*}(x) < v(x)$. This means that we can find $y \in \widetilde{S}$ such that 
\begin{equation} \label{y-contra}
v(y)+k_{\Omega^*}(x-y) < v(x).
\end{equation} 
Let now $p \in \partial v(x)$. We have $p \in \tir{\Omega^*}$. By definition, $v(y) \geq v(x) + p (y-x)$. Thus, by \eqref{y-contra}
\begin{align*}
v(y) > v(y)+k_{\Omega^*}(x-y)  + p \cdot (y-x).
\end{align*} 
It follows that $p \cdot (x-y) > k_{\Omega^*}(x-y) =  \sup_{p \in \tir{\Omega^*}} p \cdot (x-y)$
This contradicts $p \in \tir{\Omega^*}$.  We conclude that $v= v_{\infty} \msquare k_{\Omega^*}$ on $(\widetilde{S})^\circ$. Recall that $v$ is continuous on $\widetilde{S}$. Also, $v_{\infty} \msquare k_{\Omega^*}$ is a proper convex function which is bounded above on $\widetilde{S}$, and hence continuous on $\widetilde{S}$, c.f. \cite[Lemma 2]{awanou2019uweakcvg}. It follows that $v_{\infty} \msquare k_{\Omega^*} = v$ on $\widetilde{S}$. \Qed
\end{proof}

\section*{Declarations}

Funding: The author was partially supported by NSF grant  DMS-1720276. The author would like to thank the Isaac Newton Institute for Mathematical Sciences, Cambridge, for support and hospitality during the programme ''Geometry, compatibility and structure preservation in computational differential equations'' where part of this work was undertaken. Part of this work was supported by EPSRC grant no EP/K032208/1.

The author would like to thank the unknown referees, Carmel Aboua, Julienne Kabre and Nicolae Tarfulea  for a careful reading of the manuscript and for constructive comments which substantially helped improve the quality of the paper. 

\noindent
Conflicts of interest/Competing interests: No 

\noindent
Availability of data and material (data transparency): Available upon reasonable request

\noindent
Code availability (software application or custom code): Available upon reasonable request

\noindent
Authors' contributions: N/A


\end{document}